\documentclass[10pt,a4paper]{amsart}

\usepackage{amsfonts}
\usepackage{amsthm}
\usepackage{amssymb}
\usepackage{latexsym,amsmath,setspace,amscd,verbatim}
\usepackage[pdftex]{graphicx}
\usepackage{graphics}
\usepackage[matrix,arrow]{xy}
\usepackage[active]{srcltx}
\usepackage[colorlinks=true]{hyperref}
\usepackage{mathrsfs}

\title[Global surfaces of section for Reeb flows]{Genus zero global surfaces of section for Reeb flows and a result of Birkhoff}
\author[Hryniewicz]{Umberto L. Hryniewicz}
\author[Salom\~ao]{Pedro A. S. Salom\~ao}
\author[Wysocki]{Krzysztof Wysocki}
\date{\today}

\address{Umberto L. Hryniewicz -- RWTH Aachen, Jakobstrasse 2, Aachen 52064, Germany}
\email{hryniewicz@mathga.rwth-aachen.de}

\address{Pedro A. S. Salom\~ao -- Instituto de Matem\'atica e Estat\'istica, Departamento de Matem\'atica, Universidade de S\~ao Paulo, Rua do Mat\~ao, 1010, Cidade Universit\'aria, S\~ao Paulo SP, Brazil 05508-090}
\email{psalomao@ime.usp.br}

\newcommand{\C}{\mathbb{C}}
\newcommand{\R}{\mathbb{R}}
\newcommand{\Z}{\mathbb{Z}}
\newcommand{\N}{\mathbb{N}}
\newcommand{\D}{\mathbb{D}}

\newcommand{\floor}[1]{\left\lfloor#1\right\rfloor}

\newcommand{\M}{\mathcal M}

\newcommand{\jtil}{\widetilde{J}}
\newcommand{\util}{\widetilde{u}}
\newcommand{\vtil}{\widetilde{v}}
\newcommand{\wtil}{\widetilde{w}}
\newcommand{\CZ}{{\rm CZ}}
\newcommand{\wind}{{\rm wind}}
\newcommand{\J}{\mathcal{J}}
\newcommand{\jbar}{{\bar J}}

\theoremstyle{plain}
\newtheorem{theorem}{\sc Theorem}[section]

\newtheorem{proposition}[theorem]{\sc Proposition}
\newtheorem{lemma}[theorem]{\sc Lemma}
\newtheorem{corollary}[theorem]{\sc Corollary}

\theoremstyle{definition}
\newtheorem{definition}[theorem]{\sc Definition}

\theoremstyle{remark}
\newtheorem{remark}[theorem]{\sc Remark}

\begin{document}

\maketitle

\begin{abstract}
We exhibit sufficient conditions for a finite collection of periodic orbits of a Reeb flow on a closed $3$-manifold to bound a positive global surface of section with genus zero. These conditions turn out to be $C^\infty$-generically necessary. Moreover, they involve linking assumptions on periodic orbits with Conley-Zehnder index ranging in a finite set determined by the ambient contact geometry. As an application, we reprove and generalize a classical result of Birkhoff on the existence of annulus-like global surfaces of section for geodesic flows on positively curved two-spheres.
\end{abstract}

\tableofcontents

\vfill

\newpage

\section{Introduction}

Transverse foliations can be used as tools to study flows in dimension three. 
They allow for a decomposition into discrete two-dimensional systems. 
The best scenario is a transverse foliation whose leaves are global surfaces of section. 
The purpose of this paper is to provide conditions for a finite collection of periodic trajectories of a Reeb flow to bound a positive global surface of section with genus zero. 
Our main result (Theorem~\ref{main1}) explains why it suffices to make linking assumptions on periodic orbits with Conley-Zehnder index on a finite set of values determined by the ambient contact geometry. 
Tools come from pseudo-holomorphic curve theory in symplectic cobordisms developed by Hofer, Wysocki and Zehnder~\cite{93,convex,props1,props2,props3}. 
As an application, we explain why a classical theorem of Birkhoff on the existence of annulus-like global surfaces of section for positively curved geodesic flows on $S^2$ follows as a particular case.

\subsection{Historical remarks}

Let $\phi^t$ be a smooth flow on a smooth closed connected oriented $3$-manifold $M$, generated by a vector field $X$.

\begin{definition}
A {\it global surface of section} for $\phi^t$ is an embedded compact surface $\Sigma\hookrightarrow M$ satisfying:
\begin{itemize}
\item[(i)] $\partial \Sigma$ consists of periodic orbits (if non-empty), $X$ is transverse to $\Sigma \setminus \partial \Sigma$.
\item[(ii)] For every $p\in M$ there exist $t_+>0$, $t_-<0$ such that $\phi^{t_\pm}(p) \in \Sigma$.
\end{itemize}
\end{definition}

\begin{remark}
We always orient a global surface of section by the ambient orientation and the co-orientation induced by the flow. Flow lines intersect the interior of the section with sign $+1$.
\end{remark}

\begin{definition}
A global surface of section is {\it positive} if it orients its boundary along the flow.
\end{definition}

The idea of a global surface of section goes back to Poincar\'e's work on Celestial Mechanics, in particular the {\it planar circular restricted three-body problem} (PCR3BP). 
Poincar\'e studies sub-critical energy levels when almost all mass is concentrated in the primary around which the satellite moves, and finds annulus-like global surfaces of section bounded by the retrograde and the direct orbits. 
These situations arise as perturbations of an integrable system (rotating Kepler problem).

One of the first results for systems which are far from integrable is due to Birkhoff. 
A {\it Birkhoff annulus} over an embedded closed geodesic on a Riemannian two-sphere consists of the unit vectors based at the geodesic pointing towards one of the hemispheres determined by it.

\begin{theorem}[Birkhoff~\cite{birkhoff}]
\label{thm_birkhoff}
Let $c$ be an embedded closed geodesic on a positively curved Riemannian two-sphere. Then a Birkhoff annulus over $c$ is a positive global surface of section for the geodesic flow on the unit sphere bundle.
\end{theorem}

There are no dynamical hypotheses to be checked in Birkhoff's theorem, but one gets strong dynamical conclusions. 
Positivity of the curvature and embeddedness of the closed geodesic can be checked by ``looking at'' the geometric data.

The systems studied by Poincar\'e and Birkhoff are particular examples of Reeb flows of contact forms defining the standard contact structure on $\R P^3$ \cite{AFvKP}. 
Motivated by their results, we want to understand the contact topological and dynamical properties that guarantee the existence of a global surface of section bounded by a prescribed collection of closed Reeb orbits. 
We seek for the counterpart of Theorem~\ref{thm_birkhoff} in the realm of Reeb flows on planar contact manifolds. 
Pseudo-holomorphic curves can be used to implement this program. 
The case of disk-like global surfaces of section for a non-degenerate Reeb flow on the tight three-sphere is handled by the statement below.

\begin{theorem}[\cite{duke}]
\label{thm_tight3sphere}
Consider a non-degenerate Reeb flow on the standard contact three-sphere. A periodic orbit $\gamma$ bounds a disk-like global surface of section if and only if the following conditions are satisfied:
\begin{itemize}
\item[(i)] $\gamma$ is unknotted and has self-linking number $-1$.
\item[(ii)] The Conley-Zehnder index of $\gamma$ is larger than or equal to $3$.
\item[(iii)] Every closed Reeb orbit with Conley-Zehnder index $2$ is linked with $\gamma$.
\end{itemize}
\end{theorem}

Condition (i) in Theorem~\ref{thm_tight3sphere} is purely contact-topological, it means that $\gamma$ binds an open book decomposition with disk-like pages supporting the contact structure. This is how one should rephrase (i) when looking for generalizations of the above statement. 
Condition (ii) is on the linearized dynamics along $\gamma$, it implies that the flow twists fast enough near $\gamma$. 
The linking assumption in (iii) mixes dynamics and topology. 
It distinguishes the closed Reeb orbits which potentially obstruct the existence of a global surface of section bounded by $\gamma$.

In this paper we study Reeb flows on planar contact $3$-manifolds, and prove a version of Theorem~\ref{thm_tight3sphere} for positive global surfaces of section with genus zero, and arbitrary number of boundary components. 
We provide conditions that are sufficient, and $C^\infty$-generically necessary, for a transverse link $L$ formed by closed Reeb orbits to bound such a global section. 
It is necessary that $L$ binds a planar open book decomposition supporting $\xi$ and, $C^\infty$-generically, that the Conley-Zehnder indices of the components of $L$ relative to a page are positive. 
The counterpart of the linking assumption (iii) in Theorem~\ref{thm_tight3sphere} is made only on the closed trajectories in the complement of $L$ that have Conley-Zehnder index in a finite interval of integers determined by the ambient contact geometry. 
This is clearly a necessary condition. 
Sufficiency of these conditions relies on the analysis of certain families of pseudo-holomorphic curves inspired by~\cite{planar_weinstein,convex}.

\subsection{Main results}

A {\it contact form} on an oriented closed $3$-manifold $M$ is a $1$-form $\alpha$ such that $\alpha\wedge d\alpha$ is nowhere vanishing. In particular $\ker d\alpha$ is transverse to the {\it contact structure} $\xi = \ker\alpha$. Contact structures arising in this way are co-orientable since~$\alpha$ orients $TM/\xi$. The orientation induced by $\alpha\wedge d\alpha$ depends only on $\xi$, and $\xi$ is called {\it positive} if $\alpha\wedge d\alpha>0$. The pair $(M,\xi)$ is a {\it contact manifold}. The {\it Reeb vector field}~$X_\alpha$ associated to $\alpha$ is implicitly defined by
\begin{equation}
d\alpha(X_\alpha,\cdot) = 0, \qquad \alpha(X_\alpha)=1.
\end{equation}
Its flow $\phi^t_\alpha$ is called the {\it Reeb flow} of $\alpha$.

An {\it open book decomposition} of $M$ is a pair $\Theta = (\Pi,L)$ consisting of a link $L\subset M$ and a smooth fibration $\Pi:M\setminus L \to \R/\Z$ with a normal form near~$L$. Namely, a neighborhood $N$ of $L$ is homeomorphic to $L\times\D$ in such a way that $L\simeq L\times\{0\}$ and $\Pi$ is represented as $(p,re^{i2\pi x}) \mapsto x$ on $N\setminus L$. The closure of a fiber is called a page, $L$ is called the binding, and $\Theta$ is said to be {\it planar} if the pages have genus zero. It is important to notice that $\Theta$ and the orientation of $M$ naturally orient the pages, and hence also the binding $L$. One says that $\Theta$ {\it supports} $\xi$ if there is a contact form~$\alpha_0$ such that $\xi = \ker \alpha_0$, $\alpha_0>0$ on $L$, and $d\alpha_0>0$ on the interior of the pages. The contact structure $\xi$ is {\it planar} if it is supported by a planar open book.

\begin{remark}
\label{rmk_bind_supp_ob}
It is possible to show that $L$ binds a supporting open book decomposition if, and only if, $L$ bounds a positive global surface of section for the Reeb flow of some defining contact form.
\end{remark}

When the binding $L$ consists of periodic Reeb orbits we write $\mu_{\CZ}^\Theta$ for the Conley-Zehnder index of a binding orbit computed in a symplectic trivialization of~$\xi$ that does not wind with respect to the normal of a page. This means that if we push a connected component of $L$ in the direction of such a trivialization we obtain a loop with algebraic intersection number zero with the pages. We denote by $\rho^\Theta$ the transverse rotation number computed in the same trivialization. See subsection~\ref{sssec_orbits} for precise definitions.

\begin{definition}
\label{defn_positive_Seifert}
Consider a null-homologous link $L$ consisting of periodic orbits of some flow on a $3$-manifold. An oriented Seifert surface for $L$ is said to be positive (with respect to the flow) if it orients every component of $L$ along the flow.
\end{definition}

Our first result, which is to be seen as preliminary work towards our main result (Theorem~\ref{main1}), reads as follows.

\begin{theorem}
\label{main2}
Let the contact form $\alpha$ define a positive contact structure $\xi$ on the closed, connected and oriented $3$-manifold~$M$. Let $L$ be a null-homologous link consisting of periodic Reeb orbits, and let $b \in H_2(M,L)$. 

Consider the following assertions.
\begin{itemize}
\item[(i)] $L$ bounds a positive global surface of section of genus zero for the Reeb flow of~$\alpha$ representing the class~$b$.
\item[(ii)] $L$ binds a planar open book supporting $\xi$ with pages that are global surfaces of section for the Reeb flow of $\alpha$ and represent the class~$b$.
\item[(iii)] $L$ binds a planar open book $\Theta$ supporting $\xi$ with pages that represent the class~$b$, and the following hold:
\begin{itemize}
\item[(a)] $\mu_{\CZ}^{\Theta}(\gamma)>0$ for every component $\gamma\subset L$.
\item[(b)] All periodic orbits in $M\setminus L$ have non-zero intersection number with~$b$.
\end{itemize}
\end{itemize}
Then (iii) $\Rightarrow$ (ii) $\Rightarrow$ (i). Moreover (i) $\Rightarrow$ (iii) holds $C^\infty$-generically, in fact it holds if every periodic orbit $\gamma\subset L$ satisfying $\rho^{\Theta}(\gamma) = 0$ is hyperbolic.
\end{theorem}

\begin{remark}
The $C^\infty$-generic condition used to prove (i) $\Rightarrow$ (iii) holds if all components of $L$ are non-degenerate periodic Reeb orbits.
\end{remark}

\begin{remark}
In~\cite{SFS} the reader finds a version of Theorem~\ref{main2} for general flows in dimension three, which is based on a result of Fried~\cite{fried}. 
The basic mechanism of proof comes from linking assumptions on the invariant measures of the flow, as in the theory of asymptotic cycles initiated by Schwartzman~\cite{schwartzman}. 
Proofs in~\cite{SFS} use techniques from Sullivan's work~\cite{sullivan}. 
The paper~\cite{ghys} by Ghys provides a rich introduction to this topic, with many new ideas.
\end{remark}

Now we work towards our main result, which concerns Reeb flows on contact manifolds $(M,\xi)$ for which $\xi$ is a trivial vector bundle over $M$. This provides the closed Reeb orbits with an absolute Conley-Zehnder index, allowing us to formulate a much weaker linking assumption than the one in Theorem~\ref{main2}.

Given a defining contact form $\alpha$ on $(M,\xi)$, we fix a global symplectic trivialization $\tau_{\rm gl}$ of $(\xi=\ker\alpha,d\alpha)$ and denote Conley-Zehnder indices computed in this trivialization by $\mu_{\rm CZ}^{\tau_{\rm gl}}$. Let $L$ be a null-homologous link consisting of periodic Reeb orbits. Let $\Sigma$ be an oriented Seifert surface for $L$ which is positive as in Definition~\ref{defn_positive_Seifert}. For each connected component $\gamma\subset L$ define $$ m(\gamma,\Sigma)\in\Z $$ to be the winding number, measured in $\tau_{\rm gl}$, of a non-vanishing vector field in $T_\gamma\Sigma \cap \xi$. In particular, we have $$ -{\rm sl}(L,\Sigma)=\sum_\gamma m(\gamma,\Sigma), $$ for the self-linking number of $L$ with respect to $\Sigma$. Consider
\begin{equation}
\label{important_numbers}
C_0 = \sum_{\{\gamma \mid m(\gamma,\Sigma)\geq0\}} (m(\gamma,\Sigma)+1),
\end{equation}
and set
\begin{equation}
\label{def_intervalo}
I(L,\Sigma,\tau_{\rm gl}) := [-2C_0+1,2C_0+1] \cap \Z.
\end{equation}
Our main result, formulated below only for disconnected links of periodic orbits since the connected case is handled by Theorem~\ref{thm_L_connected}, reads as follows.

\begin{theorem}
\label{main1}
Let the contact form $\alpha$ define a positive contact structure $\xi$ on the closed, connected and oriented $3$-manifold~$M$. Suppose that there exists a symplectic trivialization $\tau_{\rm gl}$ of $(\xi,d\alpha)$. Let $L$ be a null-homologous disconnected link formed by periodic Reeb orbits, and let $b\in H_2(M,L)$.

Consider the following assertions.
\begin{itemize}
\item[(i)] $L$ bounds a positive global surface of section of genus zero for the Reeb flow of~$\alpha$ representing the class~$b$.
\item[(ii)] $L$ binds a planar open book supporting $\xi$ with pages that are global surfaces of section for the Reeb flow of $\alpha$ and represent the class~$b$.
\item[(iii)] $L$ binds a planar open book supporting $\xi$ with pages that represent the class~$b$, and the following hold:
\begin{itemize}
\item[(a)] $\mu_{\CZ}^{\Theta}(\gamma)>0$ for every component $\gamma\subset L$.
\item[(b)] All periodic Reeb orbits in $M\setminus L$ satisfying $\mu_{\CZ}^{\tau_{\rm gl}} \in I(L,{\rm page},\tau_{\rm gl})$ have non-zero intersection number with~$b$.
\end{itemize}
\end{itemize}
Then (iii) $\Rightarrow$ (ii) $\Rightarrow$ (i). Moreover (i) $\Rightarrow$ (iii) holds $C^\infty$-generically, in fact it holds if every periodic orbit $\gamma\subset L$ satisfying $\rho^{\Theta}(\gamma)=0$ is hyperbolic.
\end{theorem}

\begin{remark}
In Theorem~\ref{main1} the set of periodic orbits that have to ``link'' with $L$ for the implication (iii) $\Rightarrow$ (ii) to hold is drastically smaller than the one in Theorem~\ref{main2}. Linking with all periodic orbits can not be directly obtained from the above hypotheses.
In the proof of (iii) $\Rightarrow$ (ii) one only needs to check linking with the periodic orbits with Conley-Zehnder index in $I(L,{\rm page},\tau_{\rm gl})$ and period less than some positive constant depending on $\alpha$ and the open book. Generically this is a finite set of periodic orbits.
\end{remark}

\begin{remark}
The class of closed contact $3$-manifolds which are planar and whose contact structure admit a global trivialization contains the tight $3$-sphere, the standard lens spaces $L(p,p-1)$, and is closed under contact connected sums. 
As a consequence, one can build a large family of examples that can carry contact forms satisfying the assumptions of Theorem~\ref{main1}, some of which are particularly relevant in Celestial Mechanics like, for instance, $(\R P^3 \# \R P^3,\xi_{\rm std} \# \xi_{\rm std})$.
Persistence of planarity under contact connected sums can be proved with van Koert's book sums~\cite{vK_thesis}, see~\cite[subsection~2.2]{vK_no_GSS} for more details. Persistence of triviality of the contact structure as a symplectic vector bundle can be checked by hand. 
\end{remark}

We compute $I(L,{\rm page},\tau_{\rm gl})$ in two examples. Consider $\R^4$ with coordinates $(q_1,p_1,q_2,p_2)$ and its standard symplectic form $\omega_0 = \sum_j dq_j \wedge dp_j$. The primitive $\frac{1}{2} \sum_j q_jdp_j-p_jdq_j$ induces a contact form $\lambda_0$ on $S^3 = \{ q_1^2+p_1^2+q_2^2+p_2^2 = 1 \}$, and $\xi_{\rm std} = \ker \lambda_0$ is called the {\it standard} contact structure on~$S^3$. The Reeb flow of $\lambda_0$ is $\pi$-periodic, its orbits are the fibers of the Hopf fibration $S^3 \to S^2$.  Since $\lambda_0$ is antipodal symmetric, it descends to a contact form on $\R P^3 = S^3/\{\pm 1\}$ which we again denote by $\lambda_0$. It defines a contact structure which we again denote by $\xi_{\rm std}$ and call standard. The Hopf fibration on $S^3$ is antipodal symmetric, hence it induces a fibration of $\R P^3$ by circles which we again call Hopf fibers. Any link transversely isotopic to a pair of Hopf fibers will be called a Hopf link, both in $(S^3,\xi_{\rm std})$ or in $(\R P^3,\xi_{\rm std})$.

Hopf links in $(S^3,\xi_{\rm std})$ or in $(\R P^3,\xi_{\rm std})$ bind supporting open book decompositions with annulus-like pages. We can use an ambient contact isotopy to put a Hopf link in a normal form, for instance it is interesting to see $\R P^3$ as the unit tangent bundle of the round metric on $S^2$ and realize a Hopf link $L$ as velocity vectors $\gamma_1 = \dot c$, $\gamma_2 = -\dot c$ of a unit speed great circle $c$ in $S^2$ traversed in both directions. In this form a Birkhoff annulus is a page of the aforementioned open book decomposition~$\Theta$. Let $\Sigma$ be a page of $\Theta$ and let $\tau_{\rm gl}$ be a trivialization of $\xi_{\rm std}$. The winding number measured in $\tau_{\rm gl}$ of a non-vanishing vector field in $T_{\gamma_i}\Sigma \cap \xi_{\rm std}$ vanishes for $i=1,2$ and thus $m(\gamma_i,\Sigma)=0$. Using \eqref{important_numbers} and \eqref{def_intervalo} we compute $C_0=2$ and $I(L,\Sigma,\tau_{\rm gl})=\{-3,\ldots, 5\}$. This interval is not sharp, and we expect it to be further shortened  to $\{0,1,2\}$ in the non-degenerate case. The same computation goes through for Hopf links in $(S^3,\xi_{\rm std})$.

A version of Theorem~\ref{main1} was already available when $L$ is connected, generalizing Theorem~\ref{thm_tight3sphere} to possibly degenerate contact forms. When $\xi$ is supported by a planar open book decomposition with connected binding then $(M,\xi)=(S^3,\xi_{\rm std})$ and the binding is unknotted with self-linking number~$-1$. There exists a global trivialization $\tau_{\rm gl}$ of $\xi_{\rm std}$ which is symplectic with respect to the standard symplectic structure of $\R^4$, and is unique up to homotopy. One can then use $\tau_{\rm gl}$ to define Conley-Zehnder indices $\mu_{\CZ}^{\tau_{\rm gl}}$ of periodic Reeb orbits. Note also that global surfaces of section for Reeb flows with connected boundary must be positive (Stokes' theorem), and if they have genus zero (disks) then their boundaries are unknotted with self-linking number $-1$. Finally, on the three-sphere any Seifert surface for a periodic Reeb orbit singles out a homotopy class of symplectic trivializations of the contact structure along the orbit in an obvious manner. This homotopy class does not depend on the Seifert surface. We write $\mu_{\CZ}^{\rm Seifert}$ for the associated Conley-Zehnder index.

\begin{theorem}
\label{thm_L_connected}
Let $\gamma$ be a periodic orbit of a Reeb flow on the tight three-sphere. Consider the following assertions.
\begin{itemize}
\item[(i)] $\gamma$ bounds a disk-like global surface of section for the Reeb flow.
\item[(ii)] $\gamma$ binds a planar supporting open book with pages that are disk-like global surfaces of section for the Reeb flow.
\item[(iii)] $\gamma$ is unknotted, has self-linking number $-1$, and the following hold:
\begin{itemize}
\item[(a)] $\mu_{\CZ}^{\rm Seifert}(\gamma) = \mu_{\CZ}^{\tau_{\rm gl}}(\gamma) - 2 > 0$.
\item[(b)] All periodic Reeb orbits $\gamma'$ in the complement of $\gamma$ satisfying either $\mu_{\CZ}^{\tau_{\rm gl}}(\gamma')=2$, or $\mu_{\CZ}^{\tau_{\rm gl}}(\gamma')=1$ and $\gamma'$ is degenerate, are linked with $\gamma$.
\end{itemize}
\end{itemize}
Then (iii) $\Rightarrow$ (ii) $\Rightarrow$ (i). Moreover (i) $\Rightarrow$ (iii) holds $C^\infty$-generically, in fact it holds if $\gamma$ is hyperbolic when $\rho^{\rm \tau_{\rm gl}}(\gamma)=1$.
\end{theorem}

\begin{proof}
When all periodic Reeb orbits are non-degenerate the equivalence of (i), (ii) and (iii) is given by Theorem~\ref{thm_tight3sphere}. If non-degeneracy is dropped then the implications (iii) $\Rightarrow$ (ii) $\Rightarrow$ (i) are contained in~\cite[Theorem~1.7]{elliptic}.


The proof that (i) $\Rightarrow$ (iii) holds under the above stated $C^\infty$-generic condition goes as follows. Assume that $\gamma$ bounds a disk-like global surface of section. Since its self-linking number is necessarily equal to $-1$, we compute $\mu_{\CZ}^{\rm \tau_{\rm gl}}(\gamma)-2 = \mu_{\CZ}^{\rm Seifert}(\gamma)$. Condition (iii-b) is clearly satisfied, and we need to show that (iii-a) also holds. Suppose, by contradiction, that $\mu_{\CZ}^{\rm \tau_{\rm gl}}(\gamma)\leq2$. Then the rotation number of the transverse linearized dynamics along $\gamma$ is non-positive when computed in a Seifert framing. It must vanish since otherwise nearby trajectories would intersect the global surface of section negatively. Hence $\mu_{\CZ}^{\rm \tau_{\rm gl}}(\gamma)\in\{1,2\}$ and $\gamma$ is degenerate if $\mu_{\CZ}^{\rm \tau_{\rm gl}}(\gamma)=1$. By assumption $\gamma$ is hyperbolic, hence non-degenerate, and we get $\mu_{\CZ}^{\rm \tau_{\rm gl}}(\gamma)=2$. Its stable manifold contains infinitely long trajectories that never touch the global surface of section, and this is a contradiction.
\end{proof}

\subsection{Applications}

We derive two applications. The first is Birkhoff's theorem, thus showing that Theorem~\ref{main2} is a generalization of Theorem~\ref{thm_birkhoff}.

\begin{proof}[Proof of Theorem~\ref{thm_birkhoff}]
Suppose that the Birkhoff annulus $A$ associated to an embedded closed geodesic $\gamma$ on a positively curved two-sphere is not a global surface of section. The contact structure on the unit sphere bundle induced by the Hilbert form associated to the metric is diffeomorphic to the standard contact structure on $\R P^3$, and $A$ is one page of a supporting open book decomposition. By positivity of the curvature, the rotation numbers of the binding orbits $\dot\gamma$ and $-\dot\gamma$ with respect to the pages are strictly positive. In view of implications (iii) $\Rightarrow$ (ii) $\Rightarrow$ (i) in Theorem~\ref{main2}, we conclude that there exists a closed geodesic $\beta$, say parametrized by unit speed, such that $\dot\beta$ has zero algebraic intersection number with~$A$ in the unit sphere bundle. Since the geodesic vector field is transverse to the interior of $A$, we conclude that $\beta$ does not touch $\gamma$ on the base, hence it is contained in the interior of one of the hemispheres determined by~$\gamma$. The complement of $\gamma \cup \beta$ has a distinguished connected component $\Omega$ containing boundary points in both $\gamma$ and $\beta$. Topologically, the closure of $\Omega$ is an annulus, its boundary has two components: $\gamma$ is one of them, the other consists of finitely many geodesic arcs in $\beta$. The sum of the external angles (there may be none in case $\beta$ is embedded) in the boundary of $\Omega$ is non-negative. Since the Euler characteristic of the annulus vanishes, we conclude from the Gauss-Bonnet theorem that the integral of the Gaussian curvature over $\Omega$ is non-positive, contradicting the assumption that the curvature is everywhere strictly positive.
\end{proof}

Our second application reads as follows.

\begin{proposition}
\label{thm_annulus_dyn_convex}
Let $\alpha$ be a contact form on $(S^3,\xi_{\rm std})$. 
Consider a pair of periodic Reeb orbits $\gamma_0,\gamma_1$ forming a Hopf link, with Conley-Zehnder indices $\geq 1$ in a global frame. 
Assume that $\gamma_0$ bounds a disk-like global surface of section, and that periodic orbits $\gamma$ in the complement of $\gamma_1$ with Conley-Zehnder index in $\{-3,\dots,5\}$ satisfy ${\rm link}(\gamma,\gamma_1)\geq0$. 
Then $\gamma_0\cup\gamma_1$ bounds an annulus-like global surface of section.
\end{proposition}

\begin{proof}
The link $\gamma_0 \cup \gamma_1$ binds a supporting open book decomposition $\Theta$ with annulus-like pages. 
Let $A$ denote a page of $\Theta$ oriented in such a way that the boundary orientation on $\gamma_0\cup\gamma_1$ coincides with the flow orientation.
Let $D_0$ be a disk-like global surface of section bounded by $\gamma_0$, and let $D_1$ be any embedded disk spanned by $\gamma_1$. 
Both $D_j$ are oriented so that the boundary orientation on $\gamma_j$ coincides with the flow orientation.
Then $S = A \# \overline{D_0} \# \overline{D_1}$ defines a $2$-cycle.
If $\gamma$ is any closed Reeb orbit in $S^3 \setminus (\gamma_0\cup\gamma_1)$ with Conley-Zehnder index in $\{-3,\dots,5\}$ then 
$$ 
0 = {\rm int}(\gamma,S) = {\rm int}(\gamma,A) - \sum_{j=0,1} {\rm int}(\gamma,D_j) \leq {\rm int}(\gamma,A) - 1 \Rightarrow {\rm int}(\gamma,A) \geq 1. 
$$
Another feature of $\Theta$ is that a section of $TA|_{\gamma_j} \cap \xi_{\rm std}$ does not wind with respect to a global trivialization of $\xi_{\rm std}$. 
Hence, since $\mu_{\CZ}^\Theta(\gamma_j) \geq 1$ we are done checking (iii) in Theorem~\ref{main1}. 
A direct application of this theorem completes the proof.
\end{proof}

Proposition~\ref{thm_annulus_dyn_convex} can be applied to the PCR3BP. A satellite is subjected to the gravitational fields of two primaries which move on circular trajectories about their center of mass. The satellite moves on the same plane of the primaries, in a rotating system which fixes the primaries, and is determined by the Hamiltonian

\begin{equation}\label{eq_cpr3bp}
H_\mu(q,p) = \frac{1}{2}|p|^2 + \left<q-\mu,ip\right> - \frac{\mu}{|q-1|} - \frac{1-\mu}{|q|}.
\end{equation}Here  $q\in \C \setminus\{0,1\}$ is the position of the satellite and $p\in \C$ is the momentum. The primaries have masses $\mu>0$ and $1-\mu>0$ and rest at $1\in \C$ and $0\in \C$, respectively. For energies $-c$ below the first Lagrange value $l_1(\mu)$ the energy surface $H^{-1}_\mu(-c)$ contains a component $\mathcal{B}_{\mu,c}$ which projects onto a punctured topological disk about the primary at $0\in \C$. Levi-Civita coordinates regularize collisions with the primary: $$ q=2v^2 \ \ \  \mbox{ and }  \ \ \ p=-\frac{u}{\bar v}. $$ The regularized Hamiltonian reads as
\begin{equation}\label{Kmc}
\begin{aligned}
K_{\mu,c}(v,u) & := |v|^2(H_\mu(p,q)+c)\\ & = \frac{1}{2}|u|^2+2|v|^2\left<u,iv\right>-\mu \Im (uv) -\frac{1-\mu}{2}-\mu \frac{|v|^2}{|2v^2-1|}+c|v|^2.
\end{aligned}
\end{equation}
The Jacobi constant $c$ becomes a parameter and the regularized dynamics on $\mathcal{B}_{\mu,c}$ is the dynamics on a $\Z_2$-symmetric sphere-like component $\mathcal{S}_{\mu,c}\subset K_{\mu,c}^{-1}(0)$, where $\Z_2$-symmetry refers to antipodal symmetry. The Hamiltonian flow on $\mathcal{S}_{\mu,c}/\Z_2$ is orbit equivalent to the Reeb flow of a contact form $\lambda_{\mu,c}$ on $(\R P^3,\xi_{\rm std})$.

\begin{definition}
A $T$-periodic trajectory $(q(t),p(t)), t\in \R / T\Z,$ in $\mathcal{B}_{\mu,c}$ is called retrograde if the curve $t\mapsto q(t)\in \C \setminus \{0\}$ is an embedded closed curve that winds around the origin precisely once in the clockwise direction.
\end{definition}

Retrograde orbits are shown to exist by Birkhoff's shooting method for every $0<\mu<1$ and every $-c$ below 
$l_1(\mu)$. 
A nice discussion about Birkhoff's shooting method can be found in~\cite[subsection~8.3.2]{FvK}. 
As a closed Reeb orbit $\gamma_{\mu,c}^r$ of $\lambda_{\mu,c}$, the retrograde orbit is a Hopf fiber, i.e., a $2$-unknot with rational self-linking number $-1/2$. 
If $\lambda_{\mu,c}$ is dynamically convex then $\gamma_{\mu,c}^r$ is the boundary of a rational disk-like global surface of section and a fixed point of the first return map gives a closed Reeb orbit $\gamma_{\mu,c}^d$ which is called direct and shares the same properties of $\gamma_{\mu,c}^r$, see \cite{elliptic}. 
A result of Frauenfelder and Kang~\cite{FK} applies to similar regimes of the PCR3BP and gives disk-like global surfaces of section on $\mathcal{S}_{\mu,c}$ that respect additional symmetries of the problem. 
The link $L=\gamma_{\mu,c}^r \cup \gamma_{\mu,c}^d$ is a Hopf link, hence it binds an open book decomposition supporting $\xi_{\rm std}$.

\begin{theorem}\label{thm_cpr3bp}
If the Reeb flow of $\lambda_{\mu,c}$ on $(\R P^3,\xi_{\rm std})$ is dynamically convex then the link $\gamma_{\mu,c}^r \cup \gamma_{\mu,c}^d$ formed by the retrograde and  the direct orbits is the binding of an open book decomposition whose pages are annulus-like global surfaces of section.
\end{theorem}

\begin{proof}
The link $h = \gamma_{\mu,c}^r\cup \gamma_{\mu,c}^d$ lifts to a Hopf link $\tilde h \subset (S^3,\xi_{\rm std})$. There exists a supporting open book decomposition $\tilde\Theta$ with annulus-like pages and binding $\tilde h$. We can arrange $\tilde\Theta$ so that it is $\Z_2$-symmetric: the binding already consists of antipodal symmetric knots, and the pages can be arranged so that the image of a page by the antipodal map is the same page. By Theorem~\ref{thm_L_connected} and by the assumed dynamical convexity, both components of $\tilde h$ bound disk-like global surfaces of section. It follows from Proposition~\ref{thm_annulus_dyn_convex} that $\tilde h$ bounds an annulus-like global surface of section $\tilde A$ representing the same class in $H_2(S^3,\tilde h)$ as the pages of $\tilde\Theta$. Now denote by $\Theta = \tilde\Theta/\{\pm1\}$ the induced open book decomposition in $\R P^3$. Up to taking a double cover, any periodic trajectory $\gamma$ in $\R P^3 \setminus h$ lifts to a periodic trajectory in $S^3\setminus \tilde h$, which then must have positive intersection number with the pages of $\tilde\Theta$. Hence also $\gamma$ has to have positive intersection number with the pages of $\Theta$. A direct application of Theorem~\ref{main1} concludes the argument.
\end{proof}

We observe that annulus-like global surfaces of section bounded by the retrograde and the direct orbits are proven in the following situations:
\begin{itemize}
\item[(i)] For any fixed $c>\frac{3}{2}$ and any $\mu$ sufficiently small. This goes back to celebrated work of Poincar\'e \cite{Po}, and is proved by Birkhoff in~\cite[sections~9~and~10]{birkhoff_3_bodies}. Here the system is a perturbation of the rotating Kepler problem where annulus-like global surfaces of section are easily constructed.
 \item[(ii)] For every $c$ sufficiently large and any $0<\mu <1$. This case is proved by Conley~\cite{conley}. The flow of $\lambda_{\mu,c}$ is a perturbation of the degenerate flow associated with the standard contact form $\lambda_0$ on $(\R P^3,\xi_{\rm std})$. A finer account of higher order terms is necessary in order to check that the link $\gamma_{\mu,c}^r \cup \gamma_{\mu,c}^d$ bounds an annulus-like global surface of section satisfying a boundary twist condition.
\end{itemize}
Theorem~\ref{thm_cpr3bp} gives annulus-like global surfaces of section bounded by $\gamma_{\mu,c}^r\cup \gamma_{\mu,c}^d$ in the following new case:
\begin{itemize}
\item[(iii)] For any fixed $c>\frac{3}{2}$ and any $1-\mu$ sufficiently small. 
\end{itemize}

\medskip 

Genus zero global surfaces of section for Reeb flows were used in~\cite{HLS} to characterize standard lens spaces from a dynamical viewpoint. 
We expect that every dynamically convex contact form on a lens space of order $p > 1$ carries a periodic orbit that spans a rational disk-like global surface of section, and whose $p$-th iterate has Conley-Zehnder index equal to $3$ in a trivialization that extends to a capping disk; in particular such an orbit must be elliptic. 
In~\cite{elliptic} this was confirmed when $p=2$.
Schneider~\cite{schneider} handles the standard contact structure on $L(p,1)$, for all $p > 1$.

Genus zero global surfaces of section were also used by Cristofaro-Gardiner, Hutchings and Pomerleano in~\cite{CGHP} to study non-degenerate contact forms such that the first Chern class of the contact structure is torsion.
Under these assumptions, they show that if there are finitely many periodic orbits then the Reeb flow admits a (rational) genus zero global surface of section, thus forcing the number of periodic orbits to be equal to two. 
In particular, there are either two or infinitely many periodic orbits.

Transverse foliations that are more general than open book decompositions can be used to study Reeb flows in dimension~three. 
This line of research was initiated by Hofer, Wysocki and Zehnder in their seminal paper~\cite{fols} where finite-energy foliations (FEF) in the symplectization of the standard contact $3$-sphere were constructed. 
The leaves of a FEF are finite-energy curves, and the whole foliation is invariant under the natural $\R$-action on the symplectization. 
Finitely many leaves are $\R$-invariant trivial cylinders and project onto periodic orbits, the so-called binding orbits. 
The remaining leaves project to form a smooth foliation of the complement of the binding orbits in $S^3$, whose leaves are transverse to the Reeb flow. 
The construction of a FEF can be done for every non-degenerate contact form defining the standard contact structure on $S^3$. 
Stability of FEFs under contact connected sums is proved in~\cite{FS}. 
In~\cite{PS1,PS2,lemos} one finds existence results for FEFs which apply to explicit systems coming from Celestial Mechanics and to many other explicit, real analytic and possibly degenerate, contact forms.

In~\cite[section~6]{icm2018} a definition of more general transverse foliations was given. 
Recently Colin, Dehornoy and Rechtman~\cite{CDR} introduced the so-called broken book decompositions, and showed that a non-degenerate Reeb flow on any closed $3$-manifold is transverse to the leaves of some broken book. As an application, they show that non-degenerate Reeb flows on closed $3$-manifolds carry two or infinitely many periodic orbits. 

\subsection{Outline of proofs}

We end this introduction with a rough outline of the structure of the arguments. 
In both Theorem~\ref{main2} and Theorem~\ref{main1} the implications (ii) $\Rightarrow$ (i), and (i) $\Rightarrow$ (iii) under the stated $C^\infty$-generic condition, have rather simple proofs when compared to the proof of (iii) $\Rightarrow$ (ii). 
Most of our work is concerned with the latter implication, which is done in section~\ref{sec_main2_proof} for Theorem~\ref{main2}, and some additional arguments in section~\ref{sec_main1_proof} for Theorem~\ref{main1}.

Let $\alpha$, $L$, $\Theta$ and $b$ be as in the statements.
In particular,~$L$ is the binding of $\Theta$ and consists of periodic Reeb orbits of $\alpha$, the open book~$\Theta$ is planar and supports $\xi=\ker\alpha$, and~$b$ is the class of the pages of $\Theta$. 
Denote by $n$ the number of connected components of $L$. There is no loss of generality to assume that $n\geq2$ since the case $n=1$ is contained in Theorem~\ref{thm_L_connected}.

By a result of Wendl~\cite{wendl} there exists a non-degenerate contact form $\alpha_+ = f\alpha$, where $f\in C^\infty(M,(1,+\infty))$, such that~$L$ consists of periodic Reeb orbits of $\alpha_+$ with $\mu_{\CZ}^{\Theta} = 1$, and a $d\alpha_+$-compatible complex structure $J_+:\xi\to\xi$ such that there is a foliation of $M\setminus L$ by projections of finite-energy curves in $\R\times M$ with no negative punctures, asymptotic to $L$ at its (positive) punctures. 
These curves are pseudo-holomorphic with respect to a suitably defined $\R$-invariant almost complex structure $\jtil_+$ on $\R\times M$ induced by $(\alpha_+,J_+)$.
These projected curves form the pages of an open book decomposition, represent the class $b$ and are global surfaces of section for the Reeb flow of the special contact form $\alpha_+$.

The next step is to consider an almost complex structure $\jbar$ on $\R\times M$ interpolating from $\jtil_+$ near the positive end, to a suitable $\R$-invariant almost complex structure $\jtil$ induced by $(\alpha,J)$ near the negative end, where $J:\xi \to \xi$ is some compatible complex structure. 
We can require that $\R\times L$ is a $\jbar$-complex surface since $L$ is a common set of periodic orbits for both $\alpha_+$ and $\alpha$. See~\ref{sssec_finite_energy_curves} and~\ref{sssec_gen_fin_energy_curves} for detailed definitions. 
It should be mentioned that the components of~$L$ satisfy $\mu_{\CZ}^{\Theta} = 1$ as periodic Reeb orbits of $\alpha_+$, and $\mu_{\CZ}^{\Theta} \geq 1$ as periodic Reeb orbits of $\alpha$. 
Hence the Fredholm index of the holomorphic cylinders given by the components of $\R\times L$ can be arbitrarily negative. 

The curves in the symplectization of $\alpha_+$ provided by~\cite{wendl} can be translated up and be seen as $\jbar$-curves. 
Since $\Theta$ is planar, their Fredholm index is equal to $2$.
The bulk of the proof of Theorem~\ref{main2} is to show we can deform these $\jbar$-curves indefinitely down in $\R\times M$, and obtain as SFT-limit a holomorphic building with two levels as in Figure~\ref{figure}. 
Here a non-degeneracy assumption on $\alpha$ is used.
The upper level consists precisely of the (possibly non-regular) $\jbar$-holomorphic cylinders in $\R\times L$. 
The bottom level consists of a genus zero connected curve $C$ with no negative punctures, and $n$ positive punctures where it is asymptotic to~$L$. 
Moreover, $C$ does not intersect $\R\times L$, and has an embedded projection to $M \setminus L$ transverse to the Reeb vector field of $\alpha$. 
Key to our arguments is the intersection theory from~\cite{siefring}, the properties of~$\alpha_+$, and the assumptions on~$\alpha$. 
Finally, it is absolutely crucial that the approach of $C$ to its asymptotic limits in $L$ is the `right one', more precisely, it can be shown $C$ is an element of a certain moduli space with exponential weights with respect to which the corresponding weighted Conley-Zehnder indices of the components of $L$ are again equal to $1$. 
Hence, the weighted Fredholm index is $2$, and automatic transversality holds for this weighted Fredholm problem.
The linking assumptions provide compactness of the connected component of the moduli space containing $C$.
The arguments so far are contained in subsections~\ref{ssec_Seifert_and_curves} and~\ref{ssec_existence_compactness_curves}.
Hence $M\setminus L$ can be again foliated by projections of finite-energy curves as an open book, this time by surfaces transverse to the Reeb flow of $\alpha$.
The assumption $\mu_{\CZ}^\Theta \geq 1$ made on the components of $L$ as Reeb orbits of $\alpha$ now plays a role to show that the pages are global surfaces of section, see subsection~\ref{ssec_approximating_sequences}.

The argument for (iii) $\Rightarrow$ (ii) in Theorem~\ref{main1} in the non-degenerate case is analogous. 
It depends on some estimates, based on the algebraic invariants from~\cite{props2}, to ensure that only linking assumptions with periodic orbits in $M\setminus L$ with Conley-Zehnder index in the interval $I(L,{\rm page},\tau_{\rm gl})$ are needed, see section~\ref{sec_main1_proof}.

The passage to the degenerate case is technical and delicate, but also somewhat standard in the literature. 
An important ingredient is~\cite[Proposition~4.15]{elliptic}, stated here as Proposition~\ref{prop_unif_asymptotic_analysis}. 
It summarizes some of the analysis which was originally done in~\cite[section~8]{convex}, and guarantees that if a sequence of curves with common asymptotic limits has exponential decay at the punctures uniformly bounded away from zero, then the limiting curves have the same bounds on the decay at the corresponding punctures. Details are found in subsection~\ref{ssec_passing_to_deg_case}.

\begin{figure}
\includegraphics[width=120\unitlength]{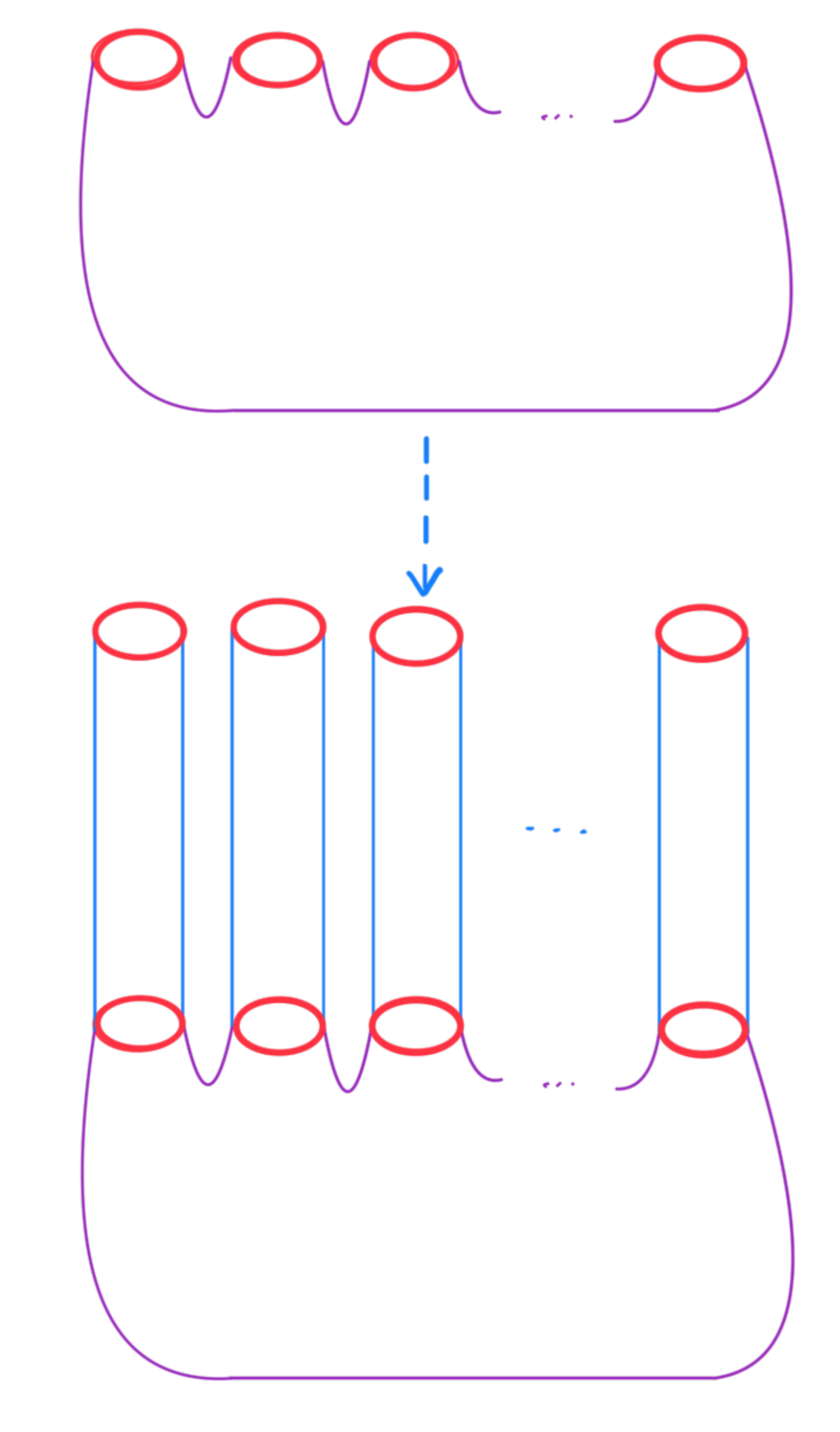}\label{figure}
\caption{Pushing curves down a cobordism, and controlling the SFT-limit using linking assumptions and intersection theory: the curve in the bottom level does not intersect $\R\times L$, has an embedded projection to $M$, approaches its asymptotic limits according to special eigenvectors of the asymptotic operators, and belongs to a moduli space with suitable exponential weights.}
\end{figure}

\medskip

\noindent \textbf{Acknowledgements.} Our friend and teacher Kris Wysocki passed away untimely in~2016. 
We started our collaboration in the fall of 2011, at the Institute for Advanced Study (IAS/Princeton) during the thematic year on Symplectic Dynamics. 
We thank the IAS for the opportunity to participate in such a wonderful and fruitful program.
P. S. was partially supported by the FAPESP grant 2017/26620-6, the CNPq grant 306106/2016-7 and the Humboldt Foundation. 
P. S. is grateful to Ruhr-Universit\"at Bochum for the hospitality. 
P. S. acknowledges the support of NYU-ECNU Institute of Mathematical Sciences at NYU Shanghai. 
U. H. thanks the IAS/Princeton for its support during the academic year 2018-19 through a von Neumann Fellowship.

\section{Preliminaries}

The purpose of this section is to review the basic facts about pseudo-holomorphic curves and periodic Reeb orbits needed in the proofs, and establish notation. Fix a smooth, closed, connected, oriented $3$-manifold $M$ with a co-orientable and positive contact structure~$\xi$. Throughout this section $\alpha$, $\alpha_+$, $\alpha_-$ etc denote contact forms that define $\xi$ and are all positive multiples of each other.

\subsection{Periodic orbits and Conley-Zehnder indices}\label{sssec_orbits}

A periodic $\alpha$-Reeb orbit is an equivalence class of pairs $\gamma=(x,T)$ where $x:\R\to M$ is a trajectory of the flow of the Reeb vector field $X_\alpha$, and $T>0$ is a period of $x$. Here we declare two pairs $\gamma=(x,T)$ and $\gamma'=(x',T')$ to be equivalent if $x(\R)=x'(\R)$ and $T=T'$. The set of equivalence classes is denoted $\mathcal{P}(\alpha)$.

\begin{remark}\label{rem_marked_point_on_orbits}
A special point (marker) is chosen on $x(\R)$ for every periodic trajectory $x:\R\to M$ of $X_\alpha$. It will be implicit in the notation $\gamma=(x,T) \in \mathcal{P}(\alpha)$ that $x(0)$ is the marker. 
\end{remark}

Given $n\in\N$ the $n$-th iterate of $\gamma$ is $\gamma^n=(x,nT)$.  We call $\gamma=(x,T)$ simply covered, or prime, if $T$ is the minimal period of $x$. When $\gamma$ is not prime then it is called multiply covered. In any case, $\gamma=\gamma_0^n$ for unique $n\in\N$ and prime orbit~$\gamma_0$, where $n$ is called the covering multiplicity of $\gamma$, and $\gamma_0$ is called the prime orbit underlying~$\gamma$.

Let $\gamma = (x,T) \in \mathcal{P}(\alpha)$. We denote the vector bundle $x(T\cdot)^*\xi \to \R/\Z$ by $\xi_\gamma$. It becomes a symplectic vector bundle with $d\alpha$. Since the symplectic group is connected, $\xi_\gamma$ is trivial as a symplectic vector bundle. Given two $d\alpha$-symplectic trivializations $\sigma_1$ and $\sigma_2$, take a section which is represented as a non-zero constant with respect to $\sigma_2$, represent it with $\sigma_1$ to obtain a non-vanishing loop $t\mapsto v(t)$ in~$\C$, and set $\wind(\sigma_1,\sigma_2)\in\Z$ to be equal to the winding number of $v(t)$. It should be thought of as the winding of $\sigma_2$ with respect to $\sigma_1$. Throughout this paper $\R/\Z$ is given its canonical orientation.

A complex structure $J:\xi\to\xi$ is said to be $d\alpha$-compatible if $d\alpha(\cdot,J\cdot)$ is an inner-product on every fiber. The set of $d\alpha$-compatible complex structures on~$\xi$ is contractible with the $C^\infty$-topology. Choose a symmetric connection $\nabla$ on $TM$ and consider the first order differential operator
\[
\eta \mapsto -J(\nabla_t\eta - T\nabla_\eta X_\alpha)
\]
on sections of $\xi_\gamma$. It is called the asymptotic operator. Here $\nabla_t$ denotes covariant derivative along the loop $t\mapsto x(Tt)$. This operator does not depend on the choice of $\nabla$. A $(d\alpha,J)$-unitary frame identifies $\xi_\gamma \simeq \R/\Z \times \C$ where $\C$ is endowed with its standard symplectic and complex structures. The asymptotic operator gets represented as $-i\partial_t-S(t)$ for some smooth loop $S:\R/\Z \to \mathcal{L}_\R(\C)$ of symmetric matrices. Its spectrum consists of eigenvalues which are all real, form a discrete subset of $\R$ and accumulate only at $\pm\infty$. Eigenvectors are smooth and non-vanishing. Thus, once a $d\alpha$-symplectic frame $\sigma$ is chosen, an eigenvector for the eigenvalue $\lambda$ gets represented as a non-vanishing vector $t\in\R/\Z \mapsto e(t)\in \C$ and, as such, has a winding number $\wind(\lambda,\sigma) \in\Z$. As the notation suggests, it does not depend on the choice of eigenvector. If $\lambda_1\leq\lambda_2$ are eigenvalues then $\wind(\lambda_1,\sigma) \leq \wind(\lambda_2,\sigma)$. Moreover, every integer is equal to $\wind(\lambda,\sigma)$ for some eigenvalue $\lambda$, and for every $k\in\Z$ the algebraic count of eigenvalues satisfying $\wind(\lambda,\sigma)=k$ is equal to two. See~\cite{props2} for more details.

Fix $\delta\in\R$ arbitrarily. Let $\lambda_-<\delta$ be the largest eigenvalue of the asymptotic operator which is strictly less than $\delta$, and let $\lambda_+\geq\delta$ be the smallest eigenvalue which is larger than or equal to~$\delta$. Choose a symplectic trivialization $\sigma$ of $\xi_\gamma$. The $\delta$-weighted Conley-Zehnder index of $\gamma$ with respect to $\sigma$ is defined as
\[
\mu_{\CZ}^{\sigma,\delta}(\gamma) = 2\wind(\lambda_-,\sigma) + p
\]
where $p=1$ if $\wind(\lambda_-,\sigma) < \wind(\lambda_+,\sigma)$, $p=0$ if $\wind(\lambda_-,\sigma)=\wind(\lambda_+,\sigma)$. It turns out to be independent of~$J$. Moreover, we have
\[
\mu_{\CZ}^{\sigma',\delta}(\gamma) = 2\wind(\sigma',\sigma) + \mu_{\CZ}^{\sigma,\delta}(\gamma).
\]
The $0$-weighted Conley-Zehnder index will be referred to as the Conley-Zehnder index and denoted by $\mu_{\CZ}^\sigma(\gamma) \in \Z$.

The linear map $d\phi^T|_{x(0)}:\xi|_{x(0)} \to \xi|_{x(0)}$ is $d\lambda$-symplectic. The orbit $\gamma=(x,T)$ is called non-degenerate if $1$ is not an eigenvalue of this map. It defines an orientation preserving diffeomorphism of the circle of rays $(\xi|_{x(0)}\setminus\{0\})/\R_+$. The frame $\sigma$ singles out an isotopy from the identity to this diffeomorphism. Hence, its rotation number $\rho^\sigma(\gamma) \in \R$ is well-defined. It satisfies
\begin{equation}
\rho^{\sigma}(\gamma^n) = n\rho^{\sigma}(\gamma), \qquad \qquad \rho^{\sigma'}(\gamma) = \wind(\sigma',\sigma) + \rho^\sigma(\gamma).
\end{equation}
The rotation number can be read as $\rho^\sigma(\gamma) = \lim_{k\to\infty} \frac{1}{2k}\mu_{\CZ}^\sigma(\gamma^k)$ in terms of Conley-Zehnder indices. Let $\nu_-<0$ and $\nu_+\geq 0$ be the largest negative and smallest non-negative eigenvalues of the asymptotic operator, respectively. If $\gamma$ is non-degenerate then $\nu_+>0$. If $\gamma$ is elliptic or negative hyperbolic then $\rho^\sigma(\gamma)$ belongs to the interval $(\wind(\nu_-,\sigma),\wind(\nu_+,\sigma))$. If $\gamma$ is positive hyperbolic then $\rho^\sigma(\gamma)$ is equal to $\wind(\nu_-,\sigma) = \wind(\nu_+,\sigma) \in \Z$. In all cases $\mu^\sigma_{\CZ}(\gamma) = 2\floor{\rho^\sigma(\gamma)} + p$ when $\gamma$ is non-degenerate.

The contact form $\alpha$ is said to be non-degenerate up to action $E \in (0,+\infty]$ if every $\gamma=(x,T)$ in $\mathcal{P}(\alpha)$ satisfying $T\leq E$ is non-degenerate. If $E=+\infty$ we simply say that $\alpha$ is non-degenerate.

If $L$ is a transverse link in $(M,\xi)$ then we shall need to consider the set
\begin{equation}
\label{set_F_L}
\mathcal{F}_L = \{ f \in C^\infty(M,(0,+\infty)) \mid df(v) = 0 \ \forall v\in \xi|_L \} \, .
\end{equation}
Note that if the Reeb vector field of $\alpha$ is tangent to $L$ then the contact forms of the form $f\alpha$ with $f\in \mathcal{F}_L$ are precisely those contact forms that co-orient $\xi$ in the same way as $\alpha$ and whose Reeb vector field is tangent to $L$.

\subsection{Pseudo-holomorphic curves}\label{ssec_pseudo_curves}

\subsubsection{Finite-energy maps in symplectizations}
\label{sssec_finite_energy_curves}

Given a $d\alpha$-compatible complex structure $J:\xi\to\xi$, denote by $\jtil$ the almost complex structure on $\R\times M$ defined~by
\begin{equation}
\label{formula_J_tilde}
\jtil \cdot \partial_a = X_\alpha \, , \qquad \qquad \jtil|_\xi=J \, .
\end{equation}
Here $a$ is the $\R$-component, and $X_\alpha,\xi$ are seen as $\R$-invariant objects. The dependence on $\alpha$ is not apparent in the notation $\jtil$, but must not be forgotten. The set of almost complex structures defined as above will be denoted by $\J(\alpha)$.

Let $(S,j)$ be a closed Riemann surface, and let $\Gamma\subset S$ be finite. A smooth map $\util:S\setminus \Gamma \to \R\times M$ is called a finite-energy map if it is $\jtil$-holomorphic, i.e. $\bar\partial_{\jtil}(\util) = \frac{1}{2} ( d\util + \jtil(\util) \circ d\util \circ j ) = 0$, and if its Hofer energy
\begin{equation}\label{energy_hofer_original}
E(\util) = \sup_{\phi \in \Lambda} \int_{S\setminus\Gamma} \util^*d(\phi\alpha)
\end{equation}
satisfies $0<E(\util)<\infty$. Here $\Lambda$ is the set of smooth functions $\phi:\R\to [0,1]$ satisfying $\phi'\geq 0$. It follows that $\util$ is non-constant. Since $\jtil$ is invariant with respect to the $(\R,+)$ action on $\R\times M$ by translations of the first coordinate, the group $(\R,+)$ also acts on a finite-energy map $\util=(a,u)$ as $c \cdot \util = (a+c,u)$.

\subsubsection{Generalized finite-energy maps}
\label{sssec_gen_fin_energy_curves}

Assume that $\alpha_+=g\alpha_-$ for some smooth $g:M\to(1,+\infty)$. 
In this case we write $\alpha_+>\alpha_-$. 
Choose a smooth function $h:\R\times M\to\R$ satisfying
\begin{itemize}
\item $h(a,p)=e^{a+1}$ on $(-\infty,-1]\times M$,
\item $h(a,p)=e^{a-1}g(p)$ if $(a,p) \in [1,+\infty)\times M$,
\item $\partial_ah>0$ everywhere.
\end{itemize}
The $2$-form $\Omega = d(h\alpha_-)$ is a symplectic form on $\R\times M$.
Let $\jtil_\pm \in \J(\alpha_\pm)$. 
We will denote by $\J_\Omega(\jtil_-,\jtil_+)$ the set of almost complex structures $\jbar$ on $\R\times M$ satisfying
\begin{itemize}
\item $\jbar$ coincides with $\jtil_-$ on $(-\infty,-1]\times M$,
\item $\jbar$ coincides with $\jtil_+$ on $[1,+\infty)\times M$,
\item $\jbar$ is $\Omega$-compatible on $[-1,1]\times M$.
\end{itemize}
The set $\J_\Omega(\jtil_-,\jtil_+)$ is non-empty and contractible with the $C^\infty$-topology; weak or strong coincide here.

Let $(S,j)$ be a closed Riemann surface and $\Gamma\subset S$ be finite. 
A smooth map $\util:S\setminus \Gamma \to \R\times M$ is $\bar J$-holomorphic if $\bar\partial_{\jbar}(\util) = \frac{1}{2} \left( d\util + \jbar(\util) \circ d\util \circ j \right) = 0$, and its Hofer energy is defined as
\begin{equation}
\label{energy_non_invariant}
\begin{aligned}
E(\util) = \sup_{\phi\in\Lambda} \int_{\util^{-1}((-\infty,-1]\times M)} \util^*d(\phi\alpha_-) &+ \int_{\util^{-1}([-1,1]\times M)} \util^*\Omega \\
&+ \sup_{\phi\in\Lambda} \int_{\util^{-1}([1,+\infty)\times M)} \util^*d(\phi\alpha_+).
\end{aligned}
\end{equation}
The set $\Lambda$ was introduced in~\ref{sssec_finite_energy_curves}. 
If $0<E(\util)<\infty$ then $\util$ is called a generalized finite-energy map.

\subsubsection{Punctures and asymptotic behavior}
\label{sssec_punctures}

Denote the projection onto $\xi$ along the Reeb direction of $\alpha$ by
\begin{equation*}
\pi_\alpha : TM\to\xi.
\end{equation*}
Choose $\jtil\in\J(\alpha)$ and let
\begin{equation}\label{reference_fin_energy_curve}
\util =(a,u) : (S\setminus\Gamma,j)\to (\R\times M,\jtil)
\end{equation}
be a finite-energy map as in~\ref{sssec_finite_energy_curves}. Points in $\Gamma$ will be called punctures. If $z_* \in\Gamma$ then $z_*$ is a positive puncture if $a(z) \to +\infty$ as $z\to z_*$, it is a negative puncture if $a(z) \to -\infty$ as $z\to z_*$, or it is a removable puncture if $\limsup_{z\to z_*}|a(z)|<\infty$. It turns out that every puncture is positive, negative or removable (see~\cite{93}), and $\util$ can be smoothly continued across removable punctures. Positive or negative punctures will be called non-removable.

Let $\psi$ be a holomorphic diffeomorphism between $(\D,i,0)$ and $(V,j,z_*)$, where $V\subset S$ is a neighborhood of $z_* \in \Gamma$. We assign positive holomorphic polar coordinates $(s,t)\in[0,+\infty)\times \R/\Z$ to the point $\psi(e^{-2\pi(s+it)})\in V\setminus\{z_*\}$. Analogously, we assign negative holomorphic polar coordinates $(s,t) \in (-\infty,0]\times \R/\Z$ to the point $\psi(e^{2\pi(s+it)})\in V\setminus\{z_*\}$. In any case $(s,t)$ are called holomorphic polar coordinates at~$z_*$.

Let $z_*\in\Gamma$ be non-removable. We associate a sign $\epsilon=+1$ or $\epsilon=-1$ to $z_*$ when $z_*$ is a positive or a negative puncture, respectively. Choose $(s,t)$ holomorphic polar coordinates at $z_*$, which are positive when $\epsilon=+1$ or negative when $\epsilon=-1$. Write $\util(s,t)=(a(s,t),u(s,t))$ with respect to these coordinates.

\begin{theorem}[Hofer~\cite{93}]\label{thm_93_existence_asymptotic_limits}
For every sequence $s_n$ satisfying $\epsilon s_n \to +\infty$ one finds a subsequence $s_{n_j}$, $t_0\in\R$, and $(x,T) \in \mathcal{P}(\alpha)$ such that $\lim_{j\to+\infty} u(s_{n_j},t) = x(Tt+t_0)$ in $C^\infty(\R/\Z,M)$.
\end{theorem}

\begin{definition}\label{definition_asymptotic_limit}
The following terminology is useful.
\begin{itemize}
\item[(a)] An asymptotic limit of $\util$ at $z_*$ is some $\gamma=(x,T) \in \mathcal{P}(\alpha)$ for which there exists a sequence $s_n\to\epsilon\infty$ and some $t_0\in\R$ such that $u(s_n,t) \to x(Tt+t_0)$ in $C^0(\R/\Z,M)$ as $n\to+\infty$.
\item[(b)] The map $\util$ is weakly asymptotic to $\gamma\in \mathcal{P}(\alpha)$ if $\gamma$ is the only asymptotic limit of $\util$ at $z_*$.
\item[(c)] The map $\util$ is asymptotic to $\gamma=(x,T) \in \mathcal{P}(\alpha)$ at $z_*$ if there exists some $t_0\in\R$ such that $u(s,t) \to x(Tt+t_0)$ in $C^0(\R/\Z,M)$ as $s \to \epsilon\infty$.
\end{itemize}
\end{definition}

\begin{remark}
Theorem~\ref{thm_93_existence_asymptotic_limits} implies that asymptotic limits always exist at non-removable punctures. In order to handle degenerate contact forms we need to consider the distinction between cases (b) and (c). The main result of~\cite{sie3} states that asymptotic limits might not be unique, but it does not address the question of whether (c) follows from (b). 
These definitions do not depend on the choice of holomorphic polar coordinates.

\end{remark}

\begin{theorem}[Hofer, Wysocki and Zehnder~\cite{props1}]\label{thm_partial_asymptotics}
Assume that there is an asymptotic limit $\gamma=(x,T)$ of $\util$ at $z_*$ which is non-degenerate. Then $\util$ is asymptotic to $\gamma$ at $z_*$. Moreover, there exists $t_0\in\R$ such that $u(s,t) \to x(Tt+t_0)$ in $C^\infty$ as $\epsilon s\to+\infty$.
\end{theorem}

A more precise description is available. Suppose that $\util$ is weakly asymptotic to some orbit $\gamma = (x,T=nT_0) = \gamma_0^n$ at $z_*$, where $\gamma_0=(x,T_0)$ is simply covered. In particular, $n$ is the covering multiplicity of $\gamma$. On $\R/\Z\times \C$ with coordinates $(\theta,z=x_1+ix_2)$ we denote by $\lambda_0$ the contact form $d\theta+x_1dx_2$.

\begin{definition}\label{def_Martinet_tube}
A Martinet tube for $(\gamma_0,\alpha)$ is a pair $(U,\Psi)$ where $U$ is a neighborhood of $x(\R)$ and $\Psi$ is a diffeomorphism $U \to \R/\Z \times B$ ($B\subset\C$ is an open ball centered at the origin) such that $\Psi(x(T_0t)) = (t,0)$ and $\Psi_*\alpha = f\lambda_0$ where $f|_{\R/\Z\times\{0\}}\equiv T_0$, $df|_{\R/\Z\times\{0\}}\equiv0$.
\end{definition}

Martinet tubes always exist, see~\cite{props1}. Choose $\Psi : U \to \R/\Z \times B$ a Martinet tube for $(\gamma_0,\alpha)$. We find $s_0\geq 0$ such that $u(s,t) \in U$ for all $\epsilon s\geq s_0$. Thus $\Psi\circ u(s,t) = (\theta(s,t),z(s,t))$ is well-defined when $\epsilon s\geq s_0$, and $\theta(s,\cdot)$ has degree $n$.


\begin{theorem}[Hofer, Wysocki and Zehnder~\cite{props1}, Siefring~\cite{siefring_CPAM}]\label{thm_precise_asymptotics}
Suppose that $\gamma$ is non-degenerate. If $z(s,t)$ does not vanish identically then there exists $b>0$ and an eigenvalue $\mu$ of the asymptotic operator at $\gamma$ such that $\epsilon\mu<0$ and the following holds.
\begin{itemize}
\item For some $a_0,t_0\in\R$ and some lift $\tilde\theta:\R\times\R\to \R$ of $\theta(s,t)$ $$ \lim_{\epsilon s\to+\infty} \sup_{t\in\R/\Z} \ e^{b\epsilon s} \left( |D^\beta[a(s,t)-Ts-a_0]| + |D^\beta[\tilde\theta(s,t+t_0)-nt]| \right) = 0 $$ holds for every $\beta=(\beta_1,\beta_2)$.
\item There exists an eigenvector of $\mu$, represented as a non-vanishing vector $v(t)$ in the frame $\{\partial_{x_1},\partial_{x_2}\}$, such that $$ z(s,t+t_0) = e^{\mu s} \left( v(t) + R(s,t) \right) $$ for some $R(s,t)$ satisfying $\sup_{t\in\R/\Z} \ |D^\beta R(s,t)| \to 0$ as $\epsilon s\to+\infty$ for every pair $\beta=(\beta_1,\beta_2)$.
\end{itemize}
\end{theorem}

We will say that $\util$ has non-trivial asymptotic formula at the puncture $z_*$ if $z(s,t)$ does not vanish identically in the above statement. Otherwise we say that $\util$ has trivial asymptotic behavior at $z_*$. Under the assumptions of the above theorem, the similarity principle implies that $\util$ has non-trivial asymptotic formula at $z_*$ when $\util$ does not map the corresponding component of $S\setminus\Gamma$ into $\R\times\gamma$.

The statements of theorems~\ref{thm_93_existence_asymptotic_limits},~\ref{thm_partial_asymptotics} and~\ref{thm_precise_asymptotics} hold for the generalized finite-energy maps explained in~\ref{sssec_gen_fin_energy_curves}. The reason is because all proofs are based on the analysis of a given end of such a map, and generalized finite-energy maps are holomorphic with respect to $\R$-invariant almost complex structures near its ends. In all cases, both when the almost complex structure is $\R$-invariant as in~\ref{sssec_finite_energy_curves} or when it is not $\R$-invariant as in~\ref{sssec_gen_fin_energy_curves}, we will refer to the eigenvalue appearing in the formula for $z(s,t)$ in Theorem~\ref{thm_precise_asymptotics} as the asymptotic eigenvalue at the puncture~$z_*$, assuming the curve $\util$ has non-trivial asymptotic formula at $z_*$. In this case, if $\gamma$ is the asymptotic limit of $\util$ at $z_*$ then we denote by
\begin{equation}\label{defintion_asymptotic_winding_number}
\wind_\infty(\util,z_*,\tau) \in \Z
\end{equation}
the winding number with respect to a symplectic trivialization $\tau$ of $\xi_\gamma$ of an eigenvector for the asymptotic eigenvalue at $z_*$. It will be referred to as asymptotic winding number.

\begin{remark}\label{rmk_convention_asymptotic_evalues}
It is convenient to agree on the following. If $z(s,t)$ vanishes identically, and consequently there is no asymptotic eigenvalue, then we declare the asymptotic eigenvalue and asymptotic winding number to be $-\infty$ at a positive puncture, and $+\infty$ at a negative puncture.
\end{remark}

In order to deal with degenerate situations we need to recall a definition from~\cite{fast}. Using the notation above we consider a finite-energy curve $\util=(a,u)$ as in~\eqref{reference_fin_energy_curve}. Let $z_*$ be a non-removable puncture of $\util$ of sign $\epsilon$, choose holomorphic polar coordinates $(s,t)$ at $z_*$ of sign $\epsilon$, and choose any exponential map $\exp$ on $M$. Let $S(z_*)$ be the component of the domain around the puncture $z_*$.

\begin{definition}[\cite{fast}]\label{def_non_deg_punctures}
The puncture $z_*$ is a non-degenerate puncture of $\util$ if:
\begin{itemize}
\item $\util$ is asymptotic to some $(x,T) \in \mathcal{P}(\alpha)$ at $z_*$. 
\item $|a(s,t)-Ts-a_0| \to 0$ as $s\to\epsilon\infty$ uniformly in $t$, for some $a_0\in\R$.
\item If $\int_{S(z_*)} u^*d\alpha>0$ then $\pi_\alpha \circ du(s,t)\neq0$ if $\epsilon s$ is large enough.
\item Let $t_0\in\R$ satisfy $u(s,t)\to x(Tt+t_0)$ in $C^0$ as $s\to\epsilon\infty$, and $\zeta(s,t)$ be defined by $\exp_{x(Tt+t_0)}\zeta(s,t) = u(s,t)$. There exists $b>0$ such that $e^{\epsilon bs}|\zeta(s,t)| \to 0$ uniformly in $t$ as $s\to\epsilon\infty$.
\end{itemize}
\end{definition}

If $\util$ is a generalized finite-energy map then $\util$ is holomorphic with respect to an $\R$-invariant almost complex structure near the non-removable punctures. Hence Definition~\ref{def_non_deg_punctures} readily extends to this case.

\begin{theorem}[Corollary~6.6 from~\cite{fast}]
\label{thm_asymptotic_formula_deg_case}
Suppose that $z_*$ is a non-degenerate puncture of $\util$ in the sense of Definition~\ref{def_non_deg_punctures}, and let $\util$ be asymptotic to $\gamma=(x,T)$ at $z_*$. Let $\gamma_0$ be the underlying prime orbit of $\gamma$. Then all the conclusions from Theorem~\ref{thm_precise_asymptotics} hold true in coordinates given by any Martinet tube for $(\gamma_0,\alpha)$.
\end{theorem}

\begin{remark}\label{rem_deg_asymptotics}
In the above statement $\gamma$ might be degenerate. In view of Theorem~\ref{thm_asymptotic_formula_deg_case} we can talk about asymptotic eigenvectors and define $\wind_\infty(\util,z_*,\tau)$ exactly as before, provided the hypotheses of Theorem~\ref{thm_asymptotic_formula_deg_case} are satisfied.
\end{remark}

\subsubsection{Energy in terms of asymptotic limits}\label{sssec_energy_rmk}

As in~\ref{sssec_gen_fin_energy_curves} consider $\Omega=d(h\alpha_-)$, $J_\pm\in \J(\alpha_\pm)$, $\jbar \in \J_\Omega(\jtil_-,\jtil_+)$ and a finite-energy map $\util:(S\setminus \Gamma,j) \to (\R\times M,\jbar)$. Here $\Gamma$ consists of non-removable punctures. Assume that $\util$ is weakly asymptotic to some periodic Reeb orbit at every puncture. If the asymptotic limits of $\util$ at the positive punctures $z_1,\dots,z_n$ are $(x_1,T_1),\dots,(x_n,T_n)\in\mathcal{P}(\alpha_+)$, respectively, then $E(\util) = \sum_{k=1}^n T_n$. This follows from Stokes theorem. The same conclusion holds in the $\R$-invariant case. In the literature the reader will find other definitions of the energy which yield the same finite-energy curves, all of which can be estimated by the energy used here up to a positive factor independent of the curves.

\subsubsection{Classical algebraic invariants}
\label{sssec_classical_alg_inv}

Fix a $d\alpha$-compatible complex structure $J$ on $\xi$, and a finite-energy map $\util=(a,u)$ as in~\ref{sssec_finite_energy_curves}. We make no non-degeneracy assumption on $\alpha$, but we do assume that every puncture is non-removable and non-degenerate as in Definition~\ref{def_non_deg_punctures}. Then $\pi_\alpha \circ du(z):(T_zS,j) \to (\xi|_{u(z)},J)$ is a holomorphic linear map for all $z\in S\setminus \Gamma$. Seen as a section of the appropriate bundle, $\pi_\alpha\circ du$ satisfies a Cauchy-Riemann type equation. It follows from the similarity principle that either $\pi_\alpha\circ du$ vanishes identically, or all its zeros are isolated and contribute positively to the algebraic count of zeros. Moreover, $\pi_\alpha \circ du$ vanishes identically if, and only if, $u(S\setminus\Gamma)$ is contained on trivial cylinders over periodic Reeb orbits, and if $\pi_\alpha\circ du$ does not vanish identically on a given component of $S\setminus\Gamma$ then $\pi_\alpha\circ du$ has finitely many zeros there; see Definition~\ref{def_non_deg_punctures}.

Assume that $\pi_\alpha\circ du$ does not vanish identically on all connected components of its domain. In~\cite{props2} the integer-valued invariant
\begin{equation}
\wind_\pi(\util) \geq 0
\end{equation}
was defined as the algebraic count of zeros of $\pi_\alpha\circ du$. We can find a $d\alpha$-symplectic trivialization of $u^*\xi$ such that it converges at the end of $S\setminus\Gamma$ corresponding to $z\in\Gamma$ to a $d\alpha$-symplectization $\tau_z$ along the asymptotic limit, for each $z$. The invariant $\wind_\infty(\util)\in\Z$ is defined as
\begin{equation}
\wind_\infty(\util) = \sum_{z\in\Gamma_+} \wind(\util,z,\tau_z) - \sum_{z\in\Gamma_-} \wind(\util,z,\tau_z)
\end{equation}
where we split $\Gamma = \Gamma^+ \cup \Gamma^-$ into positive and negative punctures, assuming that $\pi_\alpha\circ du$ does not vanish identically on every component of $S\setminus\Gamma$.

\begin{remark}\label{rmk_wind_pi_wind_infty}
The identity $\wind_\infty(\util) = \wind_\pi(\util) + \chi(S) - \#\Gamma$ follows from basic degree theory, see details in~\cite{props2}.
\end{remark}

\subsubsection{Moduli spaces}
\label{sssec_moduli_spaces}

Fix a symplectic form $\Omega$ as in~\ref{sssec_gen_fin_energy_curves}. Suppose that the following data is given: an integer $g\geq 0$, $\jtil_+ \in \J(\alpha_+)$, $\jtil_- \in \J(\alpha_-)$ and $\bar J \in \J_\Omega(\jtil_-,\jtil_+)$, $\gamma_1^+,\dots,\gamma_{m_+}^+ \in \mathcal{P}(\alpha_+)$ and $\gamma_1^-,\dots,\gamma_{m_-}^- \in \mathcal{P}(\alpha_-)$, $\delta_1^+,\dots,\delta_{m_+}^+ \leq 0$ and $\delta_1^-,\dots,\delta_{m_-}^- \geq 0$ such that $\delta_i^\pm$ is not in the spectrum of the asymptotic operator associated with $(J_\pm,\alpha_\pm,\gamma_i^\pm)$.
Denote by $T_k^+$ the action (period) of $\gamma_k^+$.
Denote $\delta=(\delta_1^+,\dots,\delta_{m_+}^+;\delta_1^-,\dots,\delta_{m_+}^-)$. The moduli space
\begin{equation}\label{moduli_space_notation}
\M_{\jbar,g,\delta}(\gamma_1^+,\dots,\gamma_{m_+}^+;\gamma_1^-,\dots,\gamma_{m_-}^-)
\end{equation}
is defined as the set of equivalence classes of tuples $(\util,S,j,\Gamma^+,\Gamma^-)$ consisting of a closed connected genus $g$ Riemann surface $(S,j)$, two disjoint finite and ordered subsets $\Gamma^\pm=\{z_1^\pm,\dots,z_{m_\pm}^\pm\}$ 
of $S$, and a finite-energy $\jbar$-holomorphic map $\util : (S\setminus (\Gamma^+\cup\Gamma^-),j) \to (\R\times M,\bar J)$ which has non-degenerate punctures as in Definition~\ref{def_non_deg_punctures}, a positive puncture at $z^+_i$ where it is asymptotic to $\gamma_i^+$, and a negative puncture at $z^-_i$ where it is asymptotic to~$\gamma_i^-$. Moreover, the asymptotic eigenvalue (Remark~\ref{rem_deg_asymptotics}) of $\util$ at $z_i^+$ is smaller than $\delta_i^+$, and the asymptotic eigenvalue of $\util$ at $z_i^-$ is larger than $\delta_i^-$. We declare $(\util_0,S_0,j_0,\Gamma^+_0,\Gamma^-_0)$ and  $(\util_1,S_1,j_1,\Gamma^+_1,\Gamma^-_1)$ equivalent if there is a biholomorphism $\phi:(S_0,j_0) \to (S_1,j_1)$ that defines order preserving bijections $\phi:\Gamma^\pm_0 \to \Gamma^\pm_1$ and satisfies $\util_1 \circ \phi = \util_0$. An element of~\eqref{moduli_space_notation} is called a pseudo-holomorphic curve. It will be called embedded, immersed or somewhere injective provided that it can be represented as finite-energy map $\util$ that is an embedding, immersion or a somewhere injective map, respectively. The assumption that punctures are non-degenerate and Theorem~\ref{thm_asymptotic_formula_deg_case} were used to guarantee that we can talk about asymptotic eigenvalues; see Remark~\ref{rmk_convention_asymptotic_evalues}.


Let $\tau=(\tau_1^+,\dots,\tau_{m_+}^+;\tau_1^-,\dots,\tau_{m_-}^-)$ be $d\alpha_\pm$-symplectic trivializations of $\xi_{\gamma_i^\pm}$. The virtual dimension of $\M_{\jbar,g,\delta}(\gamma_1^+,\dots,\gamma_{m_+}^+;\gamma_1^-,\dots,\gamma_{m_-}^-)$ at $[\util,S,j,\Gamma^+,\Gamma^-]$ is
\begin{equation}\label{virtual_dimension}
-(2-2g-m_+-m_-) + c_1^\tau(\util^*T(\R\times M)) + \sum_{i=1}^{m_+} \mu_{\CZ}^{\tau_i^+,\delta_i^+}(\gamma_i^+) - \sum_{i=1}^{m_-} \mu_{\CZ}^{\tau_i^-,\delta_i^-}(\gamma_i^-)
\end{equation}
where $c_1^\tau(\util^*T(\R\times M))$ is a relative first Chern number given as follows. First, $\Omega$-symplecti\-cally trivialize $\util^*T(\R\times M)$  in such a way that at the end $z_i^\pm \in \Gamma^\pm$ this trivialization splits as trivializations $\{\partial_a,X_{\alpha_\pm}\} \oplus \sigma_i^\pm$, with some $d\alpha_\pm$-symplectic trivialization $\sigma_i^\pm$ of $\xi_{\gamma_i^\pm}$. Then set
\begin{equation}\label{relative_Chern_number_formula}
c_1^\tau(\util^*T(\R\times M)) = 2 \left( \sum_{i=1}^{m_+} \wind(\sigma_i^+,\tau_i^+) - \sum_{i=1}^{m_-} \wind(\sigma_i^-,\tau_i^-) \right).
\end{equation}

\begin{remark}
This discussion applies to the case $\alpha_-=\alpha_+=\alpha$ and $\jbar=\jtil\in\J(\alpha)$ to define moduli spaces $\M_{\jtil,g,\delta}(\gamma_1^+,\dots,\gamma_{m_+}^+;\gamma_1^-,\dots,\gamma_{m_-}^-)$, where $\gamma^\pm_i \in \mathcal{P}(\alpha)$.
\end{remark}

Let $C$ be a curve in $\M_{\jtil,g,\delta}(\gamma_1^+,\dots,\gamma_{m_+}^+;\gamma_1^-,\dots,\gamma_{m_-}^-)$ be represented by the tuple $(\util=(a,u),S,j,\Gamma^+,\Gamma^-)$. Using the complex structure we can compactify $S\setminus (\Gamma^+\cup\Gamma^-)$ by adding at its ends the circles made of rays issuing from the origin of the tangent spaces of the corresponding punctures. We obtain a compact oriented surface with boundary~$\overline S$. Its fundamental class ($\Z$ coefficients) will be denoted by $[\overline S,\partial\overline S] \in H_2(\overline S,\partial\overline S)$. Using Theorem~\ref{thm_asymptotic_formula_deg_case} we can smoothly extend $u$ to a map $\overline u:\overline S \to M$. We call $\overline u_*[\overline S,\partial\overline S] \in H_2(M,\gamma_1^+\cup\dots\cup\gamma_1^-\cup \dots)$ the homology class induced by $C$. It does not depend on the choice of representative.


\subsubsection{Intersection numbers}\label{sssec_intersection_numbers}

We review here basic facts of the intersection theory for punctured pseudo-holomorphic curves in dimension four, as developed by Siefring in~\cite{siefring}. 
Fix non-degenerate contact forms $\alpha_-,\alpha_+$ that define $\xi$ and satisfy $\alpha_+>\alpha_-$. Later we will describe situations where the non-degeneracy assumption can be relaxed; see Remark~\ref{rmk_non_deg_relaxed}.

Let $\jtil_\pm\in\J(\alpha_\pm)$ and $\bar J \in \J_\Omega(\jtil_-,\jtil_+)$ be as in~\ref{sssec_gen_fin_energy_curves}. Choose a collection $\tau$ of symplectic trivializations of $\xi_{\gamma_0}$ for every simply covered periodic orbit $\gamma_0$ of $X_{\alpha_+}$ and of $X_{\alpha_-}$. Let $\util=(a,u):S\setminus\Gamma \to \R\times M$ be a finite-energy $\jbar$-holomorphic map, where $(S,j)$ is a closed Riemann surface and $\Gamma\subset S$ is a finite set of non-removable punctures. Decompose $\Gamma = \Gamma^+ \sqcup \Gamma^-$ into positive and negative punctures. At each $w \in\Gamma_\pm$ the asymptotic limit of $\util$ is denoted by $\gamma_w^{n_w}$ where $\gamma_w=(x_w,T_w)\in \mathcal{P}(\alpha_\pm)$ is prime and $n_w$ is its covering multiplicity. Choose holomorphic polar coordinates $(s,t)$ at $w$ which are positive or negative according to whether $w$ is a positive or a negative puncture. We can write the components $\util(s,t)=(a(s,t),u(s,t))$ as functions of $(s,t)$ near the punctures. Choose Martinet tubes $(U_w,\Psi_w)$ around $\gamma_w$, with respect to $\alpha_+$ or $\alpha_-$ accordingly, with coordinates $(\theta,z=(x_1+ix_2))$ satisfying the requirements of Definition~\ref{def_Martinet_tube} and aligned with $\tau$ in the following sense: $\partial_{x_1}$ gets represented by $\tau$ as a loop with winding number zero.

Let $w$ be a positive puncture. By the asymptotic behavior described in Theorem~\ref{thm_precise_asymptotics}, if $R\gg1$ then $(s,t) \mapsto (a(s,t),t)$ defines a smooth proper orientation-preserving embedding of $(R,+\infty) \times \R/\Z$ onto an end of $[0,+\infty)\times\R/\Z$. Hence we can use $(a,t)$ as new polar coordinates near $w$, with respect to which the map $\util$ is written as $$ (a,t) \mapsto (a,u(a,t)). $$ The analogous conclusion holds near a negative puncture. If $|a|$ is large enough we can again use Theorem~\ref{thm_precise_asymptotics} to conclude that the loop $t\mapsto u(a,t)$ lies in $N_w$. Hence
$$
\Psi_w \circ u(a,t) = (\theta(a,t),z(a,t))
$$
are well-defined functions of $(a,t)$ with $|a|\gg1$. Fix a smooth cut-off function $\beta:\R\to[0,1]$ that vanishes near $(-\infty,0]$ and is identically equal to $1$ near $[1,+\infty)$. Taking $\varepsilon>0$ small and $r>0$ large, define $\util^{\tau,w,\varepsilon,r}$ by
\begin{equation}
\util^{\tau,w,\varepsilon,r}(a,t) = (a,\theta(a,t),z(a,t) + \varepsilon\beta(a-r)) \qquad a\geq r
\end{equation}
if $w$ is a positive puncture, or by
\begin{equation}
\util^{\tau,w,\varepsilon,r}(a,t) = (a,\theta(a,t),z(a,t) + \varepsilon\beta(-a-r)) \qquad a\leq-r
\end{equation}
if $w$ is a negative puncture. Repeating this construction at all $w\in\Gamma$ using  common parameters $\varepsilon,r>0$ such that $\min\{\varepsilon^{-1},r\}$ is large enough, we obtain a small neighborhood~$V$ of $\Gamma$ and a smooth map defined on $V\setminus \Gamma$, which can be smoothly extended to a map
\begin{equation*}
\util^{\tau,\varepsilon,r} : S\setminus\Gamma \to \R\times M
\end{equation*}
by setting it to be equal to $\util$ on $S\setminus V$. The dependence on $\tau$ is hidden in the choice of the Martinet tubes.

If $\util_0$ and $\util_1$ are finite-energy $\bar J$-holomorphic curves, then Siefring~\cite{siefring} defines
\begin{equation}
i^{\tau}(\util_0,\util_1) = {\rm int}(\util_0,\util^{\tau,\varepsilon,r}_1)
\end{equation}
where ${\rm int}$ stands for the oriented intersection number. This is well-defined since for $\min\{\varepsilon^{-1},r\}$ large enough the maps $\util_0$ and $\util^{\tau,\varepsilon,r}_1$ do not intersect each other on the ends of their domains. Moreover, it depends only on the data $(\tau,\util_0,\util_1)$ as one easily checks. It is also symmetric: $i^{\tau}(\util_0,\util_1) = i^\tau(\util_1,\util_0)$.

It is interesting to compute the difference $i^{\tau'}(\util_0,\util_1) - i^\tau(\util_0,\util_1)$. This can be understood in terms of braided knots. Suppose that for $j\in\{0,1\}$ we have positive punctures $z_j$ of $\util_j$ with asymptotic limit $\gamma^{n_j}$, where $\gamma$ is a simply covered Reeb orbit and $n_j\in\N$. Assume for simplicity that $\util_0,\util_1$ are somewhere injective. By Theorem~\ref{thm_precise_asymptotics}, setting the $\R$-coordinate $a$ in $\R\times M$ equal to a large constant near the punctures $z_j$, the maps $\util_0$, $\util_1^{\tau,\varepsilon,r}$ and $\util_1^{\tau',\varepsilon,r}$ single out (up to small perturbation) braided knots $k_0(a)$, $k_1^\tau(a)$ and $k_1^{\tau'}(a)$, respectively, on a tubular neighborhood of~$\gamma$. Consider $r$ so large that the corresponding ends of $\util_0$ and $\util_1^{\tau,\varepsilon,r}$, and of $\util_0$ and $\util_1^{\tau',\varepsilon,r}$ do not intersect for all $a\geq r$. If $a\geq r$ then all the $n_1$ strands $k_1^{\tau'}(a)$ will wind around all the $n_0$ strands of $k_0(a)$ $\wind(\tau,\tau')$ more times than all the strands of $k_1^{\tau}(a)$ wind around all the $n_0$ strands of $k_0(a)$. It follows that the contribution to the difference $i^{\tau'}(\util_0,\util_1)-i^\tau(\util_0,\util_1)$ coming from $(z_0,z_1)$ is equal to $n_0n_1\wind(\tau,\tau')$. Negative punctures are treated similarly. The argument when curves are not somewhere injective is a straightforward modification. We arrive at~\cite[Proposition~4.1 item (3)]{siefring}
\begin{equation}\label{intersection_difference}
i^{\tau'}(\util_0,\util_1) - i^\tau(\util_0,\util_1) = {\sum}^+_{(z_0,z_1)} n_0n_1\wind(\tau,\tau') - {\sum}^-_{(z_0,z_1)} n_0n_1\wind(\tau,\tau')
\end{equation}
where $\Sigma^+$ $(\Sigma^-$) indicates sum over all pairs of positive (negative) punctures $(z_0,z_1)$ where $\util_0,\util_1$ are asymptotic to covers of a common simply covered Reeb orbit. As in~\cite{siefring} consider
\begin{equation}\label{complementary_bilinear_form}
\begin{aligned}
\Omega^\tau(\util_0,\util_1) &= {\sum}^+_{(z_0,z_1)}  \ n_0n_1 \max \left\{ \frac{\floor{\rho^\tau(\gamma^{n_0})}}{n_0} , \frac{\floor{\rho^\tau(\gamma^{n_1})}}{n_1} \right\} \\
& + {\sum}^-_{(z_0,z_1)}  \ n_0n_1 \max \left\{ \frac{\floor{-\rho^\tau(\gamma^{n_0})}}{n_0} , \frac{\floor{-\rho^\tau(\gamma^{n_1})}}{n_1} \right\}.
\end{aligned}
\end{equation}
Using~\eqref{intersection_difference} we get $\Omega^\tau(\util_0,\util_1) - \Omega^{\tau'}(\util_0,\util_1) = i^{\tau'}(\util_0,\util_1) - i^{\tau}(\util_0,\util_1)$. Siefring defines the generalized intersection number
\begin{equation}\label{formula_generalized_intersection_number}
\util_0 * \util_1 = i^{\tau}(\util_0,\util_1) + \Omega^\tau(\util_0,\util_1)
\end{equation}
which is independent of the choice of $\tau$ in view of the above calculations. The generalized intersection number for curves is defined by the generalized intersection number of maps representing them.

Of course, the case $\alpha_-=\alpha_+$ and $\jbar = \jtil\in\J(\alpha)$ is a special case of the above discussion.

As explained in~\cite{siefring}, one of the many motivations to consider this intersection number is that it is preserved under smooth homotopies of asymptotically cylindrical maps. We have not defined here this class of maps, see~\cite[subsection~2.3]{siefring}, but we state a special case of~\cite[Proposition~4.3 item (1)]{siefring}.

\begin{proposition}\label{prop_constancy_int_number}
The number $\util * \vtil$ does not change when $\util$ or $\vtil$ vary on smooth families of finite-energy curves with fixed asymptotic limits.
\end{proposition}

Crucial to our analysis is the adjunction inequality. Choose an $\Omega$-symplectic trivialization of $\util^*T(\R\times M)$ that splits as $\{\partial_a,X_{\alpha_\pm}\} \oplus \sigma_w$ at each puncture $w\in\Gamma_\pm$, where $\sigma_w$ is a $d\alpha_\pm$-symplectic trivialization of $\xi_{\gamma_w^{n_w}}$. Denote
\[
\begin{aligned}
& \mu_{\CZ}(\util) = \sum_{w\in\Gamma_+} \mu_{\CZ}^{\sigma_w}(\gamma_w^{n_w}) - \sum_{w\in\Gamma_-} \mu_{\CZ}^{\sigma_w}(\gamma_w^{n_w}) \\
& \bar\sigma(\util) = \sum_{w\in\Gamma} \gcd(n_w,\lfloor\rho^{\Phi_w}(\gamma_w^{n_w})\rfloor)
\end{aligned}
\]
where, for each $w\in\Gamma$, $\Phi_w$ denotes a trivialization of $\xi$ along the underlying primitive orbit  $\gamma_w$.
Let $\Gamma_{\rm odd}$ denote the set of punctures where the asymptotic limit has odd Conley-Zehnder index. The following is a special case of~\cite[Theorem~2.3]{siefring}.

\begin{theorem}[Siefring]\label{thm_adjunction_ineq}
If $\util$ is somewhere injective then $$ \util*\util-\frac{1}{2}\mu_{\CZ}(\util)+\frac{1}{2}\#\Gamma_{\rm odd}+\chi(S)-\bar\sigma(\util)\geq0 $$ and equality implies that $\util$ is an embedding.
\end{theorem}

\begin{remark}\label{rmk_non_deg_relaxed}
Assume that there exists a constant $E>0$ such that every periodic Reeb orbit $(x,T) \in \mathcal{P}(\alpha_+) \cup \mathcal{P}(\alpha_-)$ with $T\leq E$ is non-degenerate. Then all definitions discussed here make sense and all results hold when applied to curves with energy at most equal to $E$.
\end{remark}

\subsubsection{Contact forms with common closed Reeb orbits}\label{sssec_sharing_orbits}

Let $L\subset (M,\xi)$ be a transverse oriented link with components $\gamma_1,\dots,\gamma_n$. Assume that the Reeb vector fields of two defining contact forms $\alpha_+>\alpha_-$ are positively tangent to $L$. Thus $\gamma_i$ may be viewed as prime closed $\alpha_\pm$-Reeb orbits. Let $\jtil_\pm\in\J(\alpha_\pm)$ and let $\Omega$ be a symplectic form as considered in~\ref{sssec_gen_fin_energy_curves}. We will denote by $\J_{\Omega,L}(\jtil_-,\jtil_+) \subset \J_\Omega(\jtil_-,\jtil_+)$ the subset consisting of those $\bar J$ which leave the tangent space of $\R\times L$ invariant. One can check that this set is non-empty and contractible when equipped with the $C^\infty$-topology (strong equals weak here).

\subsubsection{SFT Compactness}\label{ssec_SFT_comp}

In preparation for the proofs we need to review the language necessary to use the SFT Compactness Theorem~\cite{sftcomp}. Fix contact forms $\alpha_\pm$ that define $\xi$ and satisfy $\alpha_+ > \alpha_-$. Let $\jtil_\pm\in\J(\alpha_\pm)$, let $\Omega$ be a symplectic form as considered in~\ref{sssec_gen_fin_energy_curves}, and let $\jbar \in \J_\Omega(\jtil_-,\jtil_+)$.

\subsubsection*{Nodal curves}

A nodal holomorphic curve in $(\R\times M,\jbar)$ without marked points, also called a holomorphic building of height $1$ without marked points, is the equivalence class of a tuple $$ (\util,S,j,\Gamma^+,\Gamma^-,D) $$ consisting of a (possibly disconnected) closed Riemann surface $(S,j)$, disjoint finite ordered sets $\Gamma^+,\Gamma^-\subset S$ of distinct points, a finite unordered set $D$ of unordered pairs of points in $S$, called nodal pairs, such that all nodal points together with points in $\Gamma^+\cup\Gamma^-$ make a set of distinct points of $S$, and a finite-energy $\jbar$-holomorphic map
\[
\util:S\setminus(\Gamma^+\cup\Gamma^-) \to \R\times M
\]
having positive punctures at $\Gamma^+$, negative punctures at $\Gamma^-$, satisfying $\util(z)=\util(w)$ for all $\{z,w\}\in D$. Points of $\Gamma^+\cup\Gamma^-$ will be called punctures. By nodal points we mean points belonging to nodal pairs; we may abuse the notation and still write $D$ to denote the set of nodal points. Two such tuples $$ (\util,S,j,\Gamma^+,\Gamma^-,D) \sim (\util',S',j',\Gamma'^+,\Gamma'^-,D') $$ are declared equivalent if there is a biholomorphim $\phi:(S,j)\to(S',j')$ that determines order preserving bijections $\Gamma^\pm\to\Gamma'^\pm$ and a bijection $D\to D'$ respecting pairs, and satisfies $\util'\circ\phi=\util$. We may refer to $\phi$ simply as an isomorphism.

The nodal curve is called connected if the space obtained from $S$ by identifying points in each nodal pair is connected. 
It is called stable if for every connected component $S_*\subset S$ such that $\util|_{S_*\setminus(\Gamma_+\cup\Gamma_-)}$ is constant the inequality $2g_*+\mu_*\geq 3$ holds, where $g_*$ is the genus of $S_*$ and $\mu_*$ is the total number of punctures and nodal points in $S_*$. When there are no nodes the curve is said to be smooth.

\begin{remark}\label{rmk_blown_up_circles}
If $(S,j)$ is a Riemann surface and $z\in S$ then the conformal structure endows the circle $(T_zS\setminus0)/\R_+$, called the blown up circle at $z$, with a metric which makes it isometric to $\R/2\pi\Z$ with its usual metric.
\end{remark}

The arithmetic genus of a connected nodal curve $(\util,S,j,\Gamma^+,\Gamma^-,D)$ is, by definition, $$ g=\frac{1}{2}d-b+\sum_{i=1}^bg_i+1 $$ where $d$ is the number of nodal points, $b$ is the number of connected components of $S$ and $\sum_ig_i$ is the sum of their genera.

\begin{remark}
Assume that $(\util,S,j,\Gamma^+,\Gamma^-,D)$ is a connected nodal curve. If we choose a collection $r$ of orientation reversing isometries of blown up circles at nodal pairs, and use it to glue these circles obtaining a connected closed surface $S^{D,r}$ then the arithmetic genus is equal to the genus of $S^{D,r}$.
\end{remark}

The above notion of nodal curve in $(\R\times M,\jbar)$ can be adapted in a straightforward manner to that of a nodal curve in $(\R\times M,\jtil_\pm)$. The only difference is in the concept of stability: one further requires the existence of at least one component $S_* \subset S$ such that $\util|_{S_*}$ is not an unbranched cover of a cylinder over some closed Reeb orbit (trivial cylinder).

Let $(\util,S,j,\Gamma^+,\Gamma^-,D)$ represent a nodal curve, and write $\util=(a,u)$ in components. If we denote by $\bar S$ the smooth surface obtained from $S$ by removing all punctures and nodal points and adding the corresponding blown up circles, and assume that $\alpha_+,\alpha_-$ are non-degenerate up to action $E(\util)$, then we can use Theorem~\ref{thm_partial_asymptotics} to continuously extend the map $u$ to a map
\begin{equation}\label{map_bar}
\bar u:\bar S \to M.
\end{equation}
It maps blown up circles at punctures to the corresponding asymptotic closed Reeb orbits, and blown up circles at nodal points to the corresponding point in~$M$. Blown up circles are oriented as the boundary of $\bar S$.

\subsubsection*{Holomorphic buildings}

Nodal curves are, by definition, buildings of height~$1$. We pass now to the description of a general holomorphic building ${\bf u}$ of height $k_-|1|k_+$ in $(\R\times M,\jbar)$ without marked points, see~\cite[section~7]{sftcomp}, where $k_\pm\geq0$. It is the equivalence class of the data consisting of an ordered collection of nodal holomorphic curves
\begin{equation*}
\{\util_m\} = \{(\util_m=(a_m,u_m),S_m,j_m,\Gamma^+_m,\Gamma^-_m,D_m)\} \qquad m\in\{-k_-,\dots,k_+\}
\end{equation*}
as above, such that
\begin{itemize}
\item for each $m<0$, $\util_m$ defines a stable nodal curve in $(\R\times M,\jtil_-)$,
\item $\util_0$ defines a stable nodal curve in $(\R\times M,\jbar)$,
\item for each $m>0$, $\util_m$ defines a stable nodal curve in $(\R\times M,\jtil_+)$,
\end{itemize}
and a collection $\{\Phi_m\}$, $m\in\{-k_-,\dots,k_+-1\}$, of orientation reversing isometries
\[
\Phi_m: \bigcup_{z\in\Gamma_m^+} (T_zS_m\setminus0)/\R_+ \to \bigcup_{z\in\Gamma_{m+1}^-} (T_zS_{m+1}\setminus0)/\R_+
\]
covering order preserving bijections $\Gamma_m^+ \to \Gamma_{m+1}^-$,
such that the following is true:
\begin{equation}\label{mathcing_conditions_buildings}
\text{$\bar u_{m+1}\circ \Phi_m$ coincides with $\bar u_m$ on $(T_zS_m\setminus0)/\R_+$, for every $z\in\Gamma_m^+$.}
\end{equation}
Here each $\bar u_m$ is as in~\eqref{map_bar}, which is well-defined under the assumption that $\alpha_+,\alpha_-$ are non-degenerate up to action $E$ for some constant satisfying $\max_m E(\util_m)\leq E$. The $\util_m$ define nodal holomorphic curves called the levels, or stores, of the building. Two such collections $\{\{\util_m\},\{\Phi_m\}\}$ and $\{\{\util'_m\},\{\Phi'_m\}\}$ of data are equivalent if they have the same number of levels, and there are isomorphisms between the corresponding levels (in the sense explained before), whose linearizations intertwine the corresponding orientation reversing isometries $\{\Phi_m\}$ and $\{\Phi'_m\}$ at corresponding blown up circles. Moreover, synchronized reordering of the intermediate punctures also define equivalent buildings.

Let $r$ be an arbitrary collection of orientation reversing isometries between the blown up circles at points in the nodal pairs $\cup_mD_m$. Denote
\begin{equation*}
S^{{\bf u},r} = \left( \sqcup_m \bar S_m \right) / \sim
\end{equation*}
where $\Phi_m$ identifies blown up circles at points in $\Gamma_m^+$ with blown up circles at corresponding points in $\Gamma_{m+1}^-$, and $r$ identifies blown up circles at nodal pairs. The circles in the interior of $S^{{\bf u},r}$ obtained by these identifications of blown up circles will be called special circles.

\begin{remark}
\label{rmk_subsets}
One may see $S_m\setminus (\Gamma^+_m\cup\Gamma^-_m\cup D_m)$ as an open subset of $S^{{\bf u},r}$.
\end{remark}

Condition~\eqref{mathcing_conditions_buildings} and Theorem~\ref{thm_partial_asymptotics} allow us to define a continuous map
\begin{equation*}
F_{\bf u}: S^{{\bf u},r} \to [-\infty,+\infty]\times M
\end{equation*}
as the unique continuous map that agrees with $\util_0$ on $S_0\setminus (\Gamma^+_0\cup\Gamma^-_0\cup D_0)$, agrees with $(+\infty,u_m)$ on $S_m\setminus (\Gamma^+_m\cup\Gamma^-_m\cup D_m)$ for all $m>0$, and with $(-\infty,u_m)$ on $S_m\setminus (\Gamma^+_m\cup\Gamma^-_m\cup D_m)$ for all $m<0$. Here we abuse notation and write $D_m$ for sets of nodal points.

Similarly one defines holomorphic buildings in $(\R\times M,\jtil_\pm)$. The difference is that there is no distinction between upper and lower levels. All levels are nodal curves in the same symplectization. Here we simply consider $F_{\bf u}$ as a map $S^{{\bf u},r} \to M$ built from the $M$-components of the levels in a similar manner.

\begin{lemma}
\label{lemma_constant_components}
Let $\bf u$ be a stable holomorphic building without marked points 
such that $S^{{\bf u},r}$ is a connected genus zero surface for some (and hence any) choice of~$r$. Suppose that there exists a level $(\util_m,S_m,j_m,\Gamma^+_m,\Gamma^-_m,D_m)$ of ${\bf u}$ and a connected component $Y \subset S_m$ such that $\util_m$ is constant on $Y \setminus (\Gamma^+_m\cup\Gamma^-_m) = Y$. Then there exist connected components $Y'\neq Y''$ of $S_m$ such that 
\begin{itemize}
\item $\util_m$ is non-constant on $Y' \setminus (\Gamma^+_m\cup\Gamma^-_m)$,
\item $\util_m$ is non-constant on $Y'' \setminus (\Gamma^+_m\cup\Gamma^-_m)$,
\item $\util_m(Y) \subset \util_m(Y' \setminus (\Gamma^+_m\cup\Gamma^-_m)) \cap \util_m(Y'' \setminus (\Gamma^+_m\cup\Gamma^-_m))$.
\end{itemize}
\end{lemma}

\begin{proof}
Note that there are no punctures in~$Y$, otherwise $\util_m$ would not be constant on $Y$. Also, $Y$ has genus zero. By stability there are at least three nodal points in~$Y$, and $Y$ contains no nodal pair; otherwise $S^{{\bf u},r}$ would not have genus zero. Hence we can find $Y_0,Y_1$ connected components of $S_m$, both different $Y$, for which there exist nodal pairs $\{z_0,z'_0\},\{z_1,z_1'\}$ satisfying $z_0,z_1 \in Y$, $z'_0\in Y_0$, $z'_1\in Y_1$. Hence we have 
\[
\util_m(z'_1)= \util_m(Y) = \util_m(z'_0).
\]
Moreover, $Y_0\neq Y_1$ again by the genus zero assumption. If $\util_m$ is not constant both on $Y_0\setminus (\Gamma^+_m\cup\Gamma^-_m)$ and on $Y_1\setminus (\Gamma^+_m\cup\Gamma^-_m)$ then we are done: define $Y'=Y_0\setminus (\Gamma^+_m\cup\Gamma^-_m)$, $Y''=Y_1\setminus (\Gamma^+_m\cup\Gamma^-_m)$.

Suppose $\util_m$ is constant on $Y_1\setminus (\Gamma^+_m\cup\Gamma^-_m)$; we must have $Y_1 = Y_1\setminus (\Gamma^+_m\cup\Gamma^-_m)$ in this case. The genus zero assumption implies that $Y_1\cup Y \cup Y_0$ contains exactly two nodal pairs, namely $\{z_0,z'_0\},\{z_1,z_1'\}$; it may contain many more nodal points though. Since $Y_1$ has genus zero, by stability we find at least three nodal points in $Y_1$. Let $Y_2 \not \subset Y_1\cup Y\cup Y_0$ be a connected component of $S_m$ and $\{z_2,z_2'\}$ be a nodal pair such that $z_2 \in Y_1$ and $z_2'\in Y_2$; these exist by the genus zero assumption on $S^{{\bf u},r}$. Note that
\[
\util_m(z_2') = \util_m(Y) = \util_m(z'_0).
\]
If $\util_m$ is not constant both on $Y_2\setminus (\Gamma^+_m\cup\Gamma^-_m)$ and in $Y_0\setminus (\Gamma^+_m\cup\Gamma^-_m)$ then we are done: just set $Y'=Y_0\setminus (\Gamma^+_m\cup\Gamma^-_m)$, $Y''=Y_2\setminus (\Gamma^+_m\cup\Gamma^-_m)$. This process must stop, otherwise we find infinitely many components of $S_m$. As a consequence, we find nodal pairs $\{z_1,z_1'\},\dots,\{z_k,z_k'\}$ and connected components $Y_1,\dots,Y_k$ of $S_m$ such that $z_1 \in Y$, $\{z_1',z_2\} \subset Y_1$, $\{z_2',z_3\} \subset Y_2$, $\dots$, $\{z'_{k-1},z_k\} \subset Y_{k-1}$, $z'_k \in Y_k$, $\util_m$ is constant on $Y_i \setminus (\Gamma^+_m\cup\Gamma^-_m)$ for all $i=1,\dots,k-1$, and $\util_m$ is not constant on $Y_k$. In particular, $\util_m(z'_k) = \util_m(Y)$.

If $\util_m$ is not constant on $Y_0 \setminus (\Gamma^+_m\cup\Gamma^-_m)$ then we are done since the genus zero assumption implies that $Y_0 \neq Y_k$. If $\util_m$ is constant on $Y_0\setminus (\Gamma^+_m\cup\Gamma^-_m)$ then we can repeat the above process to find a connected component $Y_{-j}$ of $S_m$, $j\geq1$, connected through a chain of nodal points to $Y$, such that $\util_m$ is not constant on $Y_{-j}\setminus \Gamma^+_m\cup\Gamma^-_m$, and a point $z_{-j}' \in Y_{-j} \setminus \Gamma^+_m\cup\Gamma^-_m$ such that $\util_m(z_{-j}')=\util_m(Y)$. The genus zero assumption implies that $Y_{-j} \neq Y_k$.
\end{proof}

\subsubsection*{SFT-convergence}

Let $(\vtil_l=(b_l,v_l),\Sigma_l,j_l,Z_l^+,Z_l^-)$ be a sequence of connected holomorphic curves in $(\R\times M,\jbar)$ satisfying $E(\vtil_l)\leq E$ for some $E>0$ such that both $\alpha_+$ and $\alpha_-$ are non-degenerate up to action $E$ . By Theorem~\ref{thm_precise_asymptotics} we have unique associated continuous maps
\begin{equation*}
F_{\vtil_l} : \bar\Sigma_l \to [-\infty,+\infty]\times M
\end{equation*}
equal to $\vtil_l$ on $\bar\Sigma_l \setminus (\text{blown up circles at punctures}) = \Sigma_l\setminus (Z_l^+ \cup Z_l^-)$. It maps blown up circles onto corresponding asymptotic limits.

This sequence of curves SFT-converges to the building~${\bf u}$ if there is a collection $r$ of orientation reversing isometries of blown up circles at nodal pairs in $\cup_mD_m$, there are ordered sets of distinct points $K_l \subset \Sigma_l$, $K \subset \sqcup_mS_m$ (additional marked points) of the same cardinality, and diffeomorphisms
\begin{equation}\label{diffeos_building}
\varphi_l : S^{{\bf u},r} \to \bar \Sigma_l
\end{equation}
such that (a)-(g) below hold:
\begin{enumerate}
\item[(a)] $K_l$ and $K$ are disjoint from punctures and nodal points, and $\varphi_l$ induces an order preserving bijection $K \to K_l$.
\item[(b)] $2g_l+\nu_l\geq 3$ where $g_l$ is the genus of $\Sigma_l$ and $\nu_l = \#Z_l^++\#Z_l^- + \#K_l$.
\item[(c)] $2g+\nu\geq 3$ for every connected component $C\subset \sqcup_mS_m$, where $g$ is the genus of $C$ and $\nu$ is the total number of punctures, nodal points and points in $K$ belonging to $C$.
\item[(d)] $\varphi_l$ defines a diffeomorphism between blown up circles at $Z_l^+$ and $\Gamma^+_{k_+}$ covering an order preserving bijection $\Gamma^+_{k_+}\to Z_l^+$, and defines a diffeomorphism between blown up circles at $Z_l^-$ and $\Gamma^-_{-k_-}$ covering an order preserving bijection $\Gamma^-_{-k_-}\to Z_l^-$.
\end{enumerate}
In view of (b) we may consider $h_l$ the hyperbolic metric in the conformal class induced by $j_l$ in $\Sigma_l\setminus (Z_l^+\cup Z_l^-\cup K_l)$ seen as an open subset of $\bar\Sigma_l$. In view of (c) we may consider $h$ the hyperbolic metric in $V = \sqcup_m(S_m\setminus(\Gamma_m^+\cup\Gamma_m^-\cup D_m\cup K))$ in the conformal class of the complex structure $j$ induced by the $j_m$'s. Here we see~$V$ as an open subset of $S^{{\bf u},r}$.
Note that $\varphi_l(V) \subset \Sigma_l\setminus (Z_l^+\cup Z_l^-\cup K_l)$. One further requires
\begin{enumerate}
\item[(e)] $\varphi_l^*h_l \to h$ in $C^\infty_{\rm loc}(V)$, and $\varphi_l$ maps special circles to closed geodesics of $h_l$.
\end{enumerate}
For the maps one further asks:
\begin{itemize}
\item[(f)] $F_{\vtil_l}\circ \varphi_l$ $C^0$-converges to $F_{\bf u}$.
\item[(g)] $\forall m \ \exists \{c_{m,l}\} \subset \R$ such that $b_l \circ \varphi_l|_{S_m\setminus(\Gamma_m^+\cup\Gamma_m^-\cup D_m)}+c_{m,l}$ $C^0_{\rm loc}$-converges to~$a_m$.
\end{itemize}

\begin{remark}
The discussion above differs slightly from the discussion in~\cite{sftcomp} where one works with maps $\psi_l:S^{\mathbf{u},r}_0 \to \Sigma_l$ defined on the closed surface $S^{\mathbf{u},r}_0$ obtained from blowing $\sqcup_m S_m$ up only at nodal points and punctures between the levels, and gluing back with $r$ and $\{\Phi_m\}$. In other words, our $S^{\mathbf{u},r}$ is obtained from $S^{\mathbf{u},r}_0$ by blowing the punctures in $\Gamma^+_{k_+} \cup \Gamma^-_{-k_-}$ up. The maps $\psi_l$ are required to define order preserving bijections $\Gamma^\pm_{\pm k_\pm} \to Z^\pm_l$, $K\to K_l$. Condition (e) above is replaced by asking that $\psi_l^*h_l \to h$ in $C^\infty_{\rm loc}(V)$ and that each $\psi_l$ maps special circles to closed geodesics. The maps $\varphi_l:S^{\mathbf{u},r} \to \overline{\Sigma}_l$ can be obtained from blowing the $\psi_l$ up at $\Gamma^+_{k_+} \cup \Gamma^-_{-k_-}$. Hence, the description of SFT convergence from~\cite{sftcomp} implies the description here. It turns out that the converse also holds: one could modify the $\varphi_l$ in~\eqref{diffeos_building} on neighbourhoods $U_l$ of $\partial S^{\mathbf{u},r}$ satisfying $U_{l+1} \subset U_l$, $\cap_l U_l = \partial S^{\mathbf{u},r}$, to new maps $\varphi'_l$ which can be obtained by blowing up diffeomorphisms $\psi_l:S^{\mathbf{u},r}_0 \to \Sigma_l$ with the properties required  in~\cite{sftcomp}.
\end{remark}

\begin{remark}
Condition (e) above does not keep track of how the decoration at the punctures between levels induced by $\{\Phi_m\}$ arises in the limit; see condition~\textbf{CRS3} in~\cite[section~4.5]{sftcomp}. 
\end{remark}

\begin{remark}
Arguing as explained in Remark~4.1 from~\cite[subsection~4.5]{sftcomp}, one can see that (e) is equivalent to
\begin{itemize}
\item[(${\rm e}'$)] $\varphi_l^*j_l \to j$ in $C^\infty_{\rm loc}(V')$ where $V' = \sqcup_m(S_m\setminus(\Gamma_m^+\cup\Gamma_m^-\cup D_m))$ seen as open subset of $S^{{\bf u},r}$.
\end{itemize}
This might be more comfortable to work with because there is no reference to sets of additional marked points $K_l,K$.
\end{remark}

The SFT-compactness theorem~\cite{sftcomp} implies, as a special case, that any sequence of stable curves with genus and energy bounds has an SFT-convergent subsequence to a stable holomorphic building. This holds under the assumption that $\alpha_+,\alpha_-$ are non-degenerate up to some action $E$ which is an upper bound for the energy of the curves in the sequence.

\begin{remark}
SFT-convergence in symplectizations is described in almost the same way, except that we may simply consider $F_{\vtil_l}$ and $F_{\bf u}$ as maps into $M$ built from $M$-components.
\end{remark}

\section{Proof of Theorem~\ref{main2}}
\label{sec_main2_proof}

The main work is to prove (iii) $\Rightarrow$ (ii). 
Implication (ii) $\Rightarrow$ (i) is obvious, and the proof of (i) $\Rightarrow$ (iii) under the stated $C^\infty$-generic assumption is as explained in the proof of Theorem~\ref{thm_L_connected}. 
For convenience of the reader we spell out this argument again here. 
Since $L$ bounds a global surface of section in class $b$, obviously all periodic orbits in the complement of $L$ must have positive intersection number with $b$. 
It only remains to be proved that $\mu_{\CZ}^\Theta(\gamma)>0$ for all components $\gamma\subset L$, or equivalently $\rho^\Theta(\gamma)>0$. 
If $\rho^\Theta(\gamma)<0$ then nearby trajectories will intersect the section negatively, which is impossible. 
If $\rho^\Theta(\gamma)=0$ then, by assumption, $\gamma$ is hyperbolic and its stable/unstable manifolds contain trajectories that never hit the global surface of section, which is impossible. 
Hence $\rho^\Theta(\gamma)>0$. 

The remainder of this section contains the proof of (iii) $\Rightarrow$ (ii). 
Let us make a brief summary of the structure of the argument.

In subsection~\ref{ssec_Seifert_and_curves} we use intersection theory to derive various properties of the finite-energy curves later used in the construction of the desired global surfaces of section. 
The crucial assumption is that these curves represent the same element in $H_2(M,L)$ as a Seifert surface with the same contact topological properties of a page of a supporting open book decomposition. 
These intersection theoretic arguments do not use the genus zero assumption. 
We also obtain automatic transversality for these curves in a Fredholm theory with suitable exponential weights, and here the genus zero hypothesis plays a role. 

In subsection~\ref{ssec_existence_compactness_curves} we prove an implied existence result for the kinds of curves analyzed in subsection~\ref{ssec_Seifert_and_curves}. 
This argument is based on the assumption that corresponding curves exist for an auxiliary contact form with controlled Reeb dynamics. 
The genus zero hypothesis and (iii) in Theorem~\ref{main2} play key roles in the study of the SFT-limits of the curves involved.

Subsection~\ref{ssec_approximating_sequences} starts by explaining that the result of~\cite{wendl} can be used as input for the results from subsection~\ref{ssec_existence_compactness_curves}. 
The curves obtained are the pages of the desired open book decomposition by global surfaces of section for the Reeb flow of a non-degenerate contact form satisfying~(iii) in~Theorem~\ref{main2}. 
Since in subsection~\ref{ssec_passing_to_deg_case} degenerate contact forms will be studied, here we need to keep track of certain approximating sequences of non-degenerate contact forms.

The degenerate case is studied in subsection~\ref{ssec_passing_to_deg_case} where sequences of holomorphic curves for approximating non-degenerate contact forms are shown to have well-behaved limits. 
One difficulty is that, strictly speaking, the SFT-compactness theorem does not apply. 
Still, we explain that some more fundamental compactness arguments apply. 
Another technical point is that the curves for the approximating contact forms present an exponential decay that respects certain weights at the punctures, and these weights do not degenerate in the limiting process. 

\subsection{Seifert surfaces and holomorphic curves}
\label{ssec_Seifert_and_curves}

Let $L\subset M$ be a link transverse to $\xi$, and suppose that there exists an oriented Seifert surface $\Sigma \subset M$ for $L$ with genus~$g$. 
Orient $L$ as the boundary of $\Sigma$ and denote its components by $\gamma_1,\dots,\gamma_n$. 
Let $\alpha$ be a contact form defining $\xi$ and assume that
\begin{itemize}
\item[(A1)] $X_\alpha$ is positively tangent to $L$,
\item[(A2)] The Reeb vector field of $f\alpha$, for some $f\in C^\infty(M,(0,+\infty))$, is positively transverse to $\Sigma\setminus\partial\Sigma = \Sigma\setminus L$, and tangent to $L$.
\end{itemize}
We see each $\gamma_i$ as a simply covered closed $\alpha$-Reeb orbit with period $T_i$. 
The outward pointing normal vector of $\Sigma$ along $L = \partial \Sigma$ induces a homotopy class of $d\alpha$-symplectic trivializations of each $\xi_{\gamma_i}$. 
Denote by $\tau_\Sigma$ a collection of trivializations in these homotopy classes. 
Assume further that
\begin{itemize}
\item[(A3)] $\mu_{\CZ}^{\tau_\Sigma}(\gamma_i) \geq 1$ for every $i$,
\item[(A4)] $\alpha$ is non-degenerate up to action $\sum_{k=1}^nT_k$.
\end{itemize}

For the remainder of this subsection we fix a $d\alpha$-compatible complex structure $J:\xi\to\xi$ arbitrarily. 
In view of (A3) there exists $\delta_i<0$ in the spectral gap of the asymptotic operator along $\gamma_i$ induced by $(\alpha,J)$ between eigenvalues with winding number equal to $0$ and $1$ computed with respect to $\tau_\Sigma$. 
In particular, $\mu_{\CZ}^{\tau_\Sigma,\delta_i}(\gamma_i) = 1, \ \forall i$. 
Denote $\delta=(\delta_1,\dots\delta_n;\emptyset)$. Let $\jtil \in \J(\alpha)$ be defined as in~\eqref{formula_J_tilde} by $\alpha$ and $J$.
From now on fix a closed connected oriented genus $g$ surface $S$ with distinct points $z_1,\dots,z_n$ in $S$, and set $\Gamma=\{z_1,\dots,z_n\}$.

\begin{lemma}
\label{lemma_no_intersection_with_limits}
Let $C=[\util=(a,u),S,j,\Gamma,\emptyset] \in \M_{\jtil,g,\delta}(\gamma_1,\dots,\gamma_n;\emptyset)$, where $j$ denotes some complex structure on $S$. If $C$ and $\Sigma$ induce the same class in $H_2(M,L)$ then $C$ does not intersect $\R\times L$. Moreover,
\begin{equation}\label{selfintersection_number_special_moduli_spaces}
\util * \util = \frac{1}{2}\mu_{\CZ}(\util) - \frac{1}{2}\#\Gamma_{\rm odd} - \chi(S) + \bar\sigma(\util)
\end{equation}
and
\begin{equation}
\label{info_classical_invariants_Seifert}
\begin{aligned}
& \wind_\pi(\util) = 0 \\
& \wind_\infty(\util,z_i,\tau_{\Sigma}) = 0 \ \forall i.
\end{aligned}
\end{equation}
\end{lemma}

By the definition of $\wind_\pi$ revised in~\ref{sssec_classical_alg_inv}, its vanishing means that the projection of the curve to $M$ is transverse to the Reeb vector field. 
The vanishing of $\wind_\infty$ at the punctures with respect to $\tau_{\Sigma}$ means that the approach to the asymptotic limits is `parallel' to $\Sigma$, up to continuous deformations. 
The conclusions of Lemma~\ref{lemma_no_intersection_with_limits} serve as input for the results from~\cite{siefring} to say that the projection of the curve satisfies some of the properties that a global surface of section spanned by $L$ in class $b$ must satisfy; see Proposition~\ref{prop_Seifert_no_intersections} below.

\begin{proof}
The first, and crucial, step is to show that $\wind_\pi(\util)=0$. 
To prove this we start by noting that $\wind_\infty(\util,z_k,\tau_\Sigma)\leq0$ for every $k$ since the asymptotic eigenvalue of $\util$ at $z_k$ is less than $\delta_k$. 
Here we used the definition of $\M_{\jtil,g,\delta}$ explained in~\ref{sssec_moduli_spaces}, the monotonicity of winding numbers of eigenvalues explained in~\ref{sssec_orbits}, and our special choice of $\delta_k$. 
Let $\Xi$ be a (homotopy class of) $d\alpha$-symplectic trivialization of $u^*\xi$. 
Then $\Xi$ induces (homotopy classes of) $d\alpha$-symplectic trivializations~$\tau^\Xi_k$ of~$\xi_{\gamma_k}$, for each $k$. 
Since $C$ and $\Sigma$ give the same element in $H_2(M,L)$, the $\tau^\Xi_k$ together extend to a $d\alpha$-symplectic trivialization of $\xi|_\Sigma$. 
By~(A2) a section of $\xi|_\Sigma$ which is tangent to $\Sigma$ along $L=\partial\Sigma$ and points outward must have the same algebraic count of zeros as a section of $T\Sigma$ which is outward pointing along $L=\partial\Sigma$. In the latter case the algebraic count is precisely $2-2g-n$. Hence
\begin{equation}\label{half_c_1_seifert}
\sum_{k=1}^n \wind(\tau^\Xi_k,\tau_\Sigma) = 2-2g-n.
\end{equation}

Now we compute
\[
\begin{aligned}
\wind_\pi(\util) + 2-2g-n & = \wind_\infty(\util) \\
& = \sum_{k=1}^n \wind_\infty(\util,z_k,\tau^\Xi_k) \\
& = \sum_{k=1}^n \wind_\infty(\util,z_k,\tau_\Sigma) + \wind(\tau^\Xi_k,\tau_\Sigma) \\
&\leq 0 + 2-2g-n.
\end{aligned}
\]
Thus $\wind_\pi(\util)\leq 0 \Rightarrow \wind_\pi(\util) = 0$. In particular, $u$ is an immersion transverse to $X_\alpha$. Plugging all this information back into the above identities we get
\begin{equation}\label{crucial_asymptotic_winding}
\wind_\infty(\util,z_k,\tau_\Sigma) = 0, \ \forall k.
\end{equation}

We now follow~\cite[Theorem~4.10]{props2}. Let ${\rm int}(\cdot,L) : H_2(M,M\setminus L;\Z) \to \Z$ be the algebraic count of intersections with $L$. 
Consider the long exact sequence of the pair $(M,M\setminus L)$
\[
\xymatrix{
& \cdots \ar[r]^{} & H_2(M)  \ar[r]^-{\pi_*} & H_2(M,M\setminus L) \ar[r]^-{\hat\delta} & H_1(M\setminus L) \ar[r]^{} &  \cdots
}
\]
where $\pi_*$ is the map induced by the inclusion $(M,\emptyset) \hookrightarrow (M,M\setminus L)$ and $\hat\delta$ is the connecting homomorphism. 
Note that ${\rm im} \ \pi_* \subset \ker {\rm int}(\cdot,L)$ since $L$ is null-homologous in $M$. 
Consider conformal disks $D_k$ centered at the $z_k$, assumed to be very small. 
Denote $R = S\setminus \cup_kD_k$. By Theorem~\ref{thm_precise_asymptotics}, $u(\partial R)\subset M\setminus L$, $u^{-1}(L) \subset R\setminus\partial R$, and the homotopy classes of the asymptotic eigenvectors govern how $u$ approaches $L$. 
Moreover, combining Theorem~\ref{thm_precise_asymptotics} and~\eqref{crucial_asymptotic_winding} we have that if the $D_k$ are small enough then $u|_{\partial R}$ is a link obtained by pushing $L$ in the direction of $\tau_{\Sigma}$. 
Thus $u|_R$ can be slightly $C^0$-perturbed near $\partial R$, without creating new intersections with $L$, and be glued to $\bar \Sigma$ with a small annular neighborhood of its boundary deleted (here $\bar\Sigma$ denotes the surface $\Sigma$ with the reversed orientation) to obtain $Q \in H_2(M)$ satisfying ${\rm int}([u(R)],L) = {\rm int}(\pi_*(Q),L)$. 
It follows that ${\rm int}([u(R)],L) = 0$ because ${\rm im} \ \pi_* \subset \ker {\rm int}(\cdot,L)$.
But, since $u$ is an immersion transverse to $X_\alpha$, all intersections between $u$ and $L$ are isolated and count positively. 
Hence there are no intersections at all.

Finally, by~\eqref{crucial_asymptotic_winding} we must have $i^{\tau_\Sigma}(\util,\util)=0$ since (by Theorem~\ref{thm_precise_asymptotics}) at a puncture~$z_k$ the $M$-component $u$ of $\util$ approaches the asymptotic limit $\gamma_k$ by loops that do not wind around the center of a Martinet tube aligned with $\tau_\Sigma$. 
Asymptotic limits of $\util$ at distinct punctures are distinct and prime closed Reeb orbits. 
Hence the pairs of punctures contributing to the sum $\Sigma^+$ in~\eqref{complementary_bilinear_form} are of the form $\{(z,z)\}_{z\in\Gamma}$. 
We get $\Omega^{\tau_\Sigma}(\util,\util) = \sum_{k=1}^n \lfloor \rho^{\tau_\Sigma}(\gamma_k) \rfloor$ since there are no negative punctures, and arrive at
\begin{equation}
\util*\util = \sum_{k=1}^n \lfloor \rho^{\tau_\Sigma}(\gamma_k) \rfloor.
\end{equation}
It follows from
\[
\mu_{\CZ}^{\tau_\Sigma}(\gamma_k) = \left\{ \begin{aligned} & 2 \lfloor \rho^{\tau_\Sigma}(\gamma_k) \rfloor \qquad \quad \text{if $\mu_{\CZ}^{\tau_\Sigma}(\gamma_k)$ is even} \\ & 2 \lfloor \rho^{\tau_\Sigma}(\gamma_k) \rfloor +1 \quad \ \text{if $\mu_{\CZ}^{\tau_\Sigma}(\gamma_k)$ is odd} \end{aligned} \right.
\]
that
\begin{equation}
\sum_{k=1}^n \lfloor \rho^{\tau_\Sigma}(\gamma_k) \rfloor = \frac{1}{2} \sum_{k=1}^n \mu_{\CZ}^{\tau_\Sigma}(\gamma_k) \ -\frac{1}{2}\#\Gamma_{\rm odd}
\end{equation}
where $\Gamma_{\rm odd}$ denotes the set of punctures where the asymptotic limit has odd Conley-Zehnder index (the parity of the Conley-Zehnder index does not depend on the choice of trivialization). 
In view of the formula $\sum_k \wind(\tau^\Xi_k,\tau_\Sigma) = \chi(S) - \#\Gamma$, which was already used, we conclude that
\[
\begin{aligned}
\util*\util &= \sum_{k=1}^n \lfloor \rho^{\tau_\Sigma}(\gamma_k) \rfloor \\
&= \frac{1}{2} \sum_{k=1}^n \mu_{\CZ}^{\tau_\Sigma}(\gamma_k) \ -\frac{1}{2}\#\Gamma_{\rm odd} \\
&= \frac{1}{2} \sum_{k=1}^n \mu_{\CZ}^{\tau^\Xi_k}(\gamma_k) \ -\frac{1}{2}\#\Gamma_{\rm odd} - \chi(S) + \#\Gamma \\
&= \frac{1}{2}\mu_{\CZ}(\util) -\frac{1}{2}\#\Gamma_{\rm odd} - \chi(S) + \#\Gamma.
\end{aligned}
\]
This finishes the proof of~\eqref{selfintersection_number_special_moduli_spaces} since $\bar\sigma(\util)=\#\Gamma$; note that all asymptotic limits are prime closed Reeb orbits.
\end{proof}

Now let $\alpha'$ be another contact form defining $\xi$. Suppose that $\alpha>\alpha'$ and also that $X_{\alpha'}$ is positively tangent to $L$. Choose $\jtil'\in\J(\alpha')$. Then we may choose $h$ and construct $\Omega$ as in~\ref{sssec_gen_fin_energy_curves}, and consider $\bar J \in \J_{\Omega,L}(\jtil',\jtil)$ as in~\ref{sssec_sharing_orbits}. In particular, $\R\times L$ is a $\jbar$-holomorphic embedded surface.

\begin{remark}
\label{rmk_special_case_non-cylindrical=cylindrical}
If $C \in \M_{\jtil,g,\delta}(\gamma_1,\dots,\gamma_n;\emptyset)$ satisfies $C\subset [1,+\infty)\times M$ then we may view it as a curve in $\M_{\bar J,g,\delta}(\gamma_1,\dots,\gamma_n;\emptyset)$. Assumption~(A4) ensures that we can apply Theorem~\ref{thm_precise_asymptotics} to all punctures of curves in $\M_{\jbar,g,\delta}(\gamma_1,\dots,\gamma_n;\emptyset)$; see~\ref{sssec_energy_rmk}.
\end{remark}

\begin{lemma}
\label{lemma_somewhere_injective}
Curves in $\M_{\jbar,g,\delta}(\gamma_1,\dots,\gamma_n;\emptyset)$ or in $\M_{\jtil,g,\delta}(\gamma_1,\dots,\gamma_n;\emptyset)$ are somewhere injective.
\end{lemma}

\begin{proof}
Represent $C\in \M_{\jbar,g,\delta}(\gamma_1,\dots,\gamma_n;\emptyset)$ as $[\util,S,j,\Gamma,\emptyset]$ for some complex structure $j$ on $S$. By standard arguments one finds a closed Riemann surface $(S_0,j_0)$ equipped with a finite set $\Gamma_0\subset S_0$, a finite-energy somewhere injective $\jbar$-holomorphic map $\util_0:(S_0\setminus\Gamma_0,j_0) \to (\R\times M,\jbar)$ and a holomorphic map $\phi:(S,j)\to(S_0,j_0)$ of degree $k$ such that $\phi^{-1}(\Gamma_0)=\Gamma$ and $\util = \util_0\circ\phi$. See~\cite[section 3.2]{nelson} for a detailed account of this fact. To prove the lemma we need to show that $k=1$. All points in $\Gamma_0$ are necessarily positive punctures of $\util_0$. Choose $\zeta \in \Gamma_0$. We have $k = \sum_{z\in\phi^{-1}(\zeta)} b_\phi(z)+1$ where $b_\phi(z)$ is the branching index: in appropriate holomorphic charts centered at $z\in\phi^{-1}(\zeta)$ and at $\zeta$ the map $\phi$ is represented as $w\mapsto w^{b_\phi(z)+1}$. Since all asymptotic limits of $\util$ are simply covered we must have $b_{\phi}(z)=0$ for all $z\in\phi^{-1}(\zeta)$. Since asymptotic limits of $\util$ are mutually geometrically distinct, we obtain $\#\phi^{-1}(\zeta)=1$. Combining these facts and the formula for $k$ we conclude that $k=1$. The same argument handles the $\R$-invariant case.
\end{proof}

\begin{proposition}
\label{prop_Seifert_no_intersections}
Let $C=[\util,S,j,\Gamma,\emptyset] \in \M_{\jtil,g,\delta}(\gamma_1,\dots,\gamma_n;\emptyset)$ induce the same class in $H_2(M,L)$ as $\Sigma$. If we denote by $u$ the $M$-component of $\util$, then $u$ defines a proper embedding $S\setminus \Gamma \hookrightarrow M\setminus L$ transverse to $X_\alpha$.
\end{proposition}

\begin{proof}
Consequence of Lemma~\ref{lemma_no_intersection_with_limits}, \cite[Theorem~2.6]{siefring} and Lemma~\ref{lemma_somewhere_injective}; see Remark~\ref{rmk_special_case_non-cylindrical=cylindrical}.
\end{proof}

\begin{proposition}\label{prop_Seifert_two_curves}
The projections to $M$ of two curves in $\M_{\jtil,g,\delta}(\gamma_1,\dots,\gamma_n;\emptyset)$ are either equal or do not intersect.
\end{proposition}

\begin{proof}
Consequence of Proposition~\ref{prop_Seifert_no_intersections} and~\cite[Theorem~2.4]{siefring}.
\end{proof}

The point of the above propositions is that later we will use projections of curves in $\M_{\jtil,0,\delta}(\gamma_1,\dots,\gamma_n;\emptyset)$ in class $b$ as pages of an open book decomposition by global surfaces of section.

\begin{proposition}
\label{prop_embeddedness}
Let $C \in \M_{\jtil,g,\delta}(\gamma_1,\dots,\gamma_n;\emptyset)$ be contained in $[1,+\infty)\times M$. If $C'$ can be connected to $C$ by a smooth path of curves in $\M_{\jbar,g,\delta}(\gamma_1,\dots,\gamma_n;\emptyset)$, and $C$ induces the same class in $H_2(M,L)$ as $\Sigma$, then $C'$ is embedded. 
\end{proposition}

\begin{proof}
Consequence of Proposition~\ref{prop_constancy_int_number}, Theorem~\ref{thm_adjunction_ineq} and formula~\eqref{selfintersection_number_special_moduli_spaces}, again noting that $C'$ is somewhere injective by Lemma~\ref{lemma_somewhere_injective}.
\end{proof}

\begin{lemma}[Automatic transversality]
\label{lemma_aut_transv}
Let $\jbar\in\J_{\Omega,L}(\jtil',\jtil)$ and let $C$ be a curve in $\M_{\jtil,0,\delta}(\gamma_1,\dots,\gamma_n;\emptyset)$ contained in $[1,+\infty)\times M$ and inducing the same class in $H_2(M,L)$ as $\Sigma$. The subset of $\M_{\jbar,0,\delta}(\gamma_1,\dots,\gamma_n;\emptyset)$ consisting of curves that can be connected to $C$ in $\M_{\jbar,0,\delta}(\gamma_1,\dots,\gamma_n;\emptyset)$ is a smooth two-dimensional manifold whose differentiable structure is compatible with the $C^\infty_{\rm loc}$-topology. Moreover, a neighborhood in $\R\times M$ of every curve in this space is smoothly foliated by curves in $\M_{\jbar,0,\delta}(\gamma_1,\dots,\gamma_n;\emptyset)$.
\end{lemma}

\begin{proof}[Sketch of proof]
Fix $C'\in \M_{\jbar,0,\delta}(\gamma_1,\dots,\gamma_n;\emptyset)$ in the connected component of $C$. We can always represent it as $[\vtil,S^2,j,\Gamma,\emptyset]$ using our fixed data $S^2,\Gamma$, where $j$ is some complex structure on $S^2$. It follows from Proposition~\ref{prop_embeddedness} that $\vtil$ is an embedding. Consider a $\jbar$-invariant normal bundle $N_{\vtil}$ and set-up a Fredholm theory for sections of this normal bundle with exponential weights~$\delta$. One can use H\"older spaces with exponential weights as in~\cite{props3}. The normal bundle can be chosen to extend as the contact plane field $\xi$ along the ends of $\vtil$. A crucial fact is that $\tau_\Sigma$ extends as an $\Omega$-symplectic trivialization of $N_{\vtil}$. The Cauchy-Riemann equation gives rise to a non-linear Fredholm map on the space of sections close to zero, whose linearization at zero is a Cauchy-Riemann type linear operator $D_{\vtil}$. The associated asymptotic operators are precisely the asymptotic operators at the asymptotic limits. The weighted Fredholm index is
\[
2-n + \sum_{k=1}^n \mu_{\CZ}^{\tau_\Sigma,\delta_k}(\gamma_k) = 2-n+n = 2.
\]
This formula has important consequences. By the similarity principle, every zero of a section in $\ker D_{\vtil}$  has strictly positive contribution to the algebraic count of zeros. Moreover, since $\tau_\Sigma$ extends as a trivialization of $N_{\vtil}$, and a section in $\ker D_{\vtil}$ has an asymptotic behavior governed by eigenvalues of the asymptotic operators smaller than the $\delta_k$, the total winding number of such a section at the ends is non-positive. Combining both these facts with basic degree theory, one concludes that non-trivial sections in the kernel are nowhere vanishing. But the Fredholm index is two, so if $D_{\vtil}$ has non-trivial cokernel then its kernel will have dimension $\geq 3$, and some non-trivial linear combination of sections in $\ker D_{\vtil}$ will vanish somewhere because $N_{\vtil}$ has rank two. This contradiction shows that there is no cokernel, in other words, we have automatically the transversality needed to apply the implicit function theorem and obtain a chart of the $2$-manifold structure of the moduli space. Moreover, all curves near $C'$ necessarily show up in this chart.

Finally, the difference of two sections near the zero section, belonging to the zero locus of the non-linear Cauchy-Riemann equation again satisfies a Cauchy-Riemann type equation with the same asymptotic operators, and is subject to asymptotic analysis. Hence, the algebraic count of intersections between two such nearby sections is zero. By positivity of intersections, all intersection points contribute positively to this algebraic count. Hence, nearby curves do not intersect, and the claim made in the statement about local foliations follows.
\end{proof}

\begin{proposition}
\label{prop_cobordisms_intersection_crucial}
Let $C \in \M_{\jtil,g,\delta}(\gamma_1,\dots,\gamma_n;\emptyset)$ be contained in $[1,+\infty)\times M$. If $C'$ is in the same connected component of $\M_{\jbar,g,\delta}(\gamma_1,\dots,\gamma_n;\emptyset)$ as $C$, where we see $C$ in $\M_{\jbar,g,\delta}(\gamma_1,\dots,\gamma_n;\emptyset)$, and if $C$ induces the same class in $H_2(M,L)$ as $\Sigma$, then ${\rm int}(C',\R\times L)=0$. Moreover, if $C'=[\vtil,S,j',\Gamma,\emptyset]$ then $\wind_\infty(\vtil,z_k,\tau_\Sigma)=0$ for all $k$.
\end{proposition}

\begin{proof}
Represent $C=[\util=(a,u),S,j,\Gamma,\emptyset]$. Following the beginning of the proof of Lemma~\ref{lemma_no_intersection_with_limits} we arrive at~\eqref{crucial_asymptotic_winding}. We will argue to show  that~\eqref{crucial_asymptotic_winding} implies that
\begin{equation*}
i^{\tau_\Sigma}(\util,\R\times L) = 0.
\end{equation*}
In fact, by~\eqref{crucial_asymptotic_winding} the $M$-component $u$ approaches the asymptotic limits by loops that do not wind around the center of Martinet tubes aligned with $\tau_\Sigma$. Thus, if we displace $\util$ near the punctures in the direction of $\tau_\Sigma$ to obtain a map $\util^{\tau_\Sigma,\varepsilon,r}$ as in~\ref{sssec_intersection_numbers}, then there is no additional contribution to the intersection number with $\R\times L$ and we get $$ i^{\tau_\Sigma}(\util,\R\times L) = {\rm int}(\util^{\tau_\Sigma,\varepsilon,r},\R\times L)={\rm int}(\util,\R\times L). $$ Here $r$ is large and $\varepsilon$ is small. Proposition~\ref{prop_Seifert_no_intersections} gives ${\rm int}(\util,\R\times L)=0$, from where we then conclude that $i^{\tau_\Sigma}(\util,\R\times L) = 0$.

Now $i^{\tau_\Sigma}(\vtil,\R\times L) = i^{\tau_\Sigma}(\util,\R\times L)$ for every finite-energy map $\vtil$ representing a curve $C'$ in $\M_{\jbar,g,\delta}(\gamma_1,\dots,\gamma_n;\emptyset)$ in the same component as $C$. This follows from Proposition~\ref{prop_constancy_int_number}. Here we used that $\min a\geq 1$ allows us to view $C$ as an element of $\M_{\jbar,g,\delta}(\gamma_1,\dots,\gamma_n;\emptyset)$, and also used that $\R\times L$ is a $\jbar$-holomorphic surface since $\jbar\in\J_{\Omega,L}$. We get
\begin{equation}\label{zero_taurel_int_number}
i^{\tau_\Sigma}(\vtil,\R\times L)=0.
\end{equation}

We would like now to argue that $i^{\tau_\Sigma}(\vtil,\R\times L) = {\rm int}(\vtil,\R\times L)$ and conclude the proof. Unfortunately, one can not use an argument like the one used in the beginning of the proof of Lemma~\ref{lemma_no_intersection_with_limits} to achieve~\eqref{crucial_asymptotic_winding} because here we deal with almost complex structures that are not $\R$-invariant. We argue differently.

The defining property of the weights $\delta_i$ implies that
\begin{equation}
\label{general_ineq_wind_seifert}
\wind_\infty(\vtil,z_i,\tau_{\Sigma}) \leq 0 \ \forall i.
\end{equation}
For each $k$ consider a Martinet tube $(U_k,\Psi_k)$ around $\gamma_k$ aligned with $\tau_{\Sigma}$. It provides coordinates $(\theta,z)$ near $\gamma_k$. Let $(s,t)$ be positive holomorphic polar coordinates around $z_k$. For $s$ large enough we can write in components $\vtil(s,t)=(b(s,t),\theta(s,t),z(s,t))$. We know that $\vtil=(b,v)$ must have non-trivial asymptotic formula at each $z_k$. Hence we find $R$ such that $v([R,+\infty)\times\R/\Z)\subset U_k$ and $z(s,t)$ does not vanish for $s\geq R$. By Theorem~\ref{thm_precise_asymptotics}, perhaps after taking $R$ larger, $(s,t) \mapsto (b(s,t),t)$ is a diffeomorphism between $[R,+\infty)\times\R/\Z$ and a positive end of $\R\times\R/\Z$. Thus we find $r\gg1$ such that $(b,t) \in [r,+\infty)\times\R/\Z$ give new polar coordinates around $z_k$. This allows us to write $\vtil(b,t)=(b,\theta(b,t),z(b,t))$. In addition $z(b,t)$ does not vanish for $b\geq r$. Theorem~\ref{thm_precise_asymptotics} and the choice of $r$ together imply that $\wind_\infty(\vtil,z_k,\tau_\Sigma)$ is equal to the winding number of $t\mapsto z(r,t)$. With these choices, we take $\varepsilon>0$ small enough, and define as
in~\ref{sssec_intersection_numbers}
\[
\vtil^{k,\tau_{\Sigma},\varepsilon,r}(b,t) = (b,\theta(b,t),z(b,t)+\varepsilon\beta(b-r)) \qquad (b,t) \in [r,+\infty)\times \R/\Z
\]
where $\beta:\R\to[0,1]$ is a smooth function  equal to $0$ near $(-\infty,0]$ and equal to $1$ near $[1,+\infty)$. In view of the choices of $r$ and $\varepsilon$, the oriented intersection number $m_k \in \Z$ between $\vtil^{k,\tau_{\Sigma},\varepsilon,r}$ and the $\jbar$-complex surface $\R\times\R/\Z\times\{0\}$ is well-defined. Repeating this construction for every $k$ we arrive at what Siefring~\cite{siefring} defines as
\[
i^{\tau_\Sigma}_\infty(\vtil,\R\times L) := \sum_{k=1}^n m_k.
\]
Note that $m_k$ is equal to the algebraic count of zeros of $z(b,t)+\varepsilon\beta(b-r)$ on $[r,+\infty)\times\R/\Z$ in the coordinates described above. Standard degree theory allows us to compute $m_k$ as a difference of winding numbers. As explained before, by the choice of $r$ the winding number of $t\mapsto z(r,t)$ is equal to $\wind_\infty(\vtil,z_k,\tau_\Sigma)$ and, by construction, the winding number of $t\mapsto z(b,t)+\varepsilon\beta(b-r)$ vanishes when $b$ is large enough. The crucial identity follows
\begin{equation}
i^{\tau_\Sigma}_\infty(\vtil,\R\times L) = - \sum_{k=1}^n \wind_\infty(\vtil,z_k,\tau_{\Sigma}).
\end{equation}
Combining with~\eqref{general_ineq_wind_seifert} we get
\begin{equation*}
i^{\tau_\Sigma}_\infty(\vtil,\R\times L) \geq 0
\end{equation*}
with equality if, and only if, $\wind_\infty(\vtil,z_k,\tau_{\Sigma})=0$ for every $k$. Now~\cite[Theorem~4.2]{siefring},~\eqref{zero_taurel_int_number} and the above inequality together give
\begin{equation*}
0 = i^{\tau_\Sigma}(\vtil,\R\times L) = {\rm int}(\vtil,\R\times L) + i^{\tau_\Sigma}_\infty(\vtil,\R\times L).
\end{equation*}
Both terms on the right-hand side are non-negative: the first by positivity of intersections, the second by the previous inequality. Hence both terms vanish. The desired conclusion follows since $i^{\tau_\Sigma}_\infty(\vtil,\R\times L)=0$ if, and only if, $\wind_\infty(\vtil,z_k,\tau_{\Sigma})=0$ for all $k$.
\end{proof}

\subsection{Existence and compactness of holomorphic curves}
\label{ssec_existence_compactness_curves}

Let $L$ be a link in~$M$ with components $\gamma_1,\dots,\gamma_n$ and let $\alpha_+$ and $\alpha_-$ be defining contact forms for~$\xi$ such that $\alpha_+>\alpha_-$. Let $\Sigma$ be an oriented genus zero Seifert surface for~$L$, and orient $L$ as the boundary of $\Sigma$. Consider the following list of hypotheses.
\begin{itemize}
\item[(H1)] Both $X_{\alpha_+}$ and $X_{\alpha_-}$ are positively tangent to $L$.
\item[(H2)] $X_{\alpha_+}$ is positively transverse to $\Sigma\setminus L$.
\item[(H3)] Every $\gamma_k$ satisfies $\mu^{\tau_\Sigma}_{\CZ}\geq 1$ with respect to both $\alpha_+$ and $\alpha_-$. 
\item[(H4)] The contact forms $\alpha_+,\alpha_-$ are non-degenerate up to action $$ A = \sum_{k=1}^n \int_{\gamma_k}\alpha_+ $$ and the following hold:
\begin{itemize}
\item[($+$)] If $\gamma =(x,T)$ in $\mathcal{P}(\alpha_+)$ satisfies $T\leq A$ and $x(\R) \subset M\setminus L$ then it also satisfies ${\rm int}(\gamma,\Sigma) \neq 0$.
\item[($-$)] If $\gamma =(x,T)$ in $\mathcal{P}(\alpha_-)$ satisfies $T\leq A$ and $x(\R) \subset M\setminus L$ then it also satisfies ${\rm int}(\gamma,\Sigma) \neq 0$.
\end{itemize}
\item[(H5)] If $I$ is any proper non-empty subset of $\{1,\dots,n\}$ then the link in $M\setminus L$ obtained by pushing $\{\gamma_k\mid k\in I\}$ in the direction of $\Sigma$ defines a non-zero homology class in $H_1(M\setminus L)$.
\end{itemize}

\begin{remark}
In~(H4) we denote by ${\rm int}:H_1(M\setminus L) \otimes H_2(M,L) \to \Z$ the algebraic intersection pairing. Hypothesis (H5) is automatically satisfied when $\Sigma$ is a page of an open book decomposition with binding $L$. This is a consequence of Lemma~\ref{lemma_topological}.
\end{remark}

Under the assumption that (H1)-(H3) hold, choose almost complex structures $\jtil_\pm \in\J(\alpha_\pm)$ and $\jbar\in\J_{\Omega,L}(\jtil_-,\jtil_+)$. Fix $h$ and $\Omega$ as in~\ref{sssec_gen_fin_energy_curves}, which allows us to define the energy~\eqref{energy_non_invariant} of a $\jbar$-holomorphic map. By~(H3) we can choose numbers $\delta^\pm_1,\dots,\delta^\pm_n<0$ such that
\begin{itemize}
\item $\delta^+_k$ is in the spectral gap between eigenvalues of winding number $0$ and $1$ relative to $\tau_\Sigma$ of the asymptotic operator at $\gamma_k$ induced by $(\alpha_+,J_+)$.
\item $\delta^-_k$ is in the spectral gap between eigenvalues of winding number $0$ and $1$ relative to $\tau_\Sigma$ of the asymptotic operator at $\gamma_k$ induced by $(\alpha_-,J_-)$.
\end{itemize}
Set $\delta^\pm = (\delta^\pm_1,\dots,\delta^\pm_n;\emptyset)$. The main goal of this subsection is to prove the statement below.
It contains an `implied existence statement' for the holomorphic curves relevant to our results: they are assumed to exist for the model contact form $\alpha_+$, and the proposition states that they need also to exist for the contact form $\alpha$.

\begin{proposition}
\label{prop_main_compactness}
Assume all hypotheses (H1)-(H5). If there exists some curve in $\M_{\jtil_+,0,\delta^+}(\gamma_1,\dots,\gamma_n;\emptyset)$ that induces the same class as $\Sigma$ in $H_2(M,L)$, then there exists a curve $C_- \in \M_{\jtil_-,0,\delta^-}(\gamma_1,\dots,\gamma_n;\emptyset)$ whose projection to $M$ is in the same class as $\Sigma$ in $H_2(M,L)$. Moreover, if $\mathcal{Y}$ denotes the connected component of
~$C_-$ then $\mathcal{Y}/\R$ is compact.
\end{proposition}

Throughout this subsection we always assume (H1)-(H3) and (H5). Different parts of (H4) will be used at different moments in the arguments below. We fix an ordered set $\Gamma \subset S^2$ of distinct points $z_1,\dots,z_n$, and assume the existence of a curve
\begin{equation}
\label{curve_C_+}
C_+ = [\util_+=(a_+,u_+),S^2,j,\Gamma,\emptyset] \in \M_{\jtil_+,0,\delta^+}(\gamma_1,\dots,\gamma_n;\emptyset)
\end{equation}
inducing the same class as $\Sigma$ in $H_2(M,L)$. As explained in Remark~\ref{rmk_special_case_non-cylindrical=cylindrical}, we can translate $C_+$ up and assume that it induces also an element of $\M_{\jbar,0,\delta^+}(\gamma_1,\dots,\gamma_n;\emptyset)$.

\begin{lemma}
\label{lemma_intersection_with_trivial_cylinders_of_limit}
The following assertions hold.
\begin{itemize}
\item[(a)] Let $C_l\in \M_{\jbar,0,\delta^+}(\gamma_1,\dots,\gamma_n;\emptyset)$ be a sequence in the same connected component as $C_+$~\eqref{curve_C_+}. Assume that $C_l$ SFT-converges to the building $$ {\bf u}=\{\{\util_m = (a_m,u_m),S_m,j_m,\Gamma^+_m,\Gamma^-_m,D_m)\},\{\Phi_m\}\} \, . $$ Then for every $m$ and every connected component $Y\subset S_m$ we have that either $\util_m(Y\setminus \Gamma^+_m\cup\Gamma^-_m)\subset \R\times L$ or $\util_m(Y\setminus \Gamma^+_m\cup\Gamma^-_m) \cap \R\times L=\emptyset$.
\item[(b)] Let $C_- \in \M_{\jtil_-,0,\delta^-}(\gamma_1,\dots,\gamma_n;\emptyset)$ induce the same class as $\Sigma$ in $H_2(M,L)$, and let the sequence $C_l$ in $\M_{\jtil_-,0,\delta^-}(\gamma_1,\dots,\gamma_n;\emptyset)$, in the same connected component of $C_-$, SFT-converge to the building $$ {\bf u}=\{\{\util_m = (a_m,u_m),S_m,j_m,\Gamma^+_m,\Gamma^-_m,D_m)\},\{\Phi_m\}\} \, . $$ Then for every $m$ and every connected component $Y\subset S_m$ we have that either $\util_m(Y\setminus \Gamma^+_m\cup\Gamma^-_m)\subset \R\times L$ or $\util_m(Y\setminus \Gamma^+_m\cup\Gamma^-_m) \cap \R\times L=\emptyset$.
\end{itemize}
\end{lemma}

\begin{proof}
We only prove (a) since (b) is proved in the same way. Represent the sequence $C_l$ in (a) as $C_l = [\vtil_l = (b_l,v_l),S^2,j_l,\Gamma,\emptyset]$. Fix $m$ and a connected component $Y\subset S_m$ arbitrarily. If the restriction of $\util_m$ to $Y\setminus \Gamma^+_m\cup\Gamma^-_m$ is constant then there is nothing to prove. Assume that it is not constant and that $\util_m(Y\setminus \Gamma^+_m\cup\Gamma^-_m) \not\subset \R\times L$. Our task is to show that $\util_m(Y\setminus \Gamma^+_m\cup\Gamma^-_m) \cap \R\times L = \emptyset$, or equivalently that
\[
X = \{ z\in Y\setminus \Gamma^+_m\cup\Gamma^-_m \mid \util_m(z) \in\R\times L \}
\]
is empty. The similarity principle implies that $X$ is discrete since $\R\times L$ is an embedded $\jbar$-complex surface. By (${\rm e}'$) in~\ref{ssec_SFT_comp}
\[
\varphi_l^*j_l \to j \quad \text{in} \quad C^\infty_{\rm loc}(Y\setminus \Gamma^+_m\cup\Gamma^-_m\cup D_m)
\]
for suitable diffeomorphisms $\varphi_l$ as explained in~\eqref{diffeos_building}. Combining with (f)-(g) from~\ref{ssec_SFT_comp} and elliptic regularity, there is a sequence $c_{m,l} \in \R$ such that
\begin{equation*}
\begin{array}{cc}
(c_{m,l} \cdot \vtil_l) \circ \varphi_l|_{Y\setminus\Gamma^+_{m}\cup\Gamma^-_m\cup D_m} \to \util_m|_{Y\setminus\Gamma^+_{m}\cup\Gamma^-_m\cup D_m} & \text{in $C^\infty_{\rm loc}(Y\setminus\Gamma^+_{m}\cup\Gamma^-_m\cup D_m)$.}
\end{array}
\end{equation*}
Here $c_{m,l}\cdot\vtil_l$ denotes the map $(b_l+c_{m,l},v_l)$. Denoting by $\tau_c$ the $(\R,+)$ action on $\R\times M$, we used that the map $(c_{m,l}\cdot\vtil_l)\circ\varphi_l = \tau_{c_{m,l}} \circ \vtil_l\circ\varphi_l$ is pseudo-holomorphic with respect to $(\varphi_l^*j_l,(\tau_{c_{m,l}})_*\jbar)$ on $Y\setminus\Gamma^+_{m}\cup\Gamma^-_m\cup D_m$. Note that $\R\times L$ is an embedded $(\tau_c)_*\jbar$-holomorphic curve for every $c\in\R$ since $\jbar \in \J_{\Omega,L}(\jtil_-,\jtil_+)$, see~\ref{sssec_sharing_orbits}. Up to choice of a subsequence, we can assume that $(\tau_{c_{m,l}})_*\jbar$ converges in $C^\infty$ as $l\to\infty$ to an almost complex structure which is either equal to $(\tau_c)^*\jbar$ for some $c\in\R$, or $\jtil_-$ or $\jtil_+$. Positivity and stability of intersections implies that if $X$ intersects $Y\setminus \Gamma^+_{m}\cup\Gamma^-_m\cup D_m$ then the image of $c_{m,l}\cdot\vtil_l$ intersects $\R\times L$ when $l$ is large enough, or equivalently the map $\vtil_l$ intersects $\R\times L$. This contradicts Proposition~\ref{prop_cobordisms_intersection_crucial}, and shows that $X\subset D_m$.

Let $z\in X \subset D_m$ and $K\subset Y\setminus \Gamma^+_{m}\cup\Gamma^-_m$ be a conformal disk with $z$ in its interior. If $K$ is small enough then $\{z\} = K \cap D_m$, and in particular we have $K \cap \util_m^{-1}(\R\times L) = \{z\}$. The $M$-component $u_m$ of $\util_m$ maps $K$ into a small tubular neighbourhood of $L$ given {\it a priori}. Positivity of intersections implies that $\util_m|_K$ has positive algebraic intersection count with $\R\times L$. An application of Lemma~\ref{lemma_appendix_intersection} tells us that ${\rm int}(u_m(\partial K),\Sigma)>0$. We see that if $l$ is large enough then $\beta_l = \varphi_l(\partial K)$ is an embedded loop in $S^2\setminus\Gamma$ such that $v_l(\beta_l)$ is uniformly close to $u_m(\partial K)$. We obtain ${\rm int}(v_l(\beta_l),\Sigma)>0$ when $l$ is large enough. We can use Proposition~\ref{prop_cobordisms_intersection_crucial} to get $0<{\rm int}(u_+(\beta_l),\Sigma) = {\rm int}(u_+(\beta_l),u_+(S^2\setminus\Gamma))$, where $u_+$ is the $M$-component of the curve $C_+$. This is in contradiction to Proposition~\ref{prop_Seifert_no_intersections}. Thus $X=\emptyset$ as desired.
\end{proof}

\begin{lemma}\label{lemma_compactness_first}
The following assertions are true.
\begin{itemize}
\item[(a)] Assume that (H$4_+$) holds. Let $C_l = [\vtil_l=(b_l,v_l),S^2,j_l,\Gamma,\emptyset]$ be a sequence in $\M_{\jbar,0,\delta^+}(\gamma_1,\dots,\gamma_n;\emptyset)$ in the same connected component as the curve $C_+$~\eqref{curve_C_+}, such that the sequence $\{\inf b_l(S^2\setminus\Gamma)\}_l$ is bounded. Then some subsequence of $\{C_l\}$ SFT-converges in $\M_{\jbar,0,\delta^+}(\gamma_1,\dots,\gamma_n;\emptyset)$.
\item[(b)] Assume that (H$4_-$) holds. Let $C_- \in \M_{\jtil_-,0,\delta^-}(\gamma_1,\dots,\gamma_n;\emptyset)$ be a curve that induces the same class as $\Sigma$ in $H_2(M,L)$. If $\mathcal{Y}$ denotes the connected component of $\M_{\jtil_-,0,\delta^-}(\gamma_1,\dots,\gamma_n;\emptyset)$ containing $C_-$ then $\mathcal{Y}/\R$ is compact.
\end{itemize}
\end{lemma}

\begin{remark}
The metrizable topology with respect to which we claim that $\mathcal{Y}/\R$ is compact is the one induced by $C^\infty_{\rm loc}$-convergence. More precisely, for every sequence $$ \left\{ C_l = [\vtil_l=(b_l,v_l),S_l,j_l,\{z^+_{1,l},\dots,z^+_{n,l}\},\emptyset] \right\}_{l\in\N} \subset \mathcal{Y} $$ one finds $C = [\vtil=(b,v),S,j,\{\hat z_1,\dots,\hat z_n\},\emptyset] \in \mathcal{Y}$, a subsequence $C_{l_m}$, and sequences $d_m\in\R$, $\varphi_m:S\to S_{l_m}$ such that: $\varphi_m$ is a diffeomorphism, $\varphi_m(\hat z_i) = z^+_{i,{l_m}}$ for all $i,m$ and
\[
\varphi_m^*j_{l_m} \to j, \qquad b_{l_m} \circ \varphi_m + d_m \to b, \qquad v_{l_m} \circ \varphi_m \to v
\]
in $C^\infty_{\rm loc}(S\setminus\{\hat z_1,\dots,\hat z_n\})$.
\end{remark}

\begin{proof}[Proof of Lemma~\ref{lemma_compactness_first}]
We first spell out the proof of (a) in full detail, and then sketch the proof of (b) since  arguments are essentially the same.

Fix representatives $(\vtil_l=(b_l,v_l),S^2,j_l,\Gamma,\emptyset)$ of $C_l$, where $\Gamma$ is the ordered set fixed in the beginning of this subsection. By SFT-compactness we may assume, up to a subsequence, that $\vtil_l$ SFT-converges to a building ${\bf u}=(\{\util_m=(a_m,u_m)\},\{\Phi_m\})$ of height $k_-|1|k_+$. Borrowing notation from the above discussion, the domain of the maps $F_{\vtil_l}$ are $n$-holed two-spheres with interior equal to $S^2\setminus \Gamma$, and we get diffeomorphisms $\varphi_l$ together with finite ordered sets $K_l,K$ such that (a)-(g) hold. Since $\sup_l |\inf_{S^2\setminus\Gamma}b_l|$ is bounded we know that $k_-=0$.

Our first task is to prove that $k_+=0$. Arguing indirectly, suppose that $k_+ \geq 1$. Consider an arbitrary level $m$ different from the top level: $m < k_+$. We claim that the asymptotic limit $\gamma$ at every positive puncture $z_* \in \Gamma^+_m$ of $\util_m$ is contained in~$L$. The argument is indirect. Suppose that such an asymptotic limit $\gamma=(x,T)$ is a closed Reeb orbit in $M\setminus L$. We know that $\gamma$ is a closed $\alpha_+$-Reeb orbit since $k_-=0 \Rightarrow m\geq0$. Hypothesis (H$4_+$) implies that ${\rm int}(\gamma,\Sigma) \neq 0$. Using (e)-(f) we obtain $l$ large and an embedded loop $\beta_l\subset S^2\setminus\Gamma$ such that $v_l(\beta_l)$ is $C^0$-close to the loop $x(T\cdot):\R/\Z\to M\setminus L$, in particular, ${\rm int}(v_l(\beta_l),\Sigma)\neq 0$. By assumption, all $C_l$ are in the same connected component of the curve $C_+$~\eqref{curve_C_+}. It follows from Proposition~\ref{prop_cobordisms_intersection_crucial} that $[{u_+}(\beta_l)] = [{v_l}(\beta_l)]$ in $H_1(M\setminus L)$ since the curves $C_+$ and $C_l$ can be homotoped one to the other through holomorphic curves that do not touch $\R\times L$. Hence ${\rm int}({u_+}(\beta_l),\Sigma)\neq 0$. But, by assumption, $\Sigma$ and $u_+(S^2\setminus \Gamma)$ induce the same element in $H_2(M,L)$. It follows that $u_+$ is not an embedding into $M\setminus L$ because ${\rm int}(u_+(\beta_l),u_+(S^2\setminus\Gamma)) \neq 0$. This contradiction to Proposition~\ref{prop_Seifert_no_intersections} shows that $\gamma \subset L$. Consider the top level which, by the contradiction assumption $k_+\geq1$, is a finite-energy $\jtil_+$-holomorphic nodal curve
\[
\util_{k_+} =(a_{k_+},u_{k_+}) : S_{k_+} \setminus \Gamma^+_{k_+}\cup\Gamma^-_{k_+} \to \R\times M.
\]
Let $Y$ be a connected component of $S_{k_+}$ such that $\util_{k_+}|_Y$ is not a constant map and is not a trivial cylinder over some periodic orbit. Such a component $Y$ exists, as one easily shows by combining Lemma~\ref{lemma_constant_components} with the fact that asymptotic limits at positive punctures of $\util_{k_+}$ are prime and mutually geometrically distinct closed Reeb orbits. Let $u_{k_+}$ denote the $M$-component of $\util_{k_+}$. At each puncture in $\Gamma^+_{k_+}$ the curve $\util_{k_+}$ is asymptotic to a prime closed $\alpha_+$-Reeb orbit (one of the $\gamma_k$) and different punctures yield geometrically different asymptotic limits. It follows that
\[
\int_{Y \setminus (\Gamma^+_{k_+} \cup \Gamma^-_{k_+})} u_{k_+}^*d\alpha_+ > 0
\]
in other words, $\util_{22k_+}|_Y$ is not a (possibly branched) cover of some trivial cylinder. In particular, $\util_{k_+}$ has non-trivial asymptotic formula at all punctures in $Y$. Assume that there exists a negative puncture $z_* \in Y\cap\Gamma^-_{k_+}$. Then, by what we proved before, there exists $i\in\{1,\dots,n\}$ such that the asymptotic limit $\gamma_*$ of $\util_{k_+}$ at $z_*$ is of the form $\gamma_* = \gamma_i^{m_i}$, $m_i\in\N$. (H3) implies that $\mu_{\CZ}^{\tau_\Sigma}(\gamma_i^{m_i}) \geq \mu_{\CZ}^{\tau_\Sigma}(\gamma_i) \geq 1$. Now we use Theorem~\ref{thm_precise_asymptotics} to conclude that $\wind_\infty(\util_{k_+},z_*,\tau_{\Sigma}) \geq 1$ and that if $\beta \subset Y\setminus \Gamma^+_{k_+} \cup \Gamma^-_{k_+}$ is a small loop winding once around $z_*$ then ${\rm int}(u_{k_+}(\beta),\Sigma)\neq 0$. The loop $\beta$ corresponds to an embedded loop in the interior of $S^{{\bf u},r}$. By (f) we obtain for $l$ large, using diffeomorphisms $\varphi_l$ as in~\eqref{diffeos_building}, an embedded loop $\beta_l = \varphi_l(\beta)$ in $S^2\setminus\Gamma$ such that ${\rm int}(v_l(\beta_l),\Sigma) \neq 0$. Arguing as before, combining Proposition~\ref{prop_cobordisms_intersection_crucial} with Proposition~\ref{prop_Seifert_no_intersections}, we arrive at a contradiction. This shows that $\util_{k_+}$ has no negative punctures in $Y$. We have proved that $\util_{k_+}|_Y$ has positive $d\alpha_+$-area and no negative punctures. Obviously, it has at least one positive puncture.

We claim that
\begin{equation}\label{asymptotic_behavior_wtil}
\wind_\infty(\util_{k_+},\zeta,\tau_\Sigma)=0 \qquad \forall \zeta \in Y \cap \Gamma^+_{k_+}.
\end{equation}
If not then, by Theorem~\ref{thm_precise_asymptotics}, $u_{k_+}$ will map a small loop $\beta_\zeta$ winding once around $\zeta$ to a loop satisfying ${\rm int}(u_{k_+}(\beta_\zeta),\Sigma)\neq 0$. By the construction of $S^{{\bf u},r}$, $\beta_\zeta$ may be seen as a loop in $S^{{\bf u},r}$, and in view of (e)-(f) we find embedded loops $\beta_l\subset S^2\setminus\Gamma$ such that ${\rm int}(v_l(\beta_l),\Sigma)\neq 0$. Arguing as in the beginning of this proof, Proposition~\ref{prop_cobordisms_intersection_crucial} implies that $[{u_+}(\beta_l)] = [{v_l}(\beta_l)]$ in $H_1(M\setminus L)$, and by assumption $\Sigma = u_+(S^2\setminus \Gamma)$ in $H_2(M,L)$. We get ${\rm int}(u_+(\beta_l),u_+(S^2\setminus \Gamma)) \neq 0$, in contradiction to Proposition~\ref{prop_Seifert_no_intersections} and the important identities~\eqref{asymptotic_behavior_wtil} are proved. We will now use~\eqref{asymptotic_behavior_wtil} to show that $Y = S_{k_+}$. In fact, suppose not. Lemma~\ref{lemma_intersection_with_trivial_cylinders_of_limit} implies that $\util_{k_+}(Y\setminus \Gamma^+_{k_+}) \cap \R\times L = \emptyset$, or equivalently $$ u_{k_+}(Y\setminus \Gamma^+_{k_+}) \cap L = \emptyset. $$ But~\eqref{asymptotic_behavior_wtil} and Theorem~\ref{thm_precise_asymptotics} tell us that $u_{k_+}$ approaches the asymptotic limits at punctures in $Y\cap\Gamma^+_{k_+}$ along $\Sigma$. Hypothesis (H5) implies that $Y \cap \Gamma^+_{k_+} = \Gamma^+_{k_+}$. Every connected component $Y'\neq Y$ of $S_{k_+}$ contains no positive puncture and, consequently, $\util_{k_+}|_{Y'}$ is constant due to exactness of symplectizations. If such $Y'$ exists then Lemma~\ref{lemma_constant_components} provides two distinct components where $\util_{k_+}$ is non-constant, this is absurd. Hence $Y=S_{k_+}$, as desired. It follows that there are no negative punctures at the top level, contradicting $k_+>0$.

We are done proving that $k_+=0$. Combining with $k_-=0$, we conclude that ${\bf u}$ has only one level $\util_0$, which must be a connected nodal curve. Let us argue that $S_0$ is connected. In fact, let $Y$ be a connected component of $S_0$ such that $\util_0|_Y$ is not constant. Arguing as above one first shows that~\eqref{asymptotic_behavior_wtil} holds, and then uses (H5) to conclude that $Y \cap \Gamma_0^+  = \Gamma_0^+$. Here we heavily relied on the fact that the image of $\util_0|_Y$ is not contained in $\R\times L$: by the similarity principle, if the image of $\util_0|_Y$ would be contained in $\R\times L$ then there would be negative punctures, in contradiction to $k_-=0$. Then $\util_0$ is constant on all connected components of $S_0$ different from $Y$. If there is a connected component $Y'\subset S_0$ different from $Y$ then Lemma~\ref{lemma_constant_components} provides two distinct components where $\util_0$ is non-constant, which is absurd. Hence $S_0=Y$ is connected, and there are no nodes since the total genus is zero.

Summarizing, the limiting holomorphic building has one level $\util_0$, no nodal pairs, no negative punctures, and $\wind_\infty(\util_0,z,\tau_{\Sigma})=0$ for all $z\in\Gamma_0^+$. In other words, it is a curve in $\M_{\jbar,0,\delta^+}(\gamma_1,\dots,\gamma_n;\emptyset)$. The proof of (a) is complete.

The proof of (b) is essentially the same. In fact, consider a sequence $\R + C_l$ in $\mathcal{Y}/\R$. We may represent it by curves $C_l=[\vtil_l=(b_l,v_l),S^2,j_l,\Gamma,\emptyset]$ satisfying $\min_{S^2\setminus\Gamma}b_l=0$ since we are allowed to translate in the $\R$-direction. As in the proof of (a), some subsequence converges to a limiting building ${\bf u}$ in the sense of SFT. This time assumption (H$4_-$) plays the exact same role that (H$4_+$) played in the proof of (a), and all arguments go through to conclude that ${\bf u}$ has one level which is the desired curve in $\M_{\jtil_-,0,\delta^-}(\gamma_1,\dots,\gamma_n;\emptyset)$ representing the desired limit of the subsequence in $\mathcal{Y}/\R$.
\end{proof}

\begin{lemma}\label{lemma_implied_existence_step1}
If (H$4_+$) holds then there exists a sequence
\begin{equation}\label{descending_sequence}
C_l = [\vtil_l=(b_l,v_l),S^2,j_l,\Gamma,\emptyset] \in \M_{\jbar,0,\delta^+}(\gamma_1,\dots,\gamma_n;\emptyset)
\end{equation}
in the same component as $C_+$ such that $\min b_l \to -\infty$ as $l\to\infty$.
\end{lemma}

\begin{proof}
Let $\mathcal{Y}$ be the connected component of $\M_{\jbar,0,\delta^+}(\gamma_1,\dots,\gamma_n;\emptyset)$ containing the curve $C_+$. Consider the set $A$ of numbers $a\in\R$ such that all curves in $\mathcal{Y}$ are contained in $[a,+\infty)\times M$. We wish to show that $A=\emptyset$. Assume, by contradiction, that $A\neq\emptyset$ and consider $\underline a = \sup A$. Then $\underline a<+\infty$ since $\mathcal Y\neq\emptyset$, every curve in $\mathcal Y$ is contained in $[\underline a,+\infty)\times M$ and there exists a sequence $$ C_l = [\vtil_l=(b_l,v_l),S^2,j_l,\Gamma,\emptyset] $$ such that $\inf b_l \to \underline a$ as $l\to\infty$. By (a) in Lemma~\ref{lemma_compactness_first} we may assume that $C_l$ SFT-converges to some $C_\infty \in \M_{\jbar,0,\delta^+}(\gamma_1,\dots,\gamma_n;\emptyset)$ represented by a finite-energy map $\vtil_\infty = (b_\infty,\vtil_\infty)$ satisfying $\inf b_\infty = \underline a$. Now Lemma~\ref{lemma_aut_transv} allows us to find curves in~$\mathcal{Y}$ whose $\R$-components reach below $\underline a$, which is absurd.
\end{proof}

\begin{lemma}
\label{lemma_implied_existence_step2}
Assume that (H4) holds. If there exists a sequence $C_l$ as in~\eqref{descending_sequence} in the same component of $C_+$ satisfying $\min b_l\to-\infty$ as $l\to\infty$, then there exists a curve $C_- \in \M_{\jtil_-,0,\delta^-}(\gamma_1,\dots,\gamma_n;\emptyset)$ inducing the same element in $H_2(M,L)$ as~$\Sigma$. Moreover, if $\mathcal{Y}$ denotes the connected component containing $C_-$ then $\mathcal{Y}/\R$ is compact.
\end{lemma}

\begin{proof}
This proof has many steps in common with the proof of Lemma~\ref{lemma_compactness_first}. Fix representatives
$$
(\vtil_l=(b_l,v_l),S^2,j_l,\Gamma,\emptyset)
$$
of a sequence of curves $C_l$ as in the statement, where $\Gamma \subset S^2$ is independent of~$l$. By SFT-compactness we may assume that, up to a subsequence, $\vtil_l$ SFT-converges to a building ${\bf u}=(\{\util_m=(a_m,u_m)\},\{\Phi_m\})$ of height $k_-|1|k_+$. We borrow all the notation used in the description of holomorphic buildings explained immediately before Lemma~\ref{lemma_compactness_first}. The domain of the maps $F_{\vtil_l}$ are $n$-holed two-spheres with interior equal to $S^2\setminus\Gamma$. For each~$l$ large enough there is a diffeomorphism~$\varphi_l$ between the domain $S^{{\bf u},r}$ of $F_{\bf u}$ and the domain of $F_{\vtil_l}$, and finite ordered sets $K_l,K$ such that (a)-(g) hold. The first important remark is that $k_-\geq 1$ follows from the assumption $\min_lb_l \to -\infty$.

\medskip

\noindent {\it Claim I. All asymptotic limits of all levels $\util_m$ are contained in $L$.}

\medskip

\noindent {\it Proof of Claim I.} We argue indirectly and suppose that the claim does not hold. Then some asymptotic limit of some level is a closed Reeb orbit $\gamma=(x,T)$ contained in $M\setminus L$, of $\alpha_+$ or of $\alpha_-$ depending on the level. Using the diffeomorphisms $\varphi_l$ and (f) we find for $l$ large enough an embedded circle $\beta_l \subset S^2\setminus\Gamma$ such that $v_l(\beta_l)$ is $C^0$-close to the loop $t\in\R/\Z \mapsto x(Tt)$. Fix~$l$ large. Hypothesis (H4) implies that ${\rm int}(v_l(\beta_l),\Sigma)\neq0$. Proposition~\ref{prop_cobordisms_intersection_crucial} and positivity of intersections imply that $C_l$ can be homotoped to $C_+$ through curves that do not touch $\R\times L$. Hence ${\rm int}(u_+(\beta_l),\Sigma)={\rm int}(u_+(\beta_l),u_+(S^2\setminus\Gamma))\neq0$. It follows that $u_+$ is not an embedding into $M\setminus L$, contradicting Proposition~\ref{prop_Seifert_no_intersections}. //

\medskip

\noindent {\it Claim II. If $-k_-\leq m\leq k_+$ and $Y\subset S_m$ is a connected component such that $\wtil = \util_m|_Y$ is non-constant and its image is not contained in $\R\times L$, then $\wtil$ does not intersect $\R\times L$ and has no negative punctures.}

\medskip

\noindent {\it Proof of Claim II.} To see this first note that, under these assumptions, $\wtil$ has a non-trivial asymptotic formula at every puncture. From claim I we already know that if $\wtil$ has a negative puncture $z_*$ then its asymptotic limit at $z_*$ is $\gamma_i^{m_i}$, for some $i\in\{1,\dots,n\}$ and some $m_i\geq 1$. Hypothesis (H3) implies that $\mu_{\CZ}^{\tau_\Sigma}(\gamma_i^{m_i}) \geq \mu_{\CZ}^{\tau_\Sigma}(\gamma_i) \geq 1$. By Theorem~\ref{thm_precise_asymptotics} we get that $\wind_\infty(\wtil,z_*,\tau_{\Sigma}) \geq 1$ and that a small embedded loop $\beta_*$ winding once around $z_*$ is mapped by the $M$-component $w$ of $\wtil$ to a loop satisfying ${\rm int}(w(\beta_*),\Sigma) \neq 0$. Using the diffeomorphisms $\varphi_l$ and (f) we find for $l$ large an embedded loop $\beta_l$ in $S^2\setminus \Gamma$ such that ${\rm int}(v_l(\beta_l),\Sigma) \neq 0$. Now $C_l$ is homotopic to $C_+$~\eqref{curve_C_+} through curves that do not touch $\R\times L$. This follows from Proposition~\ref{prop_cobordisms_intersection_crucial} and positivity of intersections. For~$l$ large enough we find that ${\rm int}(u_+(\beta_l),u_+(S^2\setminus \Gamma)) = {\rm int}(v_l(\beta_l),\Sigma) \neq 0$. This is in contradiction to Proposition~\ref{prop_Seifert_no_intersections}. We have proved that $\wtil$ has no negative punctures, that is, $\Gamma^-_{m} \cap Y = \emptyset$. Lemma~\ref{lemma_intersection_with_trivial_cylinders_of_limit} implies that the image of $\wtil$ does not intersect $\R\times L$. // 

\medskip

\noindent {\it Claim III. $k_+=0$ and the top level $\util_{k_+}=\util_0$ is precisely $\R\times L$.}

\medskip

\noindent {\it Proof of Claim III.} Suppose first, by contradiction, that there exists a connected component $Y$ of $S_{k_+}\setminus\Gamma^+_{k_+}\cup\Gamma^-_{k_+}$ such that $\util_{k_+}|_Y$ is not constant and $\util_{k_+}(Y)$ is not contained in $\R\times L$. We have that $Y \cap \Gamma^+_{k_+} \neq \emptyset$ by the exact nature of symplectizations, and by exactness of the taming symplectic form $\Omega$. The asymptotic limit of $\util_{k_+}$ at $z_* \in Y \cap \Gamma^+_{k_+}$ is a prime closed Reeb orbit given by one of the components of~$L$ and, by our contradiction assumption, $\util_{k_+}$ has a non-trivial asymptotic formula at~$z_*$. If $\wind_\infty(\util_{k_+},z_*,\tau_\Sigma)\neq0$ then, in view of Theorem~\ref{thm_precise_asymptotics}, a small embedded loop $\beta_*\subset Y\setminus\Gamma^+_{k_+}$ winding once around $z_*$ is mapped by the $M$-component $u_{k_+}$ of $\util_{k_+}$ to a loop satisfying ${\rm int}(u_{k_+}(\beta_*),\Sigma) \neq 0$. Arguing as in claim I, we find for $l$ large an embedded loop $\beta_l \subset S^2\setminus\Gamma$ such that ${\rm int}(v_l(\beta_l),\Sigma)\neq0$. Using Proposition~\ref{prop_cobordisms_intersection_crucial}, this leads to ${\rm int}(u_+(\beta_l),u_+(S^2\setminus\Gamma))\neq0$ in contradiction to Proposition~\ref{prop_Seifert_no_intersections}. Thus $\wind_\infty(\util_{k_+},z_*,\tau_\Sigma)=0$ for every $z_* \in Y \cap\Gamma^+_{k_+}$. By claim II, $Y\cap \Gamma^-_{k_+} = \emptyset$ and $\util_{k_+}$ maps $Y\setminus\Gamma^+_{k_+}$ to the complement of $\R\times L$. Hypothesis (H5) implies that $\Gamma^+_{k_+} \subset Y$. It follows that $\util_{k_+}$ is constant on every other connected component of $S_{k_+}$, in particular, $\util_{k_+}$ has no negative punctures in contradiction to $k_->0$. We are done showing that for every connected component $Y\subset S_{k_+}$ the map $\util_{k_+}|_Y$ is either constant or its image is contained in $\R\times L$. In the latter case $\util_{k_+}|_Y$ must be a trivial cylinder over some $\gamma_k$; it can not be a (possibly branched) cover of a trivial cylinder since all asymptotic limits at the positive punctures of the top level are mutually geometrically distinct and simply covered. If there is a component of $S_{k_+}$ where $\util_{k_+}$ is constant then Lemma~\ref{lemma_constant_components} provides an intersection between two geometrically distinct closed Reeb orbit, absurd. Hence all components are trivial cylinders over some orbit in $L$. Then $k_+=0$, for if $k_+>0$ then there would be a contradiction to stability of the building since the top level would be a symplectization level. The desired conclusion follows. // 

\medskip

\noindent {\it Claim IV. $k_-=-1$ and $\util_{-1}$ provides an element of $\M_{\jtil_-,0,\delta^-}(\gamma_1,\dots,\gamma_n;\emptyset)$.}

\medskip

\noindent {\it Proof of Claim IV.} Assume by contradiction that $\util_{-1}$ consists only of constant maps and of (possibly branched) covers of trivial cylinders. 
In the latter case the corresponding component would have to be a trivial cylinder over some component of $L$ since, by claim III, $\util_{-1}$ has $n$ positive punctures and its asymptotic limits are exactly the $n$ distinct components $\gamma_1,\dots,\gamma_n$ of $L$, each simply covered and seen as a closed $\alpha_-$-Reeb orbit. 
If on some component of $S_{-1}$ minus the punctures the map $\util_{-1}$ is constant then Lemma~\ref{lemma_constant_components} will provide an intersection between geometrically  distinct closed Reeb orbits, and this is impossible. As a result,~$\util_{-1}$ would consist precisely of $n$ trivial cylinders, and this is in contradiction to stability of the building. We have showed that on some connected component $Y\subset S_{-1}$ the map $\util_{-1}$ is non-constant on $Y\setminus\Gamma^+_{-1}\cup\Gamma^-_{-1}$ and $\util_{-1}(Y\setminus\Gamma^+_{-1}\cup\Gamma^-_{-1}) \not\subset \R\times L$. We can apply claim II to conclude that there are no negative punctures on $Y$ and that $\util_{-1}(Y\setminus\Gamma^+_{-1}) \cap (\R\times L) = \emptyset$. Arguing as in the proof of claim III we get $\wind_\infty(\util_{-1},z_*,\tau_\Sigma)=0$ for every $z_* \in Y \cap \Gamma^+_{-1}$. Hypothesis (H5) tells us that $\Gamma^+_{-1} \subset Y$, in particular $\util_{-1}$ is constant on all other connected components of its domain. If there are nodes then, by Lemma~\ref{lemma_constant_components}, there are at least two distinct components where $\util_{-1}$ is non-constant, and this is absurd. Hence there are no nodes, $S_{-1}=Y$ and $\util_{-1}$ provides the desired curve in $\M_{\jtil_-,0,\delta^-}(\gamma_1,\dots,\gamma_n;\emptyset)$. //

\medskip

Claims I, II, III and IV show the existence of $C_-$. Compactness of $\mathcal{Y}/\R$ is a direct consequence of (b) in Lemma~\ref{lemma_compactness_first}.
\end{proof}

Proposition~\ref{prop_main_compactness} is a direct consequence of Lemmas~\ref{lemma_implied_existence_step1} and~\ref{lemma_implied_existence_step2}.

\subsection{Approximating sequences of contact forms}
\label{ssec_approximating_sequences}

\begin{proposition}
\label{prop_main_existence_non_deg}
Let $\alpha$ be a contact form on $M$, and let $L=\gamma_1\cup\dots\cup\gamma_n$ be a link consisting of periodic Reeb orbits of $\alpha$. Assume that $L$ binds a planar open book decomposition $\Theta$ that supports~$\xi = \ker\alpha$, and assume with no loss of generality that $\alpha$ and $\Theta$ induce the same orientation on each $\gamma_i$.  
Suppose further that
\begin{itemize}
\item[(a)] For every $i$, $\mu_{\CZ}^\Theta(\gamma_i)\geq1$ as a prime periodic Reeb orbit of $\alpha$.
\end{itemize}
holds. 
Then there exists a constant $$ A > \sum_{i=1}^n \int_{\gamma_i} \alpha $$ depending only on $(\alpha,L,\Theta)$, with the following significance. Fix any $d\alpha$-compatible complex structure $J:\xi\to\xi$. For each $i$, choose any $\delta_i<0$ in the spectral gap of the asymptotic operator induced by $(\alpha,J)$ on sections of $\xi$ along $\gamma_i$, between eigenvalues of winding number $0$ and $1$ with respect to a Seifert framing induced by the pages of~$\Theta$. Denote $\delta=(\delta_1,\dots,\delta_n;\emptyset)$. If 
\begin{itemize}
\item[(b)] Every $\gamma'=(x',T') \in \mathcal{P}(\alpha)$ such that $x'(\R) \subset M\setminus L$ and $T'\leq A$ has non-zero algebraic intersection number with pages of $\Theta$.
\end{itemize}
holds, then for any sequence $f_k\in \mathcal{F}_L$ satisfying
\begin{itemize}
\item $f_k\to 1$ in $C^\infty$, $f_k|_L \equiv 1$ for all $k$,
\item $f_k\alpha$ is non-degenerate for all $k$,
\end{itemize}
and $\jtil_k$ defined as in~\eqref{formula_J_tilde} using $f_k\alpha$ and $J$, there exists $k_0$ such that for every $k\geq k_0$ the link $L$ binds a planar open book decomposition whose pages are global surfaces of section for the Reeb flow of $f_k\alpha$, all pages represent the same class in $H_2(M,L)$ as the pages of $\Theta$ and are projections of curves in $\M_{\jtil_k,0,\delta}(\gamma_1,\dots,\gamma_n;\emptyset)$.
\end{proposition}

Before embarking on the proof of Proposition~\ref{prop_main_existence_non_deg} we state the following corollary.

\begin{corollary}
\label{prop_(iii)implies(ii)_nondeg}
If $(M,\xi)$, $L$, $\alpha$, $\Theta$ and $b$ are as in Theorem~\ref{main2} then the implication (iii) $\Rightarrow$ (ii) from Theorem~\ref{main2} holds when $\alpha$ is non-degenerate.
\end{corollary}

\begin{proof}
Apply Proposition~\ref{prop_main_existence_non_deg} to the constant sequence $f_k\equiv1 \ \forall k$.
\end{proof}

Let us start with the proof of Proposition~\ref{prop_main_existence_non_deg}. Since $\Theta$ supports $\xi$, we know that the orientation of $L$ induced by $\alpha$ either simultaneously agrees or simultaneously disagrees with the orientation induced by $\Theta$ on every component $\gamma\subset L$. There is no loss of generality to assume that these orientations agree on all $\gamma$. From now on, we give $L$ this orientation and order its connected components $\gamma_1,\dots,\gamma_n$ in an arbitrary manner.

By the results from~\cite{wendl}, there exists at least one non-degenerate contact form $\alpha_+$ such that $\alpha_+>\alpha$, $X_{\alpha_+}$ is positively tangent to~$L$, $\mu_{\CZ}^{\Theta}(\gamma_i)= 1$ as closed Reeb orbits of $\alpha_+$, and every $\gamma \in \mathcal{P}(\alpha_+)$ contained in $M\setminus L$ has positive algebraic intersection number with the pages of $\Theta$. Moreover, setting $$ \delta^+=(0,\dots,0;\emptyset) $$ and choosing a suitable $d\alpha_+$-compatible complex structure $J_+:\xi\to\xi$, then
\begin{equation}
\M_{\jtil_+,0,\delta^+}(\gamma_1,\dots,\gamma_n;\emptyset) \neq \emptyset \, .
\end{equation}
Here $\jtil_+$ is the $\R$-invariant almost complex structure on $\R\times M$ defined as in~\eqref{formula_J_tilde} using $\alpha_+$ and $J_+$. This moduli space is diffeomorphic to an open cylinder and its elements project to $M$ as a circle family of embedded surfaces in $M\setminus L$ whose closures are $C^1$-close to the pages of $\Theta$, all of which are global surfaces of section for the Reeb flow of $\alpha_+$. Finally we can define
\begin{equation*}
A = \sum_{i=1}^n \int_{\gamma_i}\alpha_+.
\end{equation*}

Let us arbitrarily fix a curve
\begin{equation}
C_+ \in \M_{\jtil_+,0,\delta^+}(\gamma_1,\dots,\gamma_n;\emptyset) \, .
\end{equation}
Let $\Sigma$ be the Seifert surface for $L$ obtained as the closure of the projection of $C_+$ onto $M$. Then $\Sigma$ induces the same element in $H_2(M,L)$ as any page of $\Theta$. We shall denote by $\tau_{\Sigma}$ a collection of symplectic (with respect to $d\alpha_+$ or equivalently to $d\alpha$) trivializations of the bundles $\xi_{\gamma_i}$ aligned with the normal of $\Sigma$. These trivializations are homotopic to trivializations aligned with the normal of any page of $\Theta$.

Fix any $d\alpha$-compatible complex structure $J:\xi\to\xi$. Recall that $J$ is also compatible with $d(g\alpha)$ for every $g\in C^\infty(M,(0,+\infty))$. Consider the asymptotic operators at the $\gamma_i$ induced by the data $(\alpha,J)$. Choose numbers $\delta_i^-<0$ in the spectral gap between eigenvalues of winding number zero and one with respect to~$\tau_\Sigma$, and set $\delta^-=(\delta^-_1,\dots,\delta^-_n;\emptyset)$. It follows from (a) in Proposition~\ref{prop_main_existence_non_deg} that this choice can be made. From now on we denote by $\jtil_k$ the $\R$-invariant almost complex structure defined as in~\eqref{formula_J_tilde} using $J$ and the contact form $f_k\alpha$.

\begin{remark}
Note that the numbers $\delta^-_i<0$ lie in the spectral gap of the asymptotic operator along $\gamma_i$ defined by the data $(f_k\alpha,J)$ between eigenvalues of winding number zero and one with respect to $\tau_\Sigma$, provided $k$ is large enough.
\end{remark}

Lemma~\ref{lemma_topological} shows that if we push a proper collection of components of $L$ in the direction of $\Sigma$ then we obtain a link which is non-zero in $H_1(M\setminus L)$. Summarizing, assumption (a) together with the results from~\cite{wendl} explained above allow us to check that all hypotheses (H1), (H2), (H3) and (H5) from subsection~\ref{ssec_existence_compactness_curves} are satisfied for $\alpha_+$ and $f_k\alpha$ in the place $\alpha_-$, where $\Sigma$ and $L$ are as just described.

We claim that (H4) is also satisfied if $k$ is large enough. We only need to find~$k_0$ such that if $k\geq k_0$ then every periodic orbit $\gamma=(x,T)$ in $M\setminus L$ of the Reeb flow of $f_k\alpha$ satisfying $T \leq \sum_{i=1}^n \int_{\gamma_i} \alpha_+$ must also satisfy ${\rm int}(\gamma,\Sigma)\neq 0$. Arguing indirectly we may assume that, up to the choice of a subsequence, for every $k$ there exists a 
periodic Reeb orbit $\tilde\gamma_k=(\tilde x_k,\tilde T_k)$ of $f_k\alpha$ contained in $M\setminus L$ and satisfying $\tilde T_k \leq \sum_{i=1}^n \int_{\gamma_i} \alpha_+$ and ${\rm int}(\tilde\gamma_k,\Sigma)=0$. By these period bounds we can assume, up to choice of a further subsequence, that $\tilde\gamma_k$ $C^\infty$-converges to a periodic orbit $\tilde\gamma=(\tilde x,\tilde T)$ of the Reeb flow of $\alpha$. Inequality $\tilde T \leq \sum_{i=1}^n \int_{\gamma_i} \alpha_+$ holds. It must be true that $\tilde\gamma \subset M\setminus L$. To see why we argue indirectly and assume that $\tilde\gamma$ is a multiple cover of some $\gamma_i$, say of covering multiplicity $N$. This forces the existence of a periodic solution of the linearization of the Reeb flow of $\alpha$ along the $N$-iterated orbit $\gamma_i^N$ with zero winding number computed with a frame aligned to the normal of $\Sigma$. In particular $N\rho^{\Theta}(\gamma_i) = \rho^\Theta(\gamma_i^N)=0 \Rightarrow \rho^\Theta(\gamma_i)=0$, which is absurd. Now that we proved that $\tilde\gamma\subset M\setminus L$, we must have ${\rm int}(\tilde\gamma,\Sigma)\neq 0$ by assumption. Since $\tilde\gamma_k \to \tilde\gamma$ we find ${\rm int}(\tilde\gamma_k,\Sigma)\neq0$ for large enough $k$. This contradiction finishes the proof that (H4) holds for $\alpha_+$ and $f_k\alpha$ in the place of $\alpha_-$ when $k$ is large enough.

The numbers $\delta^-_i<0$ lie in the spectral gap of the asymptotic operator along~$\gamma_i$ defined by the data $(f_k\alpha,J)$ between eigenvalues of winding number zero and one with respect to $\tau_\Sigma$, provided $k$ is large enough. This is a simple consequence of $f_k\alpha \to \alpha$ in $C^\infty$. From Proposition~\ref{prop_main_compactness} we get $k_0$ such that if $k\geq k_0$ then there is a curve $C_- \in \M_{\jtil_k,0,\delta^-}(\gamma_1,\dots,\gamma_n;\emptyset)$ whose projection to $M$ induces the same element in $H_2(M,L)$ as $\Sigma$. Its Hofer energy is not larger than $E$. From now on we fix $k\geq k_0$ and such a curve $C_-$. Let $\mathcal Y$ be the component of $C_-$ in $\M_{\jtil_k,0,\delta^-}(\gamma_1,\dots,\gamma_n;\emptyset)$. Proposition~\ref{prop_main_compactness} also tells us that $\mathcal{Y}/\R$ is compact.

Every  curve in $\mathcal Y$ projects to $M$ as a Seifert surface for $L$ transverse to $X_\alpha$ in its interior. These Seifert surfaces induce the same element in $H_2(M,L)$ as $\Sigma$. All this follows from Proposition~\ref{prop_Seifert_no_intersections} and Lemma~\ref{lemma_no_intersection_with_limits}. Lemma~\ref{lemma_aut_transv} and Proposition~\ref{prop_Seifert_two_curves} together tell us that the curves in $\mathcal Y$ can be used to construct pieces of foliations of $M\setminus L$ transverse to the Reeb flow. Using the compactness of $\mathcal{Y}/\R$ in combination with Proposition~\ref{prop_Seifert_two_curves}, these local foliations match to yield an open book decomposition of $M$ with binding $L$ whose pages are projections of curves in $\mathcal Y$. Moreover $\mathcal Y/\R$ is diffeomorphic to a circle.

To see why pages are global surfaces of section, consider a trajectory of the Reeb flow of $f_k\alpha$ on $M\setminus L$, where $k$ is large enough. If its $\omega$-limit does not intersect $L$ then it will hit every page in the future infinitely often, by transversality of $X_\alpha$ to the pages. If its $\omega$-limit intersects $L$ then this trajectory spends arbitrarily large amounts of time arbitrarily near $L$, in the far future, and consequently can be well-controlled by the linearized Reeb flow along $L$. The condition $\mu_{\CZ}^\Theta(\gamma_i)\geq 1 \ \forall i$ is equivalent to $\rho^\Theta(\gamma_i)>0 \ \forall i$ and, consequently, forces the linearized flow near $L$ to rotate very much with respect to $\tau_\Sigma$ in large amounts of time. Hence the same will happen to nearby trajectories, forcing them to hit all pages. To analyze past times there is a similar reasoning, where one replaces $\omega$-limit sets by $\omega^*$-limit sets. This concludes the proof of Proposition~\ref{prop_main_existence_non_deg}.

\subsection{Passing to the degenerate case}
\label{ssec_passing_to_deg_case}

\subsubsection{Geometric set-up}

Let us fix a planar supporting open book decomposition $\Theta=(\Pi,L)$ on $(M,\xi)$. Order the components $\gamma_1,\dots,\gamma_n$ of its binding $L$ arbitrarily, and assume that
\begin{equation}
\label{n_at_least_2_deg}
n\geq 2.
\end{equation}
Write $b \in H_2(M,L)$ for the class of a page. As explained in the introduction, the pages get naturally oriented by $\Theta$, and $L$ gets oriented as the boundary of a page. Let $\alpha$ be a contact form such that $\xi=\ker\alpha$ and $X_\alpha$ is positively tangent to each~$\gamma_l$. Note that $\alpha$ may be very degenerate. We see the $\gamma_l$ as prime closed Reeb orbits and denote their periods by~$T_l$. Assume that
\begin{itemize}
\item[(a)] $\mu_{\CZ}^\Theta(\gamma_l)\geq 1 \ \forall l$.
\item[(b)] All periodic orbits of the Reeb flow of $\alpha$ contained in $M\setminus L$ have non-zero algebraic intersection number with $b$.
\end{itemize}
Conley-Zehnder indices in (a) are computed with respect to trivializations of $\xi$ aligned with the normal of a page. Let $J$ be a $d\alpha$-compatible complex structure on~$\xi$. Consider a sequence $f_k$ satisfying
\begin{equation}
\label{properties_f_k}
f_k \in \mathcal{F}_L, \qquad f_k\to 1 \text{ in $C^\infty$}, \qquad f_k|_L \equiv 1 \ \forall k \, .
\end{equation}
Denote $\alpha_k = f_k\alpha$ and consider compatible complex structures $J_k:\xi\to\xi$ satisfying $J_k\to J$ in~$C^\infty$. Let $\jtil_k\in\J(f_k\alpha)$ and $\jtil\in\J(\alpha)$ be induced by $(\alpha_k,J_k)$ and $(\alpha,J)$, respectively. Then $\jtil_k\to\jtil$ in $C^\infty$ (weak equals strong). Note that the tangent space of $\R\times L$ is invariant under $\jtil$ and $\jtil_k$. The $\gamma_l$ are periodic orbits of the Reeb flow of $f_k\alpha$ with prime period $T_l$, for all~$k$. Choose $\delta_l<0$ in the spectral gap of the asymptotic operator of $\gamma_l$ induced by $(\alpha,J)$ between eigenvalues of winding number equal to $0$ and $1$ with respect to frames aligned to the normal of a page. This can be done in view of (a). If $k$ is large enough then the $\delta_l$ lie in the corresponding spectral gaps of asymptotic operators of $\gamma_l$ induced by $(\alpha_k,J_k)$. Denote $\delta=(\delta_1,\dots,\delta_n;\emptyset)$.

\subsubsection{Computing $\wind_\infty$}

%

For each $l$ choose a symplectic trivialization of $\xi_{\gamma_l}$ aligned with the normal of some page of $\Theta$. The collection of homotopy classes of these trivializations will be denoted by $\tau_\Sigma$.

\begin{lemma}\label{lemma_wind_infinity}
If $\util=(a,u)$ represents a curve in $\M_{\jtil,0,\delta}(\gamma_1,\dots,\gamma_n;\emptyset)$ inducing the class $b$, then $\wind_\pi(\util)=0$ and $u$ is an immersion transverse to $X_\alpha$. 
Moreover, if the positive puncture where $\util$ is asymptotic to $\gamma_l$ is denoted by $z_l$, then  $\wind_\infty(\util,z_l,\tau_\Sigma)=0$ holds for all $l$.
\end{lemma}

\begin{proof}
We may assume that $\util$ is defined on $\C\setminus\{z_2,\dots,z_n\}$ with standard complex structure $i$, that $\util$ is asymptotic to $\gamma_1$ at $z_1=\infty$ and to $\gamma_l$ at $z_l$, $2\leq l\leq n$. We use here that all punctures are non-degenerate, by our definition of moduli spaces given in~\ref{sssec_moduli_spaces}. In view of Theorem~\ref{thm_asymptotic_formula_deg_case} the numbers $\wind_\infty(\util,z_l,\tau_\Sigma)$ are well-defined. Since each $\delta_i$ lies in the spectral gap between eigenvalues of winding number~$0$ and~$1$ with respect to~$\tau$, we get
\begin{equation}
\label{spectral_info_punctures}
\wind_\infty(\util,z_l,\tau_\Sigma) \leq 0, \ \forall l \, .
\end{equation}
Now let $\tau'$ denote a collection of homotopy classes of symplectic trivializations of the $\xi_{\gamma_l}$ which extend to a trivialization of $u^*\xi$. Since $b$ is the class of a page of the supporting planar open book $\Theta$, the self-linking number ${\rm sl}(L,{\rm page})$ of $L$ with respect to a page of $\Theta$ is equal to $n-2$. Using this crucial information we can compute
\begin{equation}
\begin{aligned}
0 \leq \wind_{\pi}(\util) &= \wind_\infty(\util) - 2 + n \\
&= \left( \sum_{l=1}^n \wind_\infty(\util,z_l,\tau') \right) - 2 + n \\
&= \left( \sum_{l=1}^n \wind_\infty(\util,z_l,\tau_\Sigma) \right) -{\rm sl}(L,{\rm page}) -2+n \\
&= \sum_{l=1}^n \wind_\infty(\util,z_l,\tau_\Sigma) \leq 0.
\end{aligned}
\end{equation}
Hence $\wind_\pi(\util)=0$. This forces $u$ to be an immersion transverse to~$X_\alpha$, in view of the definition of $\wind_\pi$, and also proves that $\wind_\infty(\util,z_l,\tau_\Sigma)=0$ for all $l$.
\end{proof}

\subsubsection{A compactness statement}

Standard SFT-compactness is not valid for degenerate contact forms unless the degeneracy is very mild (e.g., Morse-Bott). Nevertheless, we can get the compactness statement necessary to prove our results. It reads as follows.

\begin{proposition}
\label{prop_abstract_compactness}
Consider a sequence $C_k \in \M_{\jtil_k,0,\delta}(\gamma_1,\dots,\gamma_n;\emptyset)$ satisfying
\begin{itemize}
\item[(H)] $C_k$ is embedded, represents the class $b$, does not intersect $\R\times L$, and any loop in $C_k$ projects to $M\setminus L$ as a loop with algebraic intersection number with $b$ equal to zero.
\end{itemize}
There exist representatives 
$$ 
C_k = [\util_k=(a_k,u_k),\C\cup\{\infty\},i,\{z_{k,1}=\infty,z_{k,2},\dots,z_{k,n}\},\emptyset], 
$$ 
distinct points $z_{\infty,2},\dots,z_{\infty,n}\in\C$,
sequences $k_j\to+\infty$, $c_j\in\R$, and a finite-energy $\jtil$-holomorphic embedding $$ \util=(a,u):\C\setminus \{z_{\infty,2},\dots,z_{\infty,n}\} \to \R\times M $$ satisfying
\[
\text{$z_{k_j,l}\to z_{\infty,l}$ as $j\to+\infty$, $\forall l\in\{2,\dots,n\}$,}
\]
\[
C = [\util,\C\cup\{\infty\},i,\{z_{\infty,1}=\infty,z_{\infty,2},\dots,z_{\infty,n}\},\emptyset] \in \M_{\jtil,0,\delta}(\gamma_1,\dots,\gamma_n;\emptyset),
\]
\[
(a_{k_j}+c_j,u_{k_j}) \to \util \ \ \text{ in } C^\infty_{\rm loc}(\C\setminus\{z_{\infty,2},\dots,z_{\infty,n}\}) \ \ \text{as $k\to+\infty$,}
\]
\[
\util(\C\setminus \{z_{\infty,2},\dots,z_{\infty,n}\}) \cap \R\times L = \emptyset.
\]
\end{proposition}

\begin{remark}
Using positivity of intersections one sees that the obtained limiting curve $C\in \M_{\jtil,0,\delta}(\gamma_1,\dots,\gamma_n;\emptyset)$ satisfies condition (H).
\end{remark}

Now we prove Proposition~\ref{prop_abstract_compactness}. Parametrize the $C_k$ by finite-energy $\jtil_k$-holomorphic maps
\begin{equation}
\util_k = (a_k,u_k) : \C\setminus \Gamma_k \to \R\times M
\end{equation}
where domains are equipped with the standard complex structure. Each
\[
\Gamma_k = \{z_{k,2},\dots,z_{k,n}\} \subset \C
\]
is a set of $n-1$ positive non-degenerate punctures, $\util_k$ is asymptotic to $\gamma_i$ at $z_{k,i}$ for all $i\in\{2,\dots,n\}$, and $z_{k,1}:= \infty$ is a positive non-degenerate puncture where $\util_k$ is asymptotic to $\gamma_1$.

For each $i$ choose a small compact tubular neighborhood $N_i$ of $\gamma_i$ such that the $N_i$ are pairwise disjoint, and each $N_i$ does not contain periodic orbits of the Reeb flow of $\alpha$ that are contractible in $N_i$.

\begin{lemma}
\label{lemma_right_parametrization}
After holomorphic reparametrizations we can achieve:
\begin{align}
& u_k^{-1}(M\setminus N_1) \subset \D \label{out_is_mapped_in} \\
& u_k^{-1}(\partial N_1) \cap \D \supset \{1,w_k\} \text{ where } \Re(w_k)\leq 0 \\
& \Gamma_k \subset \D \label{punctures_inside_the_disk} \\
& a_k(2) = 0 \label{real_component_at_2}
\end{align}
\end{lemma}

\begin{proof}
For fixed $k$ the closure $F$ in $\C$ of the set $u_k^{-1}(M\setminus N_1)$ is compact with non-empty interior. This is true since $\util_k$ is asymptotic to $\gamma_1$ at $\infty$ and is asymptotic to orbits in $N_2\cup\dots\cup N_n \subset M\setminus N_1$ at punctures in $\Gamma_k$. In particular, $\Gamma_k$ is contained in the interior of $F$. Let $r$ be the infimum among radii of closed disks containing $F$. Then $r>0$ and there exists a closed disk $D$ of radius $r$ containing~$F$. Obviously $u_k^{-1}(\partial N_1) \supset \partial D \cap F \neq \emptyset$.

Let $\ell$ be a line such that $\ell \cap D$ is a diameter of $D$. We claim that $\partial D \cap F$ is not contained in one of the connected components of $\C\setminus \ell$. To see this we argue by contradiction and assume that some component of $\C\setminus \ell$ contains $\partial D \cap F$. If $u\neq0$ is perpendicular to $\ell$ and points towards this connected component then we can find $\epsilon>0$ small enough such that $D+\epsilon u$ contains $F$ in its interior. Hence some disk of radius smaller than that of $D+\epsilon u$ would contain $F$, in contradiction with the definition of $r$.

Let $w_0 \in \partial D \cap F$ and let $\ell$ be the line through the center of $D$ that is perpendicular to segment joining $w_0$ to the center of $D$. Let $H_0,H_1$ be the closed half-spaces determined by $\ell$, where $H_0$ contains $w_0$. Above we proved that there exists $w_1 \in \partial D \cap F \cap H_1$. Choose $A\neq0,B$ such that $\psi(z)=Az+B$ satisfies: $\psi(0)$ is the center of $D$, $\psi(\D) = D$ and $\psi(1)=w_0$. Then $\psi$ maps the imaginary axis onto $\ell$, and $\Re(\psi^{-1}(w_1))\leq 0$. Redefining $\util_k$ to be $\util_k \circ \psi$ and translating it by $-a(2)$ in the $\R$-direction of $\R\times M$ we achieve all the desired properties.
\end{proof}

Let us denote by $\Sigma$ an arbitrarily fixed page of $\Theta$, so that $\Sigma$ is a Seifert surface for $L$ representing $b$. Along each $\gamma_l$ we choose a homotopy class of $d\alpha$-symplectic trivialization of $\xi_{\gamma_l}$ aligned with the normal of $\Sigma$. The collection of these homotopy classes is denoted by $\tau_\Sigma$. Note that trivializations representing $\tau_\Sigma$ can be rescaled to give $d(f_k\alpha)$-symplectic trivializations. We continue to write $\tau_\Sigma$ to denote their homotopy classes.

Before proceeding with the analysis we recall the description of $H_1(M\setminus L)$ from appendix~\ref{app_comp_H_1}. Lemma~\ref{lemma_topological} provides an isomorphism $$ H_1(M\setminus L) = \frac{H_1({\rm page})}{{\rm im}({\rm id}-h_*)} \oplus \Z e $$ where $h$ is the monodromy of $\Theta$ and $e$ is a $1$-cycle with algebraic intersection number~$+1$ with $b$. The first factor is seen geometrically in $M\setminus L$ as follows. Consider $\gamma_i'$ the loop obtained by pushing $\gamma_i$ into $M\setminus L$ in the direction of a page. We see $\gamma_i'$ as a $1$-cycle. Since $\Theta$ is planar, the first factor is the free abelian group generated by $\gamma'_1,\dots,\gamma'_n$ modulo the single relation $\gamma_1'+ \dots + \gamma_n' = 0$. This description of $H_1(M\setminus L)$ will be used many times in our analysis.

\begin{lemma}
\label{lemma_basic_homology_class}
Let $c$ be a loop in $M\setminus L$ that is $C^0$-close to an $m$-fold cover of $\gamma_l$ and has algebraic intersection number with $b$ equal to zero. Then $c$ is homologous to $m\gamma_l'$ in $M\setminus L$.
\end{lemma}

\begin{proof}
Consider a small tubular neighborhood $N$ of $\gamma_l$ equipped with coordinates $(\theta,w) \in \R/\Z\times\D$ with respect to which $\gamma_l = (x_l,T_l)$ is represented as $x_l(T_l\theta) = (\theta,0)$, and such that $\gamma_l'$ is homologous to $\theta\mapsto (\theta,1)$ in $M\setminus L$.  Since $\gamma_l'$ has zero algebraic intersection number with $b$ we can conclude that any loop $t\mapsto (\theta(t),w(t))$ in $N\setminus L$ has algebraic intersection number with $b$ equal to $\wind(w(t))$. Let $c$ be a loop inside $N$. Represent $c$ as $t \in \R/\Z\mapsto (\theta(t),w(t))$. If $c$ is $C^0$-close enough to the $m$ cover of $\gamma_l$ then it must be true that $t\mapsto \theta(t)$ has degree equal to $m$. Moreover, $\wind(w(t))=0$ by the assumption on the intersection number with $b$. Hence $c$ is homotopic in $N\setminus\gamma_l$ to $\theta\mapsto (m\theta,1)$. The latter is homologous to $m\gamma_l'$ in $M\setminus L$.
\end{proof}

\begin{lemma}
\label{lemma_explicit_homology_class}
For every $k$ one can find $r>0$ small enough such that the loop $t\mapsto u_k(r^{-1}e^{i2\pi t})$ is homologous to $\gamma_1'$ in $M\setminus L$, and for every $l=2,\dots,n$ the loop $t\mapsto u_k(z_{k,l}+r e^{-i2\pi t})$ is homologous to $\gamma'_l$ in $M\setminus L$.
\end{lemma}

\begin{proof}
From the fact that $z_{k,l}$ is a non-degenerate positive puncture where $\util_k$ is asymptotic to $\gamma_l$ it follows that for $r>0$ small enough the loop $t\mapsto u_k(z_{k,l}+r e^{-i2\pi t})$ is uniformly close to $\gamma_l$ and, by hypothesis (H), it has algebraic intersection number with $b$ equal to zero. Lemma~\ref{lemma_basic_homology_class} implies that $t\mapsto u_k(z_{k,l}+r e^{-i2\pi t})$ is homologous to $\gamma'_l$ in $M\setminus L$. The case $l=1$ is handled analogously.
\end{proof}

Up to selection of a subsequence we may assume that
\begin{equation}
\label{limit_points_Gamma_k_exist}
\lim_{k\to\infty} z_{k,i} \ \ \ \text{exists for every $i$.}
\end{equation}
Let $\Gamma$ the set of such limits. Hence $\Gamma \subset \D$ by Lemma~\ref{lemma_right_parametrization}.

Choose some reference $\R$-invariant Riemannian metric on $\R\times M$. Domains in $\C$ are equipped with standard euclidean metric. With these choices we can consider norms of differentials $d\util_k$.

\begin{lemma}\label{lemma_bounds_derivatives}
The sequence $d\util_k$ is $C^0_{\rm loc}$-bounded on $\C\setminus \Gamma$.
\end{lemma}

\begin{proof} 
Consider the set
\begin{equation*}
Z= \{ \zeta\in\C\setminus\Gamma \mid \exists \ k_l\to+\infty, \ \zeta_l\to\zeta \ \text{satisfying} \ |d\util_{k_l}(\zeta_l)|\to+\infty \}.
\end{equation*}
The goal of the lemma is to show that $Z$ is empty. We argue indirectly, and assume that $Z\neq\emptyset$. We must have $Z\subset \D$. If not then we can use~\eqref{out_is_mapped_in} and the analysis from~\cite{93} to conclude that a non-constant finite-energy plane bubbles-off inside $N_1$. All asymptotic limits of this plane must be periodic orbits of the Reeb flow of $\alpha$ inside $N_1$ that are contractible in $N_1$. This is in contradiction to the fact that $N_1$ contains no such periodic orbits.

The sequence $d\util_k$ is $C^0_{\rm loc}$-bounded in $\C\setminus (\Gamma\cup Z)$. Moreover, $2\not\in \Gamma\cup Z$ by~\eqref{punctures_inside_the_disk} and the inclusion $Z\subset \D$ proved above. Thus we get $C^1_{\rm loc}$ bounds for the sequence $\util_k$ on $\C\setminus(\Gamma\cup Z)$. By elliptic boot-strapping arguments there is no loss of generality to assume, up to selection of a subsequence, that $\util_k$ is $C^\infty_{\rm loc}(\C\setminus(\Gamma\cup Z))$-convergent to a smooth $\jtil$-holomorphic map
\begin{equation}\label{limit_map_util}
\util = (a,u) : \C\setminus (\Gamma\cup Z) \to \R\times M
\end{equation}
with finite Hofer energy. 

We can assume also that $\util$ is not constant. In fact, up to choice of a subsequence, we can assume that a small conformal disk $D$ around $\zeta\in Z$ satisfies $\int_D u_k^*d(f_k\alpha) \geq \sigma>0 \ \forall k$, where $\sigma>0$ is any positive constant smaller than the smallest period among periodic orbits of contact forms in $\{\alpha,f_1\alpha,f_2\alpha,f_3\alpha,\dots\}$ fixed {\it a priori}. This inequality is a consequence of the fact that a non-constant finite-energy $\jtil$-holomorphic plane bubbles off from $\zeta$ in the limit, up to a subsequence. Stokes theorem, and the fact that the loop $u_k|_{\partial D}$ $C^\infty$-converges to the loop $u|_{\partial D}$, imply that $\int_{\partial D}u^*\alpha \geq \sigma$. Hence $\zeta$ is a non-removable puncture and $\util$ is not a constant map. In fact, $\zeta$ is a negative puncture. 

Fix any $\zeta\in Z$ and let $\gamma_\zeta = (x_\zeta,T_\zeta)$ be one of the (possibly many) asymptotic limits of $\util$ at $\zeta$. We can find a sequence $r_l \to 0^+$ and $t_0\in\R$ such that $u(\zeta+r_le^{i2\pi t})$ $C^\infty$-converges to the loop $x_\zeta(T_\zeta t+t_0)$ as $l\to+\infty$. If $\gamma_\zeta \subset M\setminus L$ then $x_\zeta(T_\zeta t)$ has non-zero intersection number with the class $b$, by the standing assumption~(b). The same must be true then for $u(\zeta+r_le^{i2\pi{t}})$ if $l\gg1$, hence also for $u_k(\zeta+r_le^{i2\pi t})$ if $k\gg1$. But each $\util_k$ does not intersect $\R\times L$; this is a consequence of hypothesis (H). Hence the loop $u_k(\zeta+r_le^{i2\pi{t}})$ is contractible in $M\setminus L$ and its algebraic intersection number with $b$ vanishes. We used that $\zeta+r_le^{i2\pi{t}}$ is contractible in $\C\setminus\Gamma_k$. This contradiction proves that $\gamma_\zeta = \gamma_{j_*}^m = (x_{j_*},mT_{j_*})$ for some $j_* \in \{1,\dots,n\}$ and some $m\geq1$. This argument shows independently that if $k,l$ are large enough then the loop $u_k(\zeta+r_le^{i2\pi{t}})$ has vanishing algebraic intersection number with $b$. 

It follows from Lemma~\ref{lemma_basic_homology_class} that $t\mapsto u_k(\zeta+r_le^{i2\pi{t}})$ is homologous to $m\gamma'_{j_*}$ in $M\setminus L$ when $k,l\gg1$. Combining with Lemma~\ref{lemma_explicit_homology_class} and hypothesis (H) we obtain the identity
\[
\gamma_1'+ \dots + \gamma_n' \equiv m\gamma'_{j_*} \qquad \text{in $H_1(M\setminus L)$}.
\]
By our standing assumption~\eqref{n_at_least_2_deg} this identity contradicts Lemma~\ref{lemma_topological} since $m\geq1$. We have shown that $Z=\emptyset$.
\end{proof}

Lemma~\ref{lemma_bounds_derivatives} and~\eqref{real_component_at_2} together give $C^1_{\rm loc}$ bounds for $\util_k$ on $\C\setminus\Gamma$. By elliptic boot-strapping arguments there is no loss of generality, up to selection of a subsequence, to assume that $\util_k$ converges in $C^\infty_{\rm loc}(\C\setminus \Gamma)$ to a smooth $\jtil$-holomorphic map
\begin{equation}
\util=(a,u):\C\setminus \Gamma \to \R\times M
\end{equation}
with Hofer energy not larger than $\sum_{l}\int_{\gamma_l}\alpha$.

\begin{lemma}
The map $\util$ is not constant.
\end{lemma}

\begin{proof}
For every $R>1$ and every $k$ the loop $t\mapsto u_k(Re^{i2\pi t})$ lies inside $N_1$ and is homotopic in $N_1$ to $\gamma_1$. This follows from~\eqref{out_is_mapped_in}. Since these loops converge in $C^\infty$ to the loop $t\mapsto u(Re^{i2\pi t})$ we conclude that the latter is a non-contractible loop in~$N_1$, hence non-constant.
\end{proof}

\begin{lemma}
\label{lemma_asymptotic_props_limit_curve}
The following hold:
\begin{itemize}
\item $\#\Gamma=n-1$ and $\Gamma\cup\{\infty\}$ consists of positive punctures of $\util$. 
\item Use~\eqref{limit_points_Gamma_k_exist} to write $\Gamma = \{z_{\infty,2},\dots,z_{\infty,n}\}$ where $z_{\infty,i} = \lim_{k\to\infty} z_{k,i}$. Set $z_{\infty,1}=\infty$. Then $\util$ is weakly asymptotic to $\gamma_i$ at~$z_{\infty,i}$ for every $i=1,\dots,n$. 
\item The image of $\util$ does not intersect $\R\times L$.
\end{itemize}
\end{lemma}

\begin{proof}
Let $z\in\Gamma$. By~\eqref{limit_points_Gamma_k_exist} and the definition of $\Gamma$ $$ I_z=\{l \in \{2,\dots,n\} \mid z_{k,l} \to z \text{ as } k\to+\infty\} $$ is a non-empty set.

We claim that $z$ is not a removable puncture. In fact, assume that $z$ is removable. Thus $\util$ smoothly extends at $z$. Consider first the case $u(z) \in L$. Choose a small open tubular neighborhood $N$ with coordinates $(\theta,w) \in \R/\Z \times \C$ around the component $\gamma_j \subset L$ containing the point $u(z)$, in such a way that $\gamma_j = \R/\Z\times0$ and any loop $t\mapsto(\theta(t),w(t))$ in $N\setminus \gamma_j$ has algebraic intersection number with $b$ equal to $\wind(w(t))$. If $\epsilon$ is small enough and $k$ is large enough then the loop $t\mapsto u_k(z+\epsilon e^{i2\pi t})$ is contained in $N$ since it is $C^0$-close to $u(z)\in\gamma_j$. Thus it can be represented in coordinates as $t\mapsto (\theta_k(t),w_k(t))$, where $t\mapsto \theta_k(t)$ has degree zero. By the hypothesis (H) we know that $\wind(w_k(t))=0$. Consequently $t\mapsto (\theta_k(t),w_k(t))$ is homotopic in $N\setminus\gamma_j$ to a constant loop. Putting this fact together with Lemma~\ref{lemma_explicit_homology_class} and hypothesis (H) we conclude that $\sum_{l\in I_z}\gamma'_l \equiv 0$ in $H_1(M\setminus L)$. This is impossible because $I_z$ is a proper subset of $\{1,\dots,n\}$. We are left with the case $u(z)\not\in L$, but this time we find that for $\epsilon$ small enough and $k$ large enough the loop $t\mapsto u_k(z+\epsilon e^{i2\pi t})$ is contained in a contractible open subset of $M\setminus L$. The previous argument again show that $\sum_{l\in I_z}\gamma'_l \equiv 0$ in $H_1(M\setminus L)$, which is absurd. In all cases we get a contradiction. Thus $z$ must be non-removable. An analogous argument, using~\eqref{punctures_inside_the_disk}, shows that $\infty$ is non-removable.

We can now prove that every asymptotic limit of $\util$ at a given $z\in\Gamma$ is contained in $L$. In fact, it follows from our standing assumption (b) that if some asymptotic limit of $\util$ at some $z\in\Gamma$ is contained in $M\setminus L$ then we find $\epsilon>0$ small such that $u$ maps $z+\epsilon S^1$ to a loop in $M\setminus L$ with non-zero algebraic intersection number with $b$. If $k$ is large enough then the same is true for the loop $u_k(z+\epsilon S^1)$, in contradiction to hypothesis (H). Analogously one shows that every asymptotic limit at $\infty$ is contained in $L$.

Let $e=\pm1$ be the sign of the non-removable puncture $z\in\Gamma$. By the reasoning above, any choice of asymptotic limit of $\util$ at $z$ must be equal to $\gamma_{l_*}^N = (x_{l_*},NT_{l_*})$ for some $l_*\in\{1,\dots,n\}$ and some $N\geq 1$. Let us fix such a choice. We find $\epsilon_m\to0^+$ and $t_0$ such that $t\mapsto u(z+\epsilon_m e^{i2\pi t})$ $C^\infty$-converges to $t\mapsto x_{l_*}(-eNT_{l_*}t+t_0)$. By hypothesis (H) we find $m$ large and $k_m$ such that if $k\geq k_m$ then $t\mapsto u_k(z+\epsilon_m e^{i2\pi t})$ is a loop in $M\setminus L$ which is $C^\infty$-close to $t\mapsto x_{l_*}(-eNT_{l_*}t+t_0)$ and has algebraic intersection number with $b$ equal to zero. Hence $t\mapsto u_k(z+\epsilon_m e^{i2\pi t})$ is homologous to $-eN\gamma'_{l_*}$ in $M\setminus L$ whenever $k\geq k_m$; we used Lemma~\ref{lemma_basic_homology_class}. We can now choose $k\gg k_m$ such that all $\{z_{k,l}\mid l\in I_z\}$ lie in interior of the closed $\epsilon_m$-disk $D_{\epsilon_m}(z)$ centered at $z$. Take $r_k>0$ small enough such that the closed $r_k$-disks $D_{r_k}(z_{k,l})$ centered at the $\{z_{k,l}\mid l\in I_z\}$ lie in interior of $D_{\epsilon_m}(z)$, and such that $u_k$ maps $-\partial D_{r_k}(z_{k,l})$ to a loop homologous to $\gamma_l'$. We strongly used Lemma~\ref{lemma_explicit_homology_class} here. Now consider the smooth domain $S$ obtained by removing from $D_{\epsilon_m}(z)$ the interiors of the disks $D_{r_k}(z_{k,l})$, $l\in I_z$. Since $u_k(S) \subset M\setminus L$ we get a homology relation
\begin{equation}\label{hom_identity}
-eN\gamma'_{l_*} + \sum_{l\in I_z} \gamma'_l \equiv 0 \text{ in $H_1(M\setminus L)$.}
\end{equation}
If $e=-1$ then~\eqref{hom_identity} implies that $I_z=\{2,\dots,n\}$, $l_*=1$, $N=1$ and $\Gamma=\{z\}$. 
It follows that $\util$ is a trivial cylinder over $\gamma_1$, its positive puncture is $\infty$ and its negative puncture is $z$. 
It is now, of course, crucial to observe that $\{u_k(1),u_k(w_k)\} \subset\partial N_1$ where $w_k\in\partial\D$ has negative real part. 
Up to choice of subsequence we may assume $w_k \to w_*\in\partial\D$, $\Re(w_*)\leq0$. 
Since $w_*\neq 1$ and $\Gamma=\{z\}$, there is at least one point in $\partial\D$ mapped by $u$ to $\partial N_1$, contradicting that $\util$ is a trivial cylinder over $\gamma_1$. 
We are done showing $e=+1$. 
Together with~\eqref{hom_identity} this implies that $l_*\in\{2,\dots,n\}$, $I_z=\{l_*\}$ and $N=1$.

We concluded the proof that $\#\Gamma=n-1$, that we can order the points of $\Gamma=\{z_{\infty,2},\dots,z_{\infty,n}\}$ in such a way that $z_{k,l} \to z_{\infty,l}$ as $k\to+\infty$, and that $\util$ is weakly asymptotic to $\gamma_l$ at $z_{\infty,l}$ for every $l\in\{2,\dots,n\}$.

We prove now that $\infty$ is a positive puncture where $\util$ is weakly asymptotic to $\gamma_1$. The argument is entirely analogous. Any asymptotic limit of $\util$ at $\infty$ is of the form $\gamma_{l_*}^N$ for some $l_*=1,\dots,n$ and some $N\geq 1$. By hypothesis (H) we find that $R$ and $k$ large such that the loop $t\mapsto u_k(Re^{i2\pi t})$ represents the class $eN\gamma_{l_*}'$ in $H_1(M\setminus L)$, where $e$ is the sign of the puncture $\infty$. Again using (H) and what was proved above we get an identity
\[
eN\gamma'_{l_*} + \gamma_2' + \dots + \gamma_n' \equiv 0 \text{ in $H_1(M\setminus L)$.}
\]
This forces $l_*=1$, $e=+1$ and $N=1$.

Since $\Gamma\cup\{\infty\}$ consists precisely of $n$ positive punctures $\infty,z_{\infty,2},\dots,z_{\infty,n}$ with geometrically distinct asymptotic limits, the image of $\util$ can not be contained in any component of the embedded $\jtil$-holomorphic surface $\R\times L$. Carleman's similarity principle implies that $E=\{w\in\C\setminus \Gamma \mid \util(w) \in \R\times L\}$ is finite. Positivity of intersections guarantees that the (local) algebraic intersection number between $\util$ and $\R\times L$ at any point of $E$ is positive. If $E$ is non-empty we will find intersections of the image of $\util_k$ with $\R\times L$ provided $k$ is large enough. This is in contradiction with~hypothesis (H).
\end{proof}

\begin{lemma}
\label{lemma_non_deg_punctures_missing}
The curve $[\util,\C\cup\{\infty\},i,\Gamma,\emptyset]$ belongs to $\M_{\jtil,0,\delta}(\gamma_1,\dots,\gamma_n;\emptyset)$.
\end{lemma}

Most of the remainder of the proof of Proposition~\ref{prop_abstract_compactness} consists of establishing Lemma~\ref{lemma_non_deg_punctures_missing}. The main difficulty is to prove that the punctures are non-degenerate in the sense of Definition~\ref{def_non_deg_punctures}. We collect an immediate consequence.

\begin{lemma}
\label{lemma_deg_util_is_an_embedding}
The map $\util$ is an embedding.
\end{lemma}

\begin{proof}
[Proof of Lemma~\ref{lemma_deg_util_is_an_embedding} assuming Lemma~\ref{lemma_non_deg_punctures_missing}]
Lemma~\ref{lemma_wind_infinity} implies that $\util$ is an immersion. Positivity and stability of intersections of pseudo-holomorphic immersions shows that a self-intersection point of $\util$ would force self-intersection points of the $\util_k$ for $k$ large enough. This contradicts hypothesis (H).
\end{proof}

Lemmas~\ref{lemma_asymptotic_props_limit_curve},~\ref{lemma_non_deg_punctures_missing} and~\ref{lemma_deg_util_is_an_embedding} together complete the proof of Proposition~\ref{prop_abstract_compactness}. From now we are concerned with the proof of Lemma~\ref{lemma_non_deg_punctures_missing}.

Fix $m\in\{1,\dots,n\}$ and choose positive holomorphic polar coordinates $(s,t)\in[0,+\infty)\times\R/\Z$ centered at $z_{\infty,m}$. Choose $\psi_k$ a sequence of M\"obius transformations satisfying
\[
\psi_k(\infty) = \infty, \ \psi_k(z_{\infty,m}) = z_{k,m} \qquad \psi_k\to id \ \text{in $C^\infty_{\rm loc}$.}
\]
This sequence exists since $z_{k,m} \to z_{\infty,m}$ as $k\to\infty$. Define
\begin{equation}
\util'_k = \util_k \circ \psi_k = (a'_k,u'_k)
\end{equation}
which is a finite-energy $\jtil_k$-holomorphic map with positive punctures at 
$$ 
z'_{k,1} := \infty, \ z'_{k,2} := \psi_k^{-1}(z_{k,2}),\dots, \ z'_{k,n} := \psi_k^{-1}(z_{k,n}). 
$$ 
Note that $z'_{k,m}=z_{\infty,m}$ for every $k$ and that, by Lemma~\ref{lemma_asymptotic_props_limit_curve}, $\util_k'$ is weakly asymptotic to $\gamma_l$ at $z'_{k,l}$ for every $l$. We write 
$$ 
\util_k'(s,t) = (a'_k(s,t),u'_k(s,t)) \qquad \util(s,t) = (a(s,t),u(s,t)) 
$$ 
where $(s,t)$ are the holomorphic polar coordinates at $z'_{k,m}=z_{\infty,m}$ fixed above.

\begin{lemma}\label{lemma_crucial_control_of_ends}
For all sequences $s_k\to+\infty$, $t_k\in\R/\Z$ there exists a subsequence of 
$$
(s,t) \mapsto (a_k'(s+s_k,t+t_k)-a_k'(s_k,t_k),u_k'(s+s_k,t+t_k))
$$ 
that converges in $C^\infty_{\rm loc}$ to a trivial cylinder over $\gamma_m$.
\end{lemma}

\begin{proof}
Denote $\vtil_k(s,t) = (a_k'(s+s_k,t+t_k))-a_k'(s_k,t_k),u_k'(s+s_k,t+t_k))$ which is defined on $[-s_k,+\infty)\times\R/\Z$. Write $\vtil_k=(d_k,v_k)$ for the components of $\vtil_k$. We claim that
\begin{equation}\label{contact_area_asymptotic_tail}
\lim_{k\to\infty} \int_{[\hat s,+\infty)\times\R/\Z} v_k^*d\alpha_k = 0
\end{equation}
and that
\begin{equation}\label{contact_action_asymptotic_tail}
\lim_{k\to\infty} \int_{\{\hat s\}\times\R/\Z}v_k^*\alpha_k = T_m
\end{equation}
hold for every $\hat s\in\R$. In fact, fix $\epsilon>0$ be arbitrarily. Recall that $\util$ is weakly asymptotic to $\gamma_m$ at $z_{\infty,m}$ (Lemma~\ref{lemma_asymptotic_props_limit_curve}). Hence, there exists $s_\epsilon>0$ such that
\[
\int_{\{s_\epsilon\}\times \R/\Z} (u')^*\alpha \ \in \ [T_m-\epsilon/2,T_m].
\]
Here the upper bound $T_m$ follows from Stokes theorem since $(u')^*d\alpha_k$ is a non-negative multiple of $ds\wedge dt$. Now use $u_k' \to u$ in $C^\infty_{\rm loc}$ to find $k_\epsilon\geq 1$ satisfying
\begin{equation}\label{key_action_control}
k\geq k_\epsilon \Rightarrow \int_{\{s_\epsilon\}\times\R/\Z} (u'_k)^*\alpha_k \ \in \ [T_m-\epsilon,T_m].
\end{equation}
The upper bound $T_m$ is proven as before. Putting together Stokes theorem,~\eqref{key_action_control} and the fact that $\util_k'$ is asymptotic to $\gamma_m$ at $z'_{k,m}=z_{\infty,m}$ we get
\[
k\geq k_\epsilon \Rightarrow \int_{[s_\epsilon,+\infty)\times\R/\Z} (u'_k)^*d\alpha_k \leq \epsilon.
\]
Since
\[
\int_{[\hat s,+\infty)\times\R/\Z} v_k^*d\alpha_k = \int_{[\hat s+s_k,+\infty)\times\R/\Z} (u'_k)^*d\alpha_k
\]
we get
\[
\limsup_{k\to\infty} \int_{[\hat s,+\infty)\times\R/\Z} v_k^*d\alpha_k \leq \epsilon
\]
because $\hat s+s_k \geq s_\epsilon$ when $k$ is large enough. This proves~\eqref{contact_area_asymptotic_tail} since $\epsilon$ can be taken arbitrarily small. Again using~\eqref{key_action_control} together with Stokes theorem we can estimate with $k\geq k_\epsilon$ and $\hat s+s_k \geq s_\epsilon$
\[
\begin{aligned}
\int_{\{\hat s\}\times\R/\Z} v_k^*\alpha_k &= \int_{\{\hat s+s_k\}\times\R/\Z} (u'_k)^*\alpha_k \\
&= \int_{\{s_\epsilon\}\times\R/\Z} (u'_k)^*\alpha_k + \int_{[s_\epsilon,\hat s+s_k]\times\R/\Z} (u'_k)^*d\alpha_k \\
&\geq \int_{\{s_\epsilon\}\times\R/\Z} (u'_k)^*\alpha_k \geq T_m-\epsilon
\end{aligned}
\]
and
\[
\begin{aligned}
\int_{\{\hat s\}\times\R/\Z} v_k^*\alpha_k &= \int_{\{\hat s+s_k\}\times\R/\Z} (u'_k)^*\alpha_k \\
&= T_m - \int_{[\hat s+s_k,+\infty)\times\R/\Z} (u'_k)^*d\alpha_k \leq T_m.
\end{aligned}
\]
Again the non-negativity of $(u'_k)^*d\alpha_k$ with respect to $ds\wedge dt$ was used. Now~\eqref{contact_action_asymptotic_tail} follows since $\epsilon$ can be taken arbitrarily small.

The sequence $d\vtil_k$ is $C^0_{\rm loc}$-bounded, for if not then we would find a bounded bubbling-off sequence of points, hence a compact set $F\subset\R/\times\R/\Z$ satisfying $\liminf_k \int_Fv_k^*d\alpha_k > 0$, contradicting~\eqref{contact_area_asymptotic_tail}. Since $d_k(0,0) = 0$ we obtain $C^1_{\rm loc}$-bounds for $\vtil_k$. Elliptic estimates provide $C^\infty_{\rm loc}$-bounds. Up to a further subsequence we get a smooth finite-energy $\jtil$-holomorphic map $\vtil:\R\times\R/\Z \to \R\times M$ as a $C^\infty_{\rm loc}$-limit of the $\vtil_k$. Write $\vtil=(d,v)$ for the components of $\vtil$. The map $\vtil$ is not constant since~\eqref{contact_action_asymptotic_tail} implies that
\[
\int_{\{\hat s\}\times\R/\Z} v^*\alpha = T_m, \ \ \ \forall \hat s\in\R.
\]
Moreover,
\begin{equation*}
\int_{\R\times\R/\Z} v^*d\alpha = 0
\end{equation*}
holds in view of~\eqref{contact_area_asymptotic_tail}. It follows that $\vtil$ is a cylinder over some periodic orbit~$\hat\gamma=(\hat x,\hat T=T_m)$ with the same period as $\gamma_m$. In particular, there is $t_0$ such that $\vtil(s,t) = (T_ms,\hat x(T_mt+t_0))$.

Suppose, by contradiction, that $\gamma_m$ and $\hat\gamma$ are geometrically distinct. By the standing assumption (b) in the beginning of this subsection we know that the loop $t\mapsto v(0,t) = \hat x(T_mt+t_0)$ has algebraic intersection number with $b$ different from zero if it is contained in $M\setminus L$. Since the sequence of loops $u'_k(s_k,\cdot)$ converges in $C^\infty$ to the loop $v(0,\cdot)$, then for $k\gg1$ the loop $u'_k(s_k,\cdot)$ has non-zero algebraic intersection number with $b$. This is a contradiction to hypothesis~(H). We have established that $\vtil$ is a trivial cylinder over some cover of some component of $L$. Next we claim that this component must be $\gamma_m$. If not we find $s_k'\geq s_k$ and $t'_k\in\R/\Z$ such that $u'_k(s'_k,t'_k) \in \partial N_m$. Here we used that each $\util_k'$ is asymptotic to $\gamma_m$ at $z'_{k,m}=z_{\infty,m}$. Then, by the arguments showed above, a subsequence of the sequence of maps $\vtil'_k(s,t) = (a'_k(s+s'_k,t+t_k')-a(s'_k,t'_k),u'_k(s+s'_k,t+t'_k))$ converges to a trivial cylinder over some periodic orbit contained in $L$ which also touches $\partial N_m$ because $\vtil'_k(0,0) \in 0\times\partial N_m$. This contradiction shows that $\hat\gamma$ is a cover of $\gamma_m$. Thus $\hat\gamma=\gamma_m$ since they have the same period.
\end{proof}

\begin{lemma}\label{lemma_uniform_ends_behavior}
There exist $s_0\geq 0$ and $k_0\geq1$ such that $u_k'(s,t) \in N_m$ for every $k\geq k_0$ and every $s\geq s_0$.
\end{lemma}

\begin{proof}
We argue by contradiction. If the statement is not true then, up to selection of a subsequence, we may assume that there exist sequences $s_k\to+\infty$, $t_k\in\R/\Z$ such that $u_k'(s_k,t_k) \in \partial N_m$. By the previous lemma we know that the sequence of cylinders $$ \vtil_k(s,t) = (a_k'(s+s_k,t+t_k))-a_k'(s_k,t_k),u_k'(s+s_k,t+t_k)) $$ defined on $[-s_k,+\infty)\times\R/\Z$ converges in $C^\infty_{\rm loc}$ to a trivial cylinder over $\gamma_m$. From $\vtil_k(0,0) \in 0\times\partial N_m$ we get that $\gamma_m$ intersects $\partial N_m$, absurd.
\end{proof}

Since the $N_l$ are tubular neighborhoods of the $\gamma_l$ which are allowed to be taken arbitrarily small, there is no loss of generality to assume that $N_m$ is contained in the domain of coordinates
\[
(\theta,z=x+iy) \in \R/\Z\times \C
\]
given by a Martinet tube for $(\gamma_m,\alpha)$; see Definition~\ref{def_Martinet_tube}. The contact forms $\alpha_k,\alpha$ are given on $N_m$ as
\[
\alpha_k = f_kh(d\theta+xdy) \qquad \alpha = h(d\theta+xdy)
\]
for some $h=h(\theta,z)$ satisfying $h(\theta,0)=T_m$ and $dh(\theta,0)=0$ for all $\theta$. In view of~\eqref{properties_f_k} the $f_kh$ satisfy these same properties. In coordinates we write
\[
X_\alpha = (X_\alpha^1,Y) \qquad X_{\alpha_k} = (X_{\alpha_k}^1,Y_k)
\]
and define matrix-valued functions
\begin{equation}
D_k(\theta,z) = \int_0^1 D_2Y_k(\theta,\tau z)d\tau \qquad D(\theta,z) = \int_0^1 D_2Y(\theta,\tau z)d\tau
\end{equation}

By Lemma~\ref{lemma_uniform_ends_behavior} the components
\[
w_k(s,t) = (\theta_k(s,t),z_k(s,t) = x_k(s,t)+iy_k(s,t))
\]
of $u'_k(s,t)$ are well-defined functions of $(s,t) \in [s_0,+\infty)\times\R/\Z$. So are the components
\[
w(s,t) = (\theta(s,t),z(s,t)=x(s,t)+iy(s,t))
\]
of $u(s,t)$. We already know that
\[
w_k(s,t) \to w(s,t) \ \text{in $C^\infty_{\rm loc}$.}
\]

In the frame $\{\partial_x,-x\partial_\theta+\partial_y\}$ of $\xi|_{N_m}$ we can represent $J_k,J$ as matrix-valued functions of $(\theta,z)$. The Cauchy-Riemann equations for $\util_k'(s,t)$ and $\util(s,t)$ read
\begin{equation}\label{cauchy_riemann_1}
\begin{aligned}
& \left\{
\begin{aligned}
& \partial_sa_k'-((f_kh)\circ w_k)(\partial_t\theta_k+x_k\partial_ty_k) = 0 \\
& \partial_s\theta_k + ((f_kh)\circ w_k)^{-1}\partial_ta_k' +x_k\partial_sy_k = 0
\end{aligned}
\right. \\
& \left\{
\begin{aligned}
& \partial_sa-(h\circ w)(\partial_t\theta+x\partial_ty) = 0 \\
& \partial_s\theta + (h\circ w)^{-1}\partial_ta +x\partial_sy = 0
\end{aligned}
\right.
\end{aligned}
\end{equation}
and
\begin{equation}\label{cauchy_riemann_2}
\begin{aligned}
& \partial_sz_k + (J_k\circ w_k)\partial_tz_k+S_kz_k = 0 \\
& \partial_sz + (J\circ w)\partial_tz+Sz = 0
\end{aligned}
\end{equation}
where
\begin{equation}\label{formulas_S_k_S}
\begin{aligned}
& S_k(s,t) = [\partial_ta_k' I -\partial_sa_k' (J_k\circ w_k)]D_k\circ w_k \\
& S(s,t) = [\partial_ta I -\partial_sa (J\circ w)]D\circ w.
\end{aligned}
\end{equation}

The functions $\theta(s,t),\theta_k(s,t)$ take values in $\R/\Z$. By Lemma~\ref{lemma_asymptotic_props_limit_curve} the degrees of the maps $\theta(s,\cdot)$, $\theta_k(s,\cdot)$ are equal to $1$ since $\gamma_m$ is simply covered. Since $z_{k,m}$ is a non-degenerate puncture of $\util_k$, we know that $z'_{k,m}=z_{\infty,m}$ is a non-degenerate puncture of $\util_k'$, and Theorem~\ref{thm_asymptotic_formula_deg_case} guarantees that $\lim_{s\to+\infty}\theta_k(s,0)$ exists in $\R/\Z$. Choose unique lifts of $\tilde\theta_k:\R\times\R\to\R$ determined by $\lim_{s\to+\infty} \tilde\theta_k(s,0) \in [0,1)$. Up to a subsequence we may assume without loss of generality that $$ \lim_{k\to\infty} \left( \lim_{s\to+\infty} \tilde\theta_k(s,0) \right) \ \text{exists in $[0,1]$.} $$ Hence there is a unique lift $\tilde\theta:\R\times\R\to\R$ of $\theta(s,t)$ determined by requiring $\tilde\theta_k(s,t) \to \tilde\theta(s,t)$ in $C^\infty_{\rm loc}$. We have $\tilde\theta_k(s,t+1)=\tilde\theta_k(s,t)+1$ and $\tilde\theta(s,t+1)=\tilde\theta(s,t)+1$.

\begin{lemma}\label{lemma_first_estimate}
The following holds:
\begin{equation}\label{first_estimate}
\begin{aligned}
& \lim_{s\to+\infty} \sup_{k,t} \left(|D^\beta[a_k'-T_ms]| + |D^\beta[\tilde\theta_k-t]|\right) = 0 \ \ \text{$\forall \beta$ such that $|\beta|\geq 1$} \\
& \lim_{s\to+\infty} \sup_{k,t} |D^\beta z_k| = 0 \ \ \forall \beta
\end{aligned}
\end{equation}
where $D^\beta = \partial_s^{\beta_1}\partial_t^{\beta_2}$ and $|\beta|=\beta_1+\beta_2$.
\end{lemma}

\begin{proof}
Direct consequence of Lemma~\ref{lemma_crucial_control_of_ends}.
\end{proof}

From now we fix a smooth maps $M_k(\theta,z),M(\theta,z) \in Sp(2)$ defined on a common neighborhood of $\R/\Z\times 0$ satisfying
\begin{equation*}
M_kJ_k = J_0M_k \qquad MJ = J_0M \qquad M_k \to M \ \text{in} \ C^\infty
\end{equation*}
where $J_0 = \begin{pmatrix} 0 & -1 \\ 1 & 0 \end{pmatrix}$. Here we used that $J_k \to J$ in $C^\infty$. Define
\begin{equation*}
\begin{aligned}
& J^\infty(t) = J(t,0) \\
& D^\infty(t) = D_2Y(t,0) \\
& M^\infty(t) = M(t,0) \\
& S^\infty(t) = -T_mJ^\infty(t)D^\infty(t).
\end{aligned}
\end{equation*}

\begin{lemma}{\cite[Lemma~4.12]{elliptic}}\label{lemma_asymptotic_data}
For all pairs of sequences $k_j,s_j \to +\infty$ there exists $j_i \to +\infty$ and a number $c\in[0,1]$ such that
\begin{equation*}
\begin{aligned}
& \lim_{i\to+\infty} \|D^\beta[J_{k_{j_i}}\circ w_{k_{j_i}}-J^\infty(t+c)](s_{j_i},\cdot)\|_{L^\infty(\R/\Z)} = 0 \\
& \lim_{i\to+\infty} \|D^\beta[D_{k_{j_i}}\circ w_{k_{j_i}}-D^\infty(t+c)](s_{j_i},\cdot)\|_{L^\infty(\R/\Z)} = 0 \\
& \lim_{i\to+\infty} \|D^\beta[M_{k_{j_i}}\circ w_{k_{j_i}}-M^\infty(t+c)](s_{j_i},\cdot)\|_{L^\infty(\R/\Z)} = 0 \\
& \lim_{i\to+\infty} \|D^\beta[S_{k_{j_i}}-S^\infty(t+c)](s_{j_i},\cdot)\|_{L^\infty(\R/\Z)} = 0
\end{aligned}
\end{equation*}
for every $D^\beta= \partial_s^{\beta_1}\partial_t^{\beta_2}$.
\end{lemma}

\begin{proof}
We only address the first limit. 
The second and third follow analogously, and the fourth follows as a consequence of the the first three and of~\eqref{formulas_S_k_S}.

In this proof we may see functions defined on $\R/\Z\times\C$ as functions defined on $\R\times \C$ which are $1$-periodic in the first coordinate. Set $c_j = \tilde\theta_{k_j}(s_j,0) \in [0,1)$. 
Choose $j_i$ such that $c:=\lim_{i\to+\infty} c_{j_i}$ exists in $[0,1]$. 
Consider the function $$ \Delta_k(s,t) = (\tilde\theta_k(s,t) - t - \tilde\theta_k(s,0),z_k(s,t)). $$ 
By Lemma~\ref{lemma_first_estimate} we get
\begin{equation}\label{error_1}
\lim_{s\to+\infty} \sup_{k,t} |D^\beta\Delta_k(s,t)| = 0 \qquad \forall \beta.
\end{equation}
Write
\begin{equation}
J_k\circ w_k(s,t) = J_k(t+\tilde\theta_k(s,0),0) + \epsilon_k(s,t)
\end{equation}
where $$ \epsilon_k(s,t) = \int_0^1 DJ_k((t+\tilde\theta_k(s,0),0)+\tau\Delta_k(s,t))d\tau \cdot \Delta_k(s,t). $$ 
Using~\eqref{error_1} and the fact that any partial derivative of $J_k$ is uniformly (also in $k$) bounded on a fixed compact neighborhood of $\R/\Z\times 0$, we get
\begin{equation}\label{error_2}
\lim_{s\to+\infty} \sup_{k,t} |D^\beta\epsilon_k(s,t)| = 0.
\end{equation}
Finally we can write
\[
J_{k_j}\circ w_{k_j}(s,t)-J^\infty(t+c) = \epsilon_{k_j}(s,t) + \epsilon'_j(s,t) + \epsilon''_j(t)
\]
where
\[
\begin{aligned}
& \epsilon'_j(s,t) = J_{k_j}(t+\tilde\theta_{k_j}(s,0),0) - J_{k_j}(t+c_j,0) \\
& \epsilon''_j(t) = J_{k_j}(t+c_j,0) - J(t+c,0)
\end{aligned}
\]
Write
\[
\epsilon'_j(s,t) = \int_0^1 DJ_{k_j}(t+\tau(\tilde\theta_{k_j}(s,0)-c_j),0)d\tau \cdot (\tilde\theta_{k_j}(s,0)-c_j,0)
\]
and use Lemma~\ref{lemma_first_estimate} to conclude that
\[
\lim_{j\to+\infty} \sup_{t}|D^\beta\epsilon'_j(s_j,t)| = 0 \qquad \forall \beta.
\]
Obviously
\[
\lim_{i\to+\infty} \sup_t|\partial_t^l\epsilon''_{j_i}(t)| = 0 \qquad \forall l.
\]
since $J_k\to J$ in $C^\infty$ and $c_{j_i}\to c$.
\end{proof}

From now on we denote
\begin{equation*}
\begin{aligned}
& M_k(s,t) = M_k(w_k(s,t)) & M(s,t) = M(w(s,t)) \\
& J_k(s,t) = J_k(w_k(s,t)) & J(s,t) = J(w(s,t)) \\
& D_k(s,t) = D_k(w_k(s,t)) & D(s,t) = D(w(s,t))
\end{aligned}
\end{equation*}
without fear of ambiguity. Consider
\begin{equation}
\begin{aligned}
& \zeta_k(s,t) = M_k(s,t)z_k(s,t), \qquad \Lambda_k(s,t) = (M_kS_k-\partial_sM_k-J_0\partial_tM_k)M_k^{-1} \\
\end{aligned}
\end{equation}
Then $\zeta_k$ satisfies
\begin{equation}
\partial_s\zeta_k+J_0\partial_t\zeta_k+\Lambda_k\zeta_k=0.
\end{equation}
The following statement is a consequence of the previous lemma.

\begin{corollary}\label{coro_asymptotic_data}
If we set
\[
\Lambda^\infty(t) = (M^\infty S^\infty-J_0\partial_tM^\infty)(M^\infty)^{-1}
\]
then for all pairs of sequences $k_j,s_j \to +\infty$ there exists $j_i \to +\infty$ and a number $c\in[0,1)$ such that
\begin{equation*}
\lim_{i\to+\infty} \|D^\beta[\Lambda_{k_{j_i}}-\Lambda^\infty(t+c)](s_{j_i},\cdot)\|_{L^\infty(\R/\Z)} = 0
\end{equation*}
holds for every $\beta$.
\end{corollary}

We follow~\cite{props3} closely. With $N\in\N$, $l \in \N\cup\{0\}$, $a\in(0,1)$ and $d\in(-\infty,0)$ fixed, a function $F:[s_0,+\infty)\times\R/\Z \to \R^N$ is said to be of class $C^{l,a,d}_0$ if $F$ is of class $C^{l,a}_{\rm loc}$ and
\[
\lim_{R\to+\infty} \|e^{-d s}D^\beta F\|_{C^{0,a}([R,+\infty)\times\R/\Z)} = 0 \qquad \forall \beta \ \text{with} \ |\beta|\leq l.
\]
The space of such functions becomes a Banach space with the norm
\[
\|F\|_{C^{l,a,d}_0} = \|e^{-d s}F\|_{C^{l,a}([s_0,+\infty)\times\R/\Z)}.
\]
It follows from Theorem~\ref{thm_asymptotic_formula_deg_case} and the definition of $\M_{\jtil_k,0,\delta}(\gamma_1,\dots,\gamma_n;\emptyset)$ that for every $k$ the function $z_k(s,t)$ is of class $C^{l,a,\delta}_0$ for every $l\geq0$ and every $a\in(0,1)$. Lemma~\ref{lemma_asymptotic_data} implies that $\zeta_k(s,t)$ is also of class $C^{l,a,\delta}_0$. Moreover,
\begin{equation}
\zeta_k(s,t) \to \zeta(s,t) := M(s,t)z(s,t) \ \text{in} \ C^\infty_{\rm loc}.
\end{equation}
The crucial step is now to apply the following result.

\begin{proposition}
{\rm (\cite[Proposition~4.15]{elliptic})}
\label{prop_unif_asymptotic_analysis}
Let $$ K_n : [0,+\infty)\times \R/\Z \rightarrow \R^{2N\times 2N} \ (n\geq 1) \qquad K^\infty : \R/\Z \rightarrow \R^{2N\times 2N} $$ be smooth maps satisfying:
\begin{enumerate}
\item[i)] $K^\infty(t)$ is symmetric $\forall t$.
\item[ii)] For every pair of sequences $n_l,s_l\to+\infty$ there exist $l_k\to+\infty$ and $c\in[0,1)$ such that
\[
\lim_{k\to\infty} \|D^\beta[K_{n_{l_k}}-K^\infty(t+c)](s_{l_k},\cdot)\|_{L^\infty(\R/\Z)} = 0 \ \ \ \forall \beta.
\]
\end{enumerate}
Consider the unbounded self-adjoint operator $L$ on $L^2(\R/\Z,\R^{2N})$ defined by
\[
Le = -J_0\dot{e} - K^\infty e.
\]
With $l\geq 1$, $a\in(0,1)$ and $d\in(-\infty,0)$, suppose that the $\R^{2N}$-valued smooth maps $X_n(s,t)$ are of class $C^{l,a,d}_0$ and satisfy
\begin{equation}\label{eqn_X_n_prop}
\partial_s X_n + J_0\partial_t X_n + K_n X_n = 0  \quad \forall n.
\end{equation}
If $d$ does not belong to the spectrum of $L$ and the sequence $\{X_n\}$ is $C^\infty_{loc}$-bounded then $\{X_n\}$ has a convergent subsequence in $C^{l,a,d}_0$.
\end{proposition}

Direct calculations show that $\Lambda^\infty(t)$ is symmetric for every $t$, and that the operator $-J_0\partial_t-\Lambda^\infty$ is nothing but a representation of the asymptotic operator at~$\gamma_m$ associated to $(\alpha,J)$. It follows that $\delta_m$ is not in its spectrum. Corollary~\ref{coro_asymptotic_data} allows us to apply the above proposition with $d=\delta_m$ and conclude that $\zeta_k$ has a convergent subsequence in $C^{l,a,\delta_m}_0$. It follows that $\zeta(s,t)$ is of class $C^{l,a,\delta_m}_0$. Hence, also $z(s,t)$ is of class $C^{l,a,\delta_m}$. Since $l\geq1$ is arbitrary, we get
\begin{equation}
\label{strong_decay_z}
\lim_{s\to+\infty} \sup_{t\in\R/\Z} e^{-\delta_ms}|D^\beta z(s,t)| = 0 \qquad \forall\beta.
\end{equation}
The fact that $\zeta_k\to\zeta$ in $C^{l,a,\delta_m}_0$ implies that $z_k\to z(s,t)$ in  $C^{l,a,\delta_m}_0$ by Lemma~\ref{lemma_asymptotic_data}. Using that $l$ is arbitrary we conclude from the definition of the spaces $C^{l,a,\delta_m}_0$ that
\begin{equation}
\label{uniform_control_z_k}
\lim_{s\to+\infty} \sup_{t,k} e^{-\delta_ms}|D^\beta z_k(s,t)| = 0 \qquad \forall\beta.
\end{equation}

From this point it is quite standard to use~\eqref{strong_decay_z} and~\eqref{uniform_control_z_k} together with equations~\eqref{cauchy_riemann_1}-\eqref{cauchy_riemann_2} conclude that $z_{\infty,m}$ is a non-degenerate puncture of $\util$. For the sake of completeness we provide the details, however see~\cite{openbook,elliptic}. By Theorem~\ref{thm_asymptotic_formula_deg_case}, for every~$k$ we find $d_k,\tau_k \in \R$ such that
\begin{equation*}
V_k(s,t) = \begin{pmatrix} a_k(s,t)-T_ms-d_k \\ \tilde\theta_k(s,t)-t-\tau_k \end{pmatrix}
\end{equation*}
satisfies
\[
\lim_{s\to+\infty} \sup_{t} |V_k(s,t)| = 0 \qquad \forall k.
\]
Equations~\eqref{cauchy_riemann_1} tell us that
\begin{equation}\label{CR_type_equation_V_k}
\partial_sV_k(s,t)+\begin{pmatrix} 0 & -T_m \\ T_m^{-1} & 0 \end{pmatrix}\partial_tV_k(s,t) + B_k(s,t)z_k(s,t) = 0
\end{equation}
for some $B_k(s,t)$ satisfying
\[
\limsup_{s\to+\infty} \sup_{t,k} |D^\beta B_k(s,t)|<\infty \qquad \forall\beta.
\]
Together with~\eqref{uniform_control_z_k} we get
\begin{equation}\label{control_error_right}
\lim_{s\to+\infty} \sup_{t,k} e^{-\delta_ms}|D^\beta [B_kz_k](s,t)|=0 \qquad \forall\beta.
\end{equation}

\begin{lemma}
\label{lemma_not_proved_before}
There exists $r>0$ such that 
$$ 
\lim_{s\to+\infty} \sup_{t,k} e^{rs}|D^\beta V_k(s,t)| = 0 \qquad \forall D^\beta = \partial_s^{\beta_1}\partial_t^{\beta_2} \ \text{with} \ |\beta|\geq 1.
$$
\end{lemma}

\begin{proof}
This is a version for sequences of~\cite[Lemma~6.3]{fast}.
Fix any $\beta_0$ and denote $U_k=D^{\beta_0}\partial_tV_k$. Lemma~\ref{lemma_first_estimate} implies that
\[
\lim_{s\to+\infty} \sup_{k,t} |D^\beta U_k(s,t)| = 0 \qquad \forall\beta.
\]
There is no loss of generality to assume $T_m=1$. Differentiating~\eqref{CR_type_equation_V_k} we arrive at
\[
\partial_sU_k+J_0\partial_tU_k(s,t) = h_k(s,t).
\]
Moreover,~\eqref{control_error_right} gives
\begin{equation}\label{estimates_derivatives_h_k}
\lim_{s\to+\infty} \sup_{t,k}e^{-\delta_ms}|D^{\beta'} h_k(s,t)| = 0 \qquad \forall \beta'.
\end{equation}
Let $g_k(s) = \frac{1}{2} \|U_k(s,\cdot)\|^2_{L^2(\R/\Z)}$. Compute $$ g_k'(s) = \left< -J_0\partial_tU_k(s,\cdot)+h_k(s,\cdot),U_k(s,\cdot) \right>_{L^2} $$ and
\[
\begin{aligned}
g_k''(s) &= \left< -J_0\partial_s\partial_tU_k, U_k \right>_{L^2} + \left< \partial_sh_k, U_k \right>_{L^2} + \|-J_0\partial_tU_k+h_k\|_{L^2}^2 \\
&= \left< -J_0\partial_tU_k+h_k, -J_0\partial_tU_k \right>_{L^2} + \left< \partial_sh_k, U_k \right>_{L^2} + \|-J_0\partial_tU_k+h_k\|_{L^2}^2 \\
&= 2\|\partial_tU_k\|^2_{L^2} + 3\left< -J_0\partial_tU_k,h_k \right>_{L^2} + \left< \partial_sh_k, U_k \right>_{L^2} + \|h_k\|_{L^2}^2 \\
&\geq 2\|U_k\|^2_{L^2} - 3\|U_k\|_{L^2}\|h_k\|_{L^2} - \|U_k\|_{L^2}\|\partial_sh_k\|_{L^2} - \|h_k\|_{L^2}^2
\end{aligned}
\]
where in the last inequality we used Poincar\'e inequality ($U_k$ has average zero in $t$) and Cauchy-Schwarz inequality. Fix $0<d<\min\{\frac{1}{2},-\delta_m\}$. By the ``Peter-Paul'' inequality with $\epsilon>0$, and~\eqref{estimates_derivatives_h_k}, we get
\[
\begin{aligned}
& g''_k(s) \geq (2-\epsilon) \|U_k\|^2_{L^2} - C(\epsilon)(\|\partial_sh_k\|_{L^2}^2 + \|h_k\|_{L^2}^2) \geq 2(2-\epsilon) g_k(s) - ce^{-2ds}
\end{aligned}
\]
for some $c>0$ that depends on $\epsilon$ but does not depend on $k$. Choosing $\epsilon$ small we get
\[
g''_k(s) \geq g_k(s) - ce^{-2ds}
\]
uniformly on $k$. Choose $0<\nu<2d$ ($<1$), set $L=c/(4d^2-\nu^2)$ and consider $f_k := g_k+Le^{-2ds}$. Note that $f_k(s)>0$ since $g_k(s)\geq 0$, $\forall s$. The above differential inequality translates to
\[
f''_k \geq \nu^2f_k
\]
where $\nu$ is independent of $k$. Note that $f_k(s) \to 0$ as $s\to+\infty$, $\forall k$.

We claim that $f'_k(s)\leq0$ for all $k$ and $s\gg1$. For if not then we could find $s_*\gg1$ and $k$ such that $f'_k(s_*)>0$. Consider $G_k = f'_k+\nu f_k$, so that $G_k(s_*)>0$. Let $$ \bar s = \sup\{s\geq s_* \mid G_k(y)>0 \ \forall y\in[s_*,s]\} \in (s_*,+\infty]. $$ The differential inequality $G'_k\geq \nu G_k$ implies that $$ \text{$G_k(s) \geq G_k(s_*)e^{\nu(s-s_*)}\geq G_k(s_*)>0$ for every $s\in(s_*,\bar s)$.} $$ In particular $\bar s=+\infty$ and $G_k(s) \to +\infty$ as $s\to+\infty$, forcing $f'_k(s)\to+\infty$ as $s\to+\infty$ since $f_k(s) \to 0$ as $s\to+\infty$. This is an obvious contradiction to $\lim_{s\to+\infty}f_k(s)=0$.

Consider now $H_k = f_k-\nu^{-1} f_k'$. Since $f'_k(s)\leq0$ and $f_k(s)>0$, we conclude that $H_k(s) \geq f_k(s) >0 \ \forall s\geq s_0\gg1$. The differential inequality satisfied by $f_k$ implies that $H'_k\leq-\nu H_k$, from where it follows that $$ g_k(s)\leq f_k(s)\leq H_k(s) \leq H_k(s_0)e^{-\nu(s-s_0)} \qquad \text{on $[s_0,+\infty)$.} $$

Fix any $0<r<\nu$. Since $H_k(s_0)$ has a limit as $k\to+\infty$, and $\beta_0$ was chosen arbitrarily, we get
\[
\lim_{s\to+\infty} e^{rs} \sup_k \int_{\R/\Z} |D^\beta V_k(s,t)|^2dt = 0
\]
for all $\beta=(\beta_1,\beta_2)$ satisfying $\beta_2\geq1$.
The (trivial) Sobolev inqualities for functions on $\R/\Z$ yield
\[
\lim_{s\to+\infty} e^{rs}\sup_{k,t}|D^\beta V_k(s,t)| = 0 \qquad \forall\beta=(\beta_1,\beta_2) \ \text{satisfying} \ \beta_2\geq1
\]
as desired. The lemma now follows from this and~\eqref{CR_type_equation_V_k}.
\end{proof}

Using~\eqref{CR_type_equation_V_k} together with the previous lemma we get
\[
\lim_{s\to+\infty} e^{rs} \sup_{k,t} |\partial_sV_k(s,t)| = 0.
\]
Equivalently
\[
\left\{
\begin{aligned}
& \lim_{s\to+\infty} e^{rs} \sup_{k,t} |\partial_sa'_k(s,t)-T_m| = 0 \\
& \lim_{s\to+\infty} e^{rs} \sup_{k,t} |\partial_s\tilde\theta_k(s,t)| = 0
\end{aligned}
\right.
\]
Passing to the limit as $k\to+\infty$ we achieve
\[
\left\{
\begin{aligned}
& \lim_{s\to+\infty} e^{rs} \sup_{t} |\partial_sa(s,t)-T_m| = 0 \\
& \lim_{s\to+\infty} e^{rs} \sup_{t} |\partial_s\tilde\theta(s,t)| = 0
\end{aligned}
\right.
\]
In view of the exponential factor we conclude that the following integrals converge
\[
\begin{aligned}
d = a(s_0,0)-T_ms_0 + \int_{s_0}^{+\infty} (\partial_sa(s,0)-T_m) ds \\
\tau = \tilde\theta(s_0,0) + \int_{s_0}^{+\infty} \partial_s\tilde\theta(s,0) ds
\end{aligned}
\]
Now the constants $c,\tau$ satisfy
\[
\lim_{s\to+\infty} \sup_t |a(s,t)-T_ms-d| \to 0 \qquad \lim_{s\to+\infty} \sup_t |\tilde\theta(s,t)-t-\tau| \to 0
\]
We are finally done with showing that $z_{\infty,m}$ is a non-degenerate puncture of $\util$. From~\eqref{strong_decay_z} we see that the asymptotic eigenvalue at this puncture, which exists by Theorem~\ref{thm_asymptotic_formula_deg_case}, is $\leq\delta_m$. Since $m\in\{1,\dots,n\}$ is arbitrary, the proof of Lemma~\ref{lemma_non_deg_punctures_missing} is complete.

As explained before, this completes the proof of Proposition~\ref{prop_abstract_compactness}.

\subsubsection{Proof of (iii) $\Rightarrow$ (ii) in Theorem~\ref{main2}}

Consider $g_k:M\to(0,+\infty)$ smooth functions such that
\begin{itemize}
\item $g_k\to 1$ in $C^\infty$
\item $g_k(p)=1$ and $dg_k(p)=0$ for all $p\in L$
\item $g_k\alpha$ is non-degenerate for every $k$
\end{itemize}
Fix $J:\xi\to\xi$ a $d\alpha$-compatible complex structure, define $\jtil_k$ and $\jtil$ by~\eqref{formula_J_tilde} using $(g_k\alpha,J)$ and $(\alpha,J)$, respectively.

Using Proposition~\ref{prop_main_existence_non_deg} we get, for every $k$ large enough, embedded curves in $\M_{\jtil_k,0,\delta}(\gamma_1,\dots,\gamma_n;\emptyset)$ inside $\R\times (M\setminus L)$ representing the class $b \in H_2(M,L)$ of a page of $\Theta$. By Proposition~\ref{prop_abstract_compactness} we get an embedded curve $C_* \in \M_{\jtil,0,\delta}(\gamma_1,\dots,\gamma_n;\emptyset)$ contained in $\R\times (M\setminus L)$ that represents $b$.

Let $X$ denote the connected component of $\M_{\jtil,0,\delta}(\gamma_1,\dots,\gamma_n;\emptyset)/\R$ that contains~$C_*$. Since $\alpha$ is possibly degenerate, we do not have at our disposal the statements from~\cite{siefring} to study geometric properties of curves in $X$. We have to provide explicit arguments to conclude that curves in $X$ are embedded, are contained in $\R\times (M\setminus L)$ and their projections to $M$ do not intersect. We start by noting that curves in $X$ are immersed and transverse to the Reeb vector field; this is the content of Lemma~\ref{lemma_wind_infinity}.

\begin{lemma}
\label{lemma_curves_in_X_are_embedded}
Every curve in $X$ is embedded.
\end{lemma}

\begin{proof}
The Fredholm theory described in Lemma~\ref{lemma_aut_transv} does not use non-degeneracy of the contact form, it only depends on the exponential weights and on the assumption that the reference curve is embedded. 
Note that in~\ref{sssec_moduli_spaces} we defined our moduli spaces to be spaces of curves with non-degenerate behaviour at the punctures; see~Definition~\ref{def_non_deg_punctures}. 
Thus, the neighboring curves in $\M_{\jtil,0,\delta}(\gamma_1,\dots,\gamma_n;\emptyset)$ of a fixed embedded curve $C_0\in X$ foliate a neighborhood of $C_0$ in $\R\times M$, in particular they do not intersect each other. 
This was explained in the sketch of Lemma~\ref{lemma_aut_transv}. 
Moreover, it follows that the set of curves in $X$ that are embedded form an open subset of $X$. 
Now we argue to show that this set is also closed in $X$. 
If not then we would have a sequence of embedded curves $C_k\in X$ converging to a curve $C_\infty\in X$ that is not embedded. Above we saw that $C_\infty$ is immersed. 
Hence $C_\infty$ must have self-intersections, and these must be isolated since curves in $\M_{\jtil,0,\delta}(\gamma_1,\dots,\gamma_n;\emptyset)$ are somewhere injective; Lemma~\ref{lemma_somewhere_injective} does not use a non-degeneracy assumption on the contact form.
Positivity and stability of intersections provide self-intersections of $C_k$ for large $k$, and this is absurd. 
Since $X$ is connected and contains the embedded curve $C_*$, we conclude that all curves in $X$ are embedded.
\end{proof}

Let $C\in X$ be represented by the map $\util=(a,u):\C\setminus \Gamma \to \R\times M$ defined on a punctured plane. 
The map $\util$ is pseudo-holomorphic and has finite-energy. 
The set of punctures of $\util$ is $\Gamma\cup\{\infty\}$. 
We saw above that $\util$ is an embedding. 
Suppose that for some $c\neq 0$ we have
\[
F_c := \{(z,\zeta) \in (\C\setminus\Gamma) \times (\C\setminus\Gamma) \mid \util(z) = \util_c(\zeta)\} \neq\emptyset.
\]
Here we denoted by $\util_c=(a+c,u)$ the $\R$-translation of $\util$ by $c$. Note that $(z,\zeta)\in F_c$ and $c\neq 0$ together imply that $z\neq\zeta$. 
Without loss of generality assume $c>0$. 
We claim that for every $0<\epsilon<c$ there exists a compact set $K\subset (\C\setminus\Gamma) \times (\C\setminus\Gamma)$ containing $\cup_{t\in[\epsilon,c]}F_t$. 
For if not then we find $t_n\in[\epsilon,c]$ and $(z_n,\zeta_n)\in F_{t_n}$ such that $z_n$ or $\zeta_n$ converges to $\Gamma\cup\{\infty\}$. 
It follows that $a(z_n) \to +\infty$ or $a(w_n) \to +\infty$. 
Hence both $a(z_n)$ and $a(w_n)$ converge $+\infty$ since $|a(z_n)-a(w_n)|=t_n$ is bounded. 
In other words, both sequences $z_n$ and $\zeta_n$ converge to $\Gamma\cup\{\infty\}$. Since asymptotic limits are geometrically distinct, both $z_n$ and $\zeta_n$ converge to the same puncture, say to $z_* \in \Gamma\cup\{\infty\}$. 
But, since asymptotic limits are simple orbits, we know that the $M$-component $u$ is injective around $z_*$, in contradiction to $u(z_n)=u(\zeta_n)$ and $z_n\neq\zeta_n$. We have proved existence of the desired compact set $K$. 
By positivity of intersections, the image of $\util$ intersects the image of $\util_\epsilon$ for every $\epsilon>0$. But when $\epsilon$ is small enough the images of the maps $\util_\epsilon$ show up in the local foliations around $C$ given by the Fredholm theory, hence they can not intersect $C$, absurd. 
It follows that $F_c=\emptyset$ for every $c\neq0$. 
There are two important consequences:
\begin{itemize}
\item $C \cap \R\times L = \emptyset$.
\item $u$ is injective.
\end{itemize}
In particular, since we already know that $C$ projects to an immersed surface in $M$, the closure of the projection of $C$ to $M$ is a Seifert surface for $L$ representing the class $b$.

We need to argue that if $C,C'$ are distinct curves in $X$ then their projections to $M$ do not intersect.
Represent $C,C'$ as finite-energy maps
\[
\begin{aligned}
\util=(a,u) : (\C P^1\setminus\Gamma,i) &\to (\R\times M,\jtil) \\
\util'=(a',u') : (\C P^1\setminus\Gamma',i) &\to (\R\times M,\jtil) \\
\end{aligned}
\]
respectively, where $\Gamma$ and $\Gamma'$ are finite sets of positive punctures. 
Let
\[
E_c = \{z \in \C P^1\setminus\Gamma \mid \util(z)=\util'_c(z') \ \text{for some} \ z'\in \C P^1\setminus\Gamma' \}.
\]
and consider the set
\begin{equation}\label{set_A_intersections}
A := \{ c\in \R \mid E_c \neq \emptyset \}.
\end{equation}

\begin{lemma}\label{lemma_int_trans_below}
$A=\emptyset$.
\end{lemma}

The proof will be presented below. Lemma~\ref{lemma_int_trans_below} is equivalent to saying that the projections of $C$ and $C'$ to $M$ do not intersect. 
We conclude that the projections of curves in $X$ foliate $M\setminus L$ by surfaces which are transverse to the Reeb flow, forming an planar open book. It remains only to argue that pages are global surfaces of section; this argument was already explained at the end of subsection~\ref{ssec_approximating_sequences}. By transversality of the pages with $X_\alpha$, if the $\omega$-limit set of a point in $M\setminus L$ does not intersect $L$ then the future trajectory of this point will hit every page infinitely often. If the $\omega$-limit set of a point in $M\setminus L$  intersects $L$ then the future trajectory of this point will spend arbitrarily large amounts of time arbitrarily close to $L$, where it can be well controlled by the linearized dynamics along components of $L$. Assumption (iiia) in Theorem~\ref{main2} guarantees that this trajectory hits every page infinitely often. The argument for past times is analogous. We are finally done proving (iii) $\Rightarrow$ (ii) in Theorem~\ref{main2}, up to presenting proof of Lemma~\ref{lemma_int_trans_below}.

\begin{proof}[Proof of Lemma~\ref{lemma_int_trans_below}]
The proof is accomplished by establishing two contradictory claims.

\medskip

\noindent {\it Step 1.} If $A \neq\emptyset$ then $\inf A >-\infty$.

\medskip

\noindent {\it Step 2.} If $A \neq \emptyset$ then $\inf A = -\infty$.

\medskip

Consider a sequence $c_k \in A$ such that $$ c_k\to c = \inf A \in [-\infty,+\infty). $$ We can find sequences $z_k \in \C P^1\setminus\Gamma$, $z'_k \in \C P^1\setminus\Gamma'$ such that $\util(z_k) = \util'_{c_k}(z'_k)$. We claim that either $(z_k,z'_k)$ is compactly contained in $\C P^1\setminus\Gamma \times \C P^1\setminus\Gamma'$ or, up to a subsequence, we can assume that $(z_k,z'_k)$ converges to a point in $\Gamma \times \Gamma'$. In fact, if $z_k$ is compactly contained in $\C P^1\setminus\Gamma$ and $z'_k$ is not compactly contained in $\C P^1\setminus\Gamma'$ then we may assume $u(z_k) = u'(z'_k) \to L$, forcing $\util$ to intersect $\R\times L$, this is absurd. A similar argument handles the case where the roles of $z_k$ and $z'_k$ are interchanged.

Up to choice of a subsequence we can assume $(z_k,z'_k) \to (z_*,z_*') \in \C P^1 \times \C P^1$. By what was proved above either $(z_*,z_*') \in \C P^1\setminus\Gamma \times \C P^1\setminus\Gamma'$ or $(z_*,z_*') \in \Gamma \times \Gamma'$. \\

\noindent \textbf{Case 1.} $(z_*,z_*') \in \C P^1\setminus\Gamma \times \C P^1\setminus\Gamma'$. \\

We prove {\it Step 1} in this case. If $c=-\infty$ then the identity $a'(z'_k)+c_k = a(z_k) \geq \inf a > -\infty$ shows that $a'(z'_k) \to +\infty$, contradicting $a'(z'_k) \to a'(z'_*) \in \R$.

Now we prove {\it Step 2} in this case. If $c>-\infty$ then $\util(z_*) = \util'_c(z'_*)$. The point $(z_*,z'_*)$ is isolated as an intersection point of $\util$ with $\util'_c$; otherwise we would use the similarity principle to conclude that $C=C'$ (these curves are embedded). By stability and positivity of intersections we conclude that $E_y \neq \emptyset$ for some $y<c$, in contradiction to the definition of $c$.

\medskip

\noindent \textbf{Case 2.} $(z_*,z_*') \in \Gamma \times \Gamma'$. \\


At the punctures $z_*,z_*'$ the curves $\util,\util'$ must be asymptotic to the same periodic orbit $\gamma_l=(x_l,T_l)$ since all asymptotic limits of both curves are geometrically distinct. 
For the $M$-components we get $u(z_k)=u'(z'_k)$ since the asymptotic limits are simply covered.

Choose asymptotic representatives $(\phi,U)$ and $(\phi',U')$ of $\util$ and $\util'$ at $z_*$ and $z_*'$, respectively. The existence of asymptotic representatives makes use of Theorem~\ref{thm_asymptotic_formula_deg_case}. This means, see~\cite{siefring_CPAM}, we find $R,R'\gg1$, compact neighborhoods $D_*,D_*'$ of $z_*,z_*'$, smooth maps
\[
\begin{aligned}
U : [R,+\infty)\times\R/\Z &\to \xi \qquad U(s,t) \in \xi|_{x_l(T_lt)} \\
U' : [R',+\infty)\times\R/\Z &\to \xi \qquad U'(s,t) \in \xi|_{x_l(T_lt)}
\end{aligned}
\]
and diffeomorphisms
\[
\begin{aligned}
\phi : [R,+\infty)\times\R/\Z &\to D_*\setminus\{z_*\} \\
\phi' : [R',+\infty)\times\R/\Z &\to D_*'\setminus\{z_*'\}
\end{aligned}
\]
satisfying
\begin{equation*}
\begin{aligned}
\util(\phi(s,t)) &= (T_ls,\exp_{x_l(T_lt)}U(s,t)) \\
\util'(\phi'(s,t)) &= (T_ls,\exp_{x_l(T_lt)}U'(s,t)).
\end{aligned}
\end{equation*}
Here $\exp$ is the exponential map given by an arbitrarily chosen auxiliary Riemannian metric on $M$. Then for every $\tau \in \R$ it follows that $(\phi_\tau',U_\tau')$ is an asymptotic representative of $\util'_\tau$, where
\[
\begin{aligned}
& \phi_\tau'(s,t) = \phi'(s-\tau/T_l,t) \\
& U_\tau'(s,t) = U'(s-\tau/T_l,t).
\end{aligned}
\]
The choice of parametrization $x_l(t)$ is fixed by the convention made in Remark~\ref{rem_marked_point_on_orbits}. The sequences $z_k\to z_*$, $z'_k\to z'_*$ correspond to sequences $(s_k,t_k)$, $(s_k',t_k')$ satisfying $\min\{s_k,s'_k\}\to+\infty$ and
\begin{equation*}
\begin{aligned}
(T_ls_k,\exp_{x_l(T_lt_k)}U(s_k,t_k)) &= \util(\phi(s_k,t_k)) \\
&= \util_{c_k}(\phi'_{c_k}(s'_k,t'_k)) \\
&= (T_ls_k',\exp_{x_l(T_lt_k')}U'(s_k'-c_k/T_l,t_k'))
\end{aligned}
\end{equation*}
For the $\R$-components we get $s_k=s_k'$. Now using that $\exp$ determines a diffeomorphism between a neighborhood of the zero section of $\xi_{\gamma_l}$ onto a neighborhood of $x_l(\R)$, the identity $$ \exp_{x_l(T_lt_k)}U(s_k,t_k)=\exp_{x_l(T_lt_k')}U'(s_k'-c_k/T_l,t_k') $$ forces $t_k=t_k'$ and
\begin{equation}\label{intersection_points_crucial}
U(s_k,t_k) = U'(s_k-c_k/T_l,t_k) = U'_{c_k}(s_k,t_k).
\end{equation}

Let $A$ denote the asymptotic operator on sections of $\xi_{\gamma_l} = x_l(T_l\cdot)^*\xi$ induced by $(\alpha,J)$. Theorem~\ref{thm_asymptotic_formula_deg_case} implies that
\begin{equation}
\label{asymptotic_reps_Us}
\begin{aligned}
& U(s,t) = e^{\lambda s}(v(t)+r(s,t)) \\
& U'(s,t) = e^{\lambda' s}(v'(t)+r'(s,t)) \\
& \lim_{s\to+\infty}\sup_t \ (|r(s,t)|+|r'(s,t)|)=0
\end{aligned}
\end{equation}
where $\lambda,\lambda'<0$ are eigenvalues of $A$ and $v,v'$ are corresponding eigensections. Moreover, both $v$ and $v'$ have winding number equal to zero in a symplectic trivialization aligned to the normal of a page of $\Theta$; this follows from Lemma~\ref{lemma_wind_infinity}. There are two cases to be considered. \\

\noindent {\it Subcase 2.1.} $\lambda=\lambda'$. \\

From~\eqref{intersection_points_crucial} and~\eqref{asymptotic_reps_Us} we obtain
\[
\begin{aligned}
0 &= e^{\lambda s_k}(v(t_k)+r(s_k,t_k))-e^{\lambda(s_k-c_k/T_l)}(v'(t_k)+r'(s_k-c_k/T_l,t_k)) \\
&= e^{\lambda s_k}(v(t_k)-e^{-\lambda c_k/T_l}v'(t_k) + r(s_k,t_k) - e^{-\lambda c_k/T_l} r'(s_k-c_k/T_l,t_k)).
\end{aligned}
\]
Dividing by $e^{\lambda s_k}$ and letting $k\to\infty$ we obtain that $v(t)$ and $e^{-\lambda c/T_l}v'(t)$ coincide for some $t$, where the latter is zero if $c=-\infty$.
Since they solve the same linear ODE we get 
\begin{equation}
\label{critical_evector}
v(t) = \begin{cases} e^{-\lambda c/T_l}v'(t) \ \forall t & \text{if} \ c>-\infty \\ 0 \ \forall t & \text{if} \ c=-\infty \end{cases}
\end{equation} 
We can now prove {\it Step~1} in Subcase 2.1. 
If $c=-\infty$ we get a contradiction from~\eqref{critical_evector} to the fact that $v(t)$ is nowhere vanishing. 

\medskip

We proceed assuming that $c>-\infty$ and work towards the proof of {\it Step 2} in Subcase 2.1. 
Below we will show that for $\hat{c_k} \in A$ holds for any sequence $\hat{c}_k \to c$, $\hat{c}_k < c$, in contradiction to the definition of $c$. 
Using~\eqref{critical_evector} we obtain the following important fact:
\begin{equation}
\label{estimate_subcase_21}
\liminf_{s\to+\infty} \ e^{-\lambda s}|U(s,t)-U'_\tau(s,t)| \ > \ 0 \qquad \forall \tau\neq c.
\end{equation}
To see this we plug the identity $v'(t) = e^{\lambda c/T}v(t)$, which we get from~\eqref{critical_evector} in the case $c>-\infty$, and compute for $\tau\neq c$:
\begin{equation}
\label{asymp_form_rel_notcrit}
\begin{aligned}
& U(s,t)-U'_{\tau}(s,t) \\
&= e^{\lambda s}(v(t)+r(s,t))-e^{\lambda(s-\tau/T_l)}(e^{\lambda c/T_l}v(t)+r'(s-\tau/T_l,t)) \\
&= e^{\lambda s}((1-e^{\lambda(c-\tau)/T_l})v(t) + r(s,t) - e^{-\lambda \tau/T_l} r'(s-\tau/T_l,t)) \, .
\end{aligned}
\end{equation}
The hypothesis $\tau\neq c$ implies that the coefficient $1-e^{\lambda(c-\tau)/T_l}$ in front of $v(t)$ does not vanish, and~\eqref{estimate_subcase_21} follows.

However, for the special value $\tau=c$ we get
\begin{equation*}
\begin{aligned}
& U(s,t)-U'_c(s,t) \\
&= e^{\lambda s}(v(t)+r(s,t)) - e^{\lambda (s-c/T_l)}(e^{\lambda c/T_l}v(t)+r'(s-c/T_l,t)) \\
&= e^{\lambda s}(r(s,t)-e^{-\lambda c/T_l}r'(s-c/T_l,t))
\end{aligned}
\end{equation*}
from where we conclude that
\begin{equation*}
e^{-\lambda s}|U(s,t)-U'_c(s,t)| \to 0 \ \text{uniformly in $t$ as} \ s\to+\infty.
\end{equation*}
Theorem~\ref{thm_asymptotic_formula_deg_case} implies that we can apply~\cite[Theorem~2.2]{siefring_CPAM} to write
\begin{equation}
\label{asympt_formula_rel_difference}
\begin{aligned}
& U(s,t)-U'_c(s,t) = e^{\nu s}(w(t)+\hat r(s,t)) \\
& \lim_{s\to+\infty}\sup_t |\hat r(s,t)|=0
\end{aligned}
\end{equation}
where $\nu$ is an eigenvalue of $A$ satisfying $\nu<\lambda$, and $w(t)$ is a corresponding eigensection.
We claim that $\wind(w) < \wind(v)$ with respect to some (hence any) homotopy class of symplectic trivializations of~$\xi_{\gamma_l}$. 
To see this, fix $s_*$ large enough so that $U-U'_c$ does not vanish on $[s_*,+\infty)\times\R/\Z$ and
\[
\begin{aligned}
\wind(U(s_*,\cdot)-U'_c(s_*,\cdot)) = \wind(w) 
\end{aligned}
\]
This is possible in view of~\eqref{asympt_formula_rel_difference}. 
Since $U'_{c_k} \to U'_c$ in $C^0_{\rm loc}$ we find $k_*$ such that
\begin{equation}
\label{winding_at_s_and_k_large}
k\geq k_* \Rightarrow \left\{ \begin{aligned} & U(s_*,\cdot)-U'_{c_k}(s_*,\cdot) \ \text{ does not vanish, and} \\ & \wind(U(s_*,\cdot)-U'_{c_k}(s_*,\cdot)) = \wind(w) \, .
\end{aligned} \right.
\end{equation}
The algebraic count of zeros of $U(s,t)-U'_{c_k}(s,t)$ on $[s_*,+\infty)\times\R/\Z$, with respect to obvious orientations, is equal to the intersection number of the embedded holomorphic cylinders
\[
\util\circ \phi([s_*,+\infty)\times\R/\Z) \qquad \text{and} \qquad \util'_{c_k}\circ \phi'_{c_k}([s_*,+\infty)\times\R/\Z) \, .
\]
This is an immediate consequence of standard degree theory. 
From the existence of the sequence $(s_k,t_k)$ we conclude that this intersection number is positive for $k$ large.
But~\eqref{asymp_form_rel_notcrit},~\eqref{asympt_formula_rel_difference}, standard degree theory, and the assumption that $c_k \neq c$ together imply that this intersection number is equal to
\[
0 < \wind(v) - \wind(U(s_*,\cdot)-U'_{c_k}(s_*,\cdot)) = \wind(v) -\wind(w)
\]
as claimed.

We are now in a position to finish the proof of {\it Step 2} in Subcase 2.1. 
Consider $\hat c_k \to c$ and satisfying $\hat c_k < c$ for all $k$. 
Arguing as above we get 
$$
\wind(U(s_*,\cdot)-U'_{\hat c_k}(s_*,\cdot)) = \wind(w) \qquad \text{for $k$ large enough.}
$$
As before we use this fact together with the argument principle and~\eqref{asymp_form_rel_notcrit} to conclude that the half-cylinder $\util\circ \phi([s_*,+\infty)\times\R/\Z)$ intersects the half-cylinder $\util'_{\hat c_k}\circ \phi'_{\hat c_k}([s_*,+\infty)\times\R/\Z)$ for $k$ large enough. 
Hence $\hat c_k \in A$, in contradiction to $c=\inf A$.

\medskip

\noindent {\it Subcase 2.2)} $\lambda\neq\lambda'$. 

\medskip

In this case
\begin{equation}
\label{pointwise_LI}
v'(t) \not\in\R v(t) \quad \forall t \, .
\end{equation}
This is a consequence of the following two facts:
\begin{itemize}
\item 
The winding numbers of $v$ and $v'$ in any symplectic frame are equal: both these winding numbers vanish on a trivialization aligned to the normal of a page of $\Theta$ (Lemma~\ref{lemma_wind_infinity}).
\item 
If eigensections associated to distinct eigenvalues have the same winding number (in any frame) then they are pointwise linearly independent.
\end{itemize}
From~\eqref{intersection_points_crucial} and~\eqref{asymptotic_reps_Us} we get
\begin{equation}
\label{basic_identity_to_work}
\begin{aligned}
0 &= e^{\lambda s_k}(v(t_k)+r(s_k,t_k)) - e^{\lambda'(s_k-c_k/T_l)}(v'(t_k)+r'(s_k-c_k/T_l,t_k)) \, .
\end{aligned}
\end{equation}
We claim that $(\lambda-\lambda')s_k+\lambda'c_k/T_l$ is a bounded sequence. 
If not then, up to choice of a subsequence, we may assume that it converges to $+\infty$ or to $-\infty$. 
In the former case we can divide~\eqref{basic_identity_to_work} by $e^{\lambda s_k}$ and get
\begin{equation*}
0 = v(t_k) + r(s_k,t_k) - e^{(\lambda'-\lambda)s_k-\lambda'c_k/T_l}(v'(t_k)+r'(s_k-c_k/T_l,t_k)) \, .
\end{equation*}
Since in present case the coefficient in front of $v'(t_k)+r'$ converges to zero, we conclude that $v(t)$ vanishes somewhere, and this is impossible.
The argument in the case $(\lambda-\lambda')s_k+\lambda'c_k/T_l \to -\infty$ is analogous, this time one divides~\eqref{basic_identity_to_work} by $e^{\lambda' (s_k-c_k/T_l)}$ and arrives at a zero of $v'(t)$. 
It follows from this claim that, up to a subsequence, we may assume that $e^{(\lambda-\lambda')s_k+\lambda'c_k/T_l} \to a$ for some $a\neq0$. 
Hence, up to a further subsequence, we can use~\eqref{basic_identity_to_work} to find $t_*$ such that $av(t_*)=v'(t_*)$. 
This is a contradiction to~\eqref{pointwise_LI}.
\end{proof}

\section{Proof of Theorem~\ref{main1}}
\label{sec_main1_proof}

Theorem~\ref{main2} contains (ii) $\Rightarrow$ (i), and also (i) $\Rightarrow$ (iii) under the stated $C^\infty$-generic assumption. 
Only (iii) $\Rightarrow$ (ii) remains to be proved, and this will be accomplished by reducing to Theorem~\ref{main2} as follows.
Let $\gamma \subset M \setminus L$ be a closed Reeb orbit of a contact form as in Theorem~\ref{main1} satisfying (iii). 
We can consider non-degenerate $C^\infty$-small perturbations of the contact form keeping $L \cup \gamma$ as periodic Reeb orbits. 
By Proposition~\ref{prop_main_existence_non_deg_global} below, $L$ bounds a GSS in class $b$ for the perturbed Reeb flows, hence ${\rm int}(\gamma,b) \neq 0$ and we are now again under the assumptions of Theorem~\ref{main2}.

Consider contact forms $\alpha_+>\alpha_-$ on $M$ defining the same positive contact structure $\xi$, a null-homologous link $L$ transverse to $\xi$, and an oriented and positive Seifert surface $\Sigma$ for $L$ of genus zero. Orient $L$ as the boundary of $\Sigma$. Denote the connected components of $L$ by $\gamma_1,\dots,\gamma_n$. Assume that $n\geq2$. As in subsection~\ref{ssec_existence_compactness_curves}, we consider a conformal $d\alpha_\pm$-symplectic trivialization $\tau_\Sigma$ of $\xi|_L$ which respect to which a non-vanishing vector field in $T_{\gamma_i}\Sigma \cap \xi$ has winding number equal to zero, for every $i$.

We assume the existence of a global conformal $d\alpha_\pm$-symplectic trivialization $\tau_{\rm gl}$ of $\xi$. Using $\tau_{\rm gl}$ we can define $m_i =m(\gamma_i,\Sigma) \in\Z$ as the winding number of a non-vanishing
vector field in $T_{\gamma_i}\Sigma\cap \xi$ computed with respect to $\tau_{\rm gl}$. Note that
\begin{equation}\label{identity_minus_sl}
-{\rm sl}(L,\Sigma) = \sum_i m_i.
\end{equation}
With these numbers we define the interval $I(L,\Sigma,\tau_{\rm gl}) \subset \Z$ as in~\eqref{def_intervalo}. Consider also the interval $$ I_0(L,\Sigma,\tau_{\rm gl}) = [2c_0-2,2(C_0-\ell_0)+1] \subset I(L,\Sigma,\tau_{\rm gl}) $$ where $C_0$ is defined as in~\eqref{important_numbers},
$$
c_0 = \sum_{\{\gamma|m(\gamma,\Sigma)+1<0\}} (m(\gamma,\Sigma)+1) \ \ \mbox{ and } \ \
\ell_0 = \max\{0,\min_\gamma\{m(\gamma,\Sigma)+1\}\}.
$$
Below we will check that $c_0 = 2-C_0$. Hence
$$
I_0(L,\Sigma,\tau_{\rm gl}) \subset I(L,\Sigma,\tau_{\rm gl}), \qquad I_0(L,\Sigma,\tau_{\rm gl}) \neq I(L,\Sigma,\tau_{\rm gl}) \, .
$$

We assume that hypotheses (H1)-(H3) and (H5) described in subsection~\ref{ssec_existence_compactness_curves} are valid. In particular, each $\gamma_i$ is a closed Reeb orbit for both $\alpha_+$ and $\alpha_-$, both $\alpha_\pm$ orient each $\gamma_i$ as the boundary of $\Sigma$, and $X_{\alpha_+}$ is positively transverse to $\Sigma$. We replace (H4) by the following weaker hypothesis.
\begin{itemize}
\item[(H$4'$)] The contact forms $\alpha_+,\alpha_-$ are non-degenerate up to action $$ A = \sum_{k=1}^n \int_{\gamma_k}\alpha_+ $$ and the following hold:
\begin{itemize}
\item[($+$)] If $\gamma =(x,T)\in\mathcal{P}(\alpha_+)$ satisfies $T\leq A$ and $x(\R) \subset M\setminus L$ then it also satisfies ${\rm int}(\gamma,\Sigma) \neq 0$.
\item[($-$)] If $\gamma =(x,T)\in\mathcal{P}(\alpha_-)$ satisfies $T\leq A$, $x(\R) \subset M\setminus L$ and its Conley-Zehnder index satisfies $\mu_{\CZ}^{\tau_{\rm gl}}(\gamma) \in I_0(L,\Sigma,\tau_{\rm gl})$ then ${\rm int}(\gamma,\Sigma) \neq 0$.
\end{itemize}
\end{itemize}

Let $\jtil_\pm \in\J(\alpha_\pm)$ be fixed arbitrarily. As before, by (H3) we can choose numbers $\delta^\pm_1,\dots,\delta^\pm_n<0$ such that
\begin{itemize}
\item $\delta^+_k$ is in the spectral gap between eigenvalues with winding number $0$ and $1$ relative to $\tau_\Sigma$ of the asymptotic operator at $\gamma_k$ given by $(\alpha_+,J_+)$.
\item $\delta^-_k$ is in the spectral gap between eigenvalues with winding number $0$ and $1$ relative to $\tau_\Sigma$ of the asymptotic operator at $\gamma_k$ given by $(\alpha_-,J_-)$.
\end{itemize}
Set $\delta^\pm = (\delta^\pm_1,\dots,\delta^\pm_n;\emptyset)$. The main analytical fact we need to prove is the following statement analogous to Proposition~\ref{prop_main_compactness}.

\begin{proposition}\label{prop_main_compactness_global}
Assume (H1)-(H3), (H$4'$) and (H5). If there exists a curve in the moduli space $\M_{\jtil_+,0,\delta^+}(\gamma_1,\dots,\gamma_n;\emptyset)$ with projection to $M$ in the same class as $\Sigma$ in $H_2(M,L)$, then there exists a curve $C_-$ in $\M_{\jtil_-,0,\delta^-}(\gamma_1,\dots,\gamma_n;\emptyset)$ whose projection to $M$ is in the same class as $\Sigma$ in $H_2(M,L)$. Moreover, if $\mathcal{Y}$ denotes the connected component of $C_-$ in $\M_{\jtil_-,0,\delta^-}(\gamma_1,\dots,\gamma_n;\emptyset)$ then $\mathcal{Y}/\R$ is compact.
\end{proposition}

The main consequence of Proposition~\ref{prop_main_compactness_global} is a version of Proposition~\ref{prop_main_existence_non_deg} in the present set-up.

\begin{proposition}
\label{prop_main_existence_non_deg_global}
Let $\alpha$ be a contact form on $M$, and let $L=\gamma_1\cup\dots\cup\gamma_n$ be a link consisting of periodic Reeb orbits. Assume that $L$ binds a planar open book decomposition $\Theta$ that supports~$\xi=\ker\alpha$, and assume with no loss of generality that $\alpha$ and $\Theta$ induce the same orientation on each $\gamma_i$. Assume the existence of a global symplectic trivialization $\tau_{\rm gl}$ of $(\xi,d\alpha)$ and use it to define $I(L,{\rm page},\tau_{\rm gl}) \subset \Z$ as in~\eqref{def_intervalo}. 
Suppose that
\begin{itemize}
\item[(a)] For every $i$, $\mu_{\CZ}^\Theta(\gamma_i)\geq1$ as a prime periodic Reeb orbit of $\alpha$.
\end{itemize}
holds. 
Then there exists a constant $$ A > \sum_{i=1}^n \int_{\gamma_i} \alpha $$ that depends only on the triple $(\alpha,L,\Theta)$ and has the following significance. Fix any $d\alpha$-compatible complex structure $J:\xi\to\xi$, and for each $i$ choose $\delta_i<0$ in the spectral gap of the asymptotic operator induced by $(\alpha,J)$ on sections of $\xi$ along $\gamma_i$, between eigenvalues of winding number $0$ and $1$ with respect to a Seifert framing induced by pages of $\Theta$. Denote $\delta=(\delta_1,\dots,\delta_n;\emptyset)$. If 
\begin{itemize}
\item[(b)] Every $\gamma'=(x',T') \in \mathcal{P}(\alpha)$ such that $x'(\R) \subset M\setminus L$, $T'\leq A$ and $\mu_{\CZ}^{\tau_{\rm gl}}(\gamma')$ is in $I(L,{\rm page},\tau_{\rm gl})$ has non-zero algebraic intersection number with pages of~$\Theta$.
\end{itemize}
holds, then for any arbitrary sequence~$f_k\in \mathcal{F}_L$ satisfying
\begin{itemize}
\item $f_k\to 1$ in $C^\infty$, $f_k|_L \equiv 1$ for all $k$.
\item $f_k\alpha$ is non-degenerate for all $k$.
\end{itemize}
and $\jtil_k$ defined as in~\eqref{formula_J_tilde} using $f_k\alpha$ and $J$, there exists $k_0$ such that for every $k\geq k_0$ the link $L$ binds a planar open book decomposition whose pages are global surfaces of section for the Reeb flow of $f_k\alpha$, all pages represent the same class in $H_2(M,L)$ as the pages of $\Theta$, and are projections of curves in $\M_{\jtil_k,0,\delta}(\gamma_1,\dots,\gamma_n;\emptyset)$.
\end{proposition}

The proof of Proposition~\ref{prop_main_existence_non_deg_global} is exactly the same as the proof of Proposition~\ref{prop_main_existence_non_deg} given in subsection~\ref{ssec_approximating_sequences}, with Proposition~\ref{prop_main_compactness} replaced by Proposition~\ref{prop_main_compactness_global}. 
We leave details to the reader. 
The only novelty is to show that (b) in Proposition~\ref{prop_main_existence_non_deg_global} implies the following: If $f_k$ is as in the statement and $k$ is large enough, then every periodic Reeb orbit of $f_k\alpha$ in $M\setminus L$ with action $\leq A$ and Conley-Zehnder index $\mu_{\CZ}^{\tau_{\rm gl}}\in I_0(L,{\rm page},\tau_{\rm gl})$ has non-zero intersection number with $b$. 
This is a simple consequence of the way that Conley-Zehnder indices behave under limits.  
In fact, suppose by contradiction that we can find $k_j\to\infty$ and closed Reeb orbits $\gamma'_j \subset M \setminus L$ of $f_{k_j}\alpha$ with period $\leq A$, $\mu_{\CZ}^{\tau_{\rm gl}}(\gamma'_j) \in I_0(L,{\rm page},\tau_{\rm gl})$ and ${\rm int}(\gamma_j,b)=0$. 
Up to choice of a subsequence we may assume that $\gamma_j' \to \gamma'$ for some closed Reeb $\gamma'$ of $\alpha$ with period $\leq A$. 
If $\gamma'\subset L$ then $\mu_{\CZ}^{\Theta} = 0$ holds for some component of $L$, which is forbidden by our assumptions. 
Hence $\gamma' \subset M \setminus L$ and ${\rm int}(\gamma',b)=0$. 
Note that $I_0(L,{\rm page},\tau_{\rm gl}) = [x_0,y_0] \cap \Z$ where $x_0,y_0$ are even integers, and $I(L,{\rm page},\tau_{\rm gl}) = [x_0-1,y] \cap \Z$ where $y \in \{y_0,y_0+1\}$. 
By lower semi-continuity properties of Conley-Zehnder index, we know that $\mu_{\CZ}^{\tau_{\rm gl}}(\gamma') \in [x_0-2,y_0]$, but since $x_0$ is even we must have $\mu_{\CZ}^{\tau_{\rm gl}}(\gamma') \in [x_0-1,y_0] \cap \Z \subset I(L,{\rm page},\tau_{\rm gl})$. 
The existence of the orbit $\gamma'$ contradicts assumption (b).

\begin{proof}[Proof of Theorem~\ref{main1} assuming Proposition~\ref{prop_main_existence_non_deg_global}]
By Theorem~\ref{main2}, it suffices to show that if all periodic Reeb orbits in $M\setminus L$ satisfying $\mu^{\tau_{\rm gl}}_{\CZ} \in I(L,{\rm page},\tau_{\rm gl})$ have non-zero intersection number with $b$, then all periodic Reeb orbits in $M\setminus L$ have non-zero intersection number with $b$.

Let $\gamma=(x,T) \in \mathcal{P}(\alpha)$ satisfy $x(\R)\subset M\setminus L$ and assume by contradiction that ${\rm int}(\gamma,b)=0$. Choose a sequence $f_k\in\mathcal{F}_L$ that satisfies: (i) $f_k\to 1$ in $C^\infty$ and $f_k|_L \equiv 1$ for all $k$; (ii) $f_k\alpha$ is non-degenerate for all $k$; and (iii) $f_k|_{x(\R)}\equiv1$ and $df_k|_{x(\R)}\equiv0$ for all $k$. In particular, $\gamma \in \mathcal{P}(f_k\alpha)$ for all $k$. We can apply Proposition~\ref{prop_main_existence_non_deg_global} to conclude that $L$ bounds a global surface of section representing the class $b$ when $k$ is large enough. This forces ${\rm int}(\gamma,b)\neq 0$, which is a contradiction.
\end{proof}

The remainder of this section consists of the proof of Proposition~\ref{prop_main_compactness_global}.

\begin{lemma}\label{lemsl}
The following identity holds:
\begin{equation}\label{identitysum}
\sum_{i=1} ^n (m_i+1) = 2.
\end{equation}
\end{lemma}

\begin{proof}
Let $\mathcal{Z}$ be a section of $\xi|_\Sigma$ such that $\mathcal{Z}$ is non-vanishing and tangent to $\Sigma$ on a neighborhood of $\partial\Sigma = L$ in $\Sigma$. After a $C^\infty$-small perturbation we may assume that $\mathcal{Z}$ has only finitely many zeros, all of which are non-degenerate. By standard degree theory, the algebraic count of zeros of $\mathcal{Z}$ is equal to $\sum_{i=1}^n m_i$.

Now observe that ${\rm sl}(L,\Sigma)$ does not depend on the Seifert surface $\Sigma$ since $\xi$ is trivial. 
It does not depend either on the choice of global trivialization, while each individual term $m(\gamma,\Sigma)$ does. 
Moreover, by (H2) the Reeb vector field of $\alpha_+$ is transverse to $\Sigma \setminus \partial\Sigma$. 
In particular, the projection of $\mathcal{Z}$ to $\Sigma$ along $X_{\alpha_+}$ over points of $\Sigma \setminus \partial\Sigma$ is a vector field $\mathcal{Z}'$ on $\Sigma\setminus\partial\Sigma$ which coincides with $\mathcal{Z}$ near $L=\partial \Sigma$, and whose zeros (together with their signs) coincide with those of $\mathcal{Z}$. 
The Poincar\' e-Hopf theorem asserts that the algebraic count of zeros of $\mathcal{Z}'$ is the Euler characteristic of $\Sigma$, which is equal to $2-n$ since $\Sigma$ is a $n$-holed $2$-sphere. 
We conclude that $\sum_{i=1}^n m_i = 2-n$ as desired.
\end{proof}

\begin{corollary}
\label{coro_ell_0}
We have $0\leq \ell_0 \leq 1$ and $C_0\geq2$.
\end{corollary}

\begin{proof}
By definition $\ell_0\geq0$. Using the lemma and the definition of $\ell_0$ we get $$ \ell_0 \geq 2 \ \Rightarrow \ m_i+1\geq 2, \forall i \ \Rightarrow \ 2 = \sum_{i=1}^n (m_i+1) \geq 2n $$ contradicting our standing hypothesis $n\geq 2$. The inequality $C_0 \geq \sum_{i=1}^n (m_i+1) = 2$ follows directly from the definition of $C_0$.
\end{proof}

Our choice of $\delta_i^\pm$ implies that
\begin{equation*}
\begin{aligned}
& \mu_{\CZ}^{\tau_\Sigma,\delta_i^+}(\gamma_i)=1 \ \text{where $\gamma_i$ is seen as a closed Reeb orbit of $\alpha_+$}, \\
& \mu_{\CZ}^{\tau_\Sigma,\delta_i^-}(\gamma_i)=1 \ \text{where $\gamma_i$ is seen as a closed Reeb orbit of $\alpha_-$},
\end{aligned}
\end{equation*}
and
\begin{equation*}
\begin{aligned}
& \mu_{\CZ}^{\tau_{\rm gl},\delta_i^+}(\gamma_i)=2m_i+1 \ \text{where $\gamma_i$ is seen as a closed Reeb orbit of $\alpha_+$}, \\
& \mu_{\CZ}^{\tau_{\rm gl},\delta_i^-}(\gamma_i)=2m_i+1 \ \text{where $\gamma_i$ is seen as a closed Reeb orbit of $\alpha_-$},
\end{aligned}
\end{equation*}
for all $i$. We get
\begin{equation}
\sum_{i=1}^n \left(\mu_{\CZ}^{\tau_{\rm gl},\delta_i^\pm}(\gamma_i)+1\right) = \sum_{i=1}^n \left(2m_i+2\right) = 4,
\end{equation}
in view of Lemma~\ref{lemsl}. Again by Lemma~\ref{lemsl} and the assumption $n\geq 2$, we have
\begin{equation}
\label{desigcC}
c_0+C_0 = 2, \qquad c_0 \leq 0, \qquad 0\leq \ell_0 \leq 1 \qquad \mbox{ and } \qquad C_0 \geq 2.
\end{equation}
Denote by
\begin{equation}\label{decC}
c = 2c_0 -2\leq -2 \qquad {and} \qquad C = 2(C_0-\ell_0)+1 \geq 3
\end{equation}
the end points of the interval $I_0(L,\Sigma,\tau_{\rm gl})$.

Choose $h$ and $\Omega$ as in~\ref{sssec_gen_fin_energy_curves}. Then we can consider the energy~\eqref{energy_non_invariant} of a $\jbar$-holomorphic map whenever $\jbar\in\J_{\Omega,L}(\jtil_-,\jtil_+)$. Let $\util=(a,u):S^2\setminus \Gamma \to \R\times M$ be a finite-energy $\jbar$-holomorphic map with $E(\util)\leq A$, where $\Gamma$ is the set of non-removable punctures. In the following for any such map we split
\begin{equation}\label{splittingGamma}
\Gamma = \Gamma^b \cup \Gamma^+ \cup \Gamma^-
\end{equation}
where $\Gamma^b$ is the set of positive punctures of $\util$ whose asymptotic limits are simple and coincide with one of the components of $L$. The set $\Gamma^+$ consists of the remaining positive punctures of $\util$ and $\Gamma^-$ is the set of negative punctures of $\util$.

\begin{theorem}[Dragnev~\cite{drag}, Hofer-Wysocki-Zehnder~\cite{props3}, Wendl~\cite{wendl_transv}]\label{thm_transversality}
There is a dense subset $\mathcal{J}_{\rm reg} \subset \mathcal{J}_{\Omega,L}(\widetilde J_-, \widetilde J_+)$ such that if $\util=(a,u):S^2 \setminus \Gamma \to \R \times M$ is a somewhere injective $\bar J$-holomorphic curve for some $\bar{J}\in\mathcal{J}_{\rm reg}$, satisfying
\begin{itemize}
\item[(i)] $E(\util) \leq A$,
\item[(ii)] $\Gamma^b \neq \emptyset$, $\Gamma^+=\emptyset$, $\Gamma^- \neq \emptyset$,
\item[(iii)] $u(S^2 \setminus \Gamma)\subset M\setminus L$, and
\item[(iv)] $\wind_\infty(\util,z,\tau_\Sigma)=0$ for all $z\in \Gamma^b$
\end{itemize}
then
\begin{equation}\label{desig_fred}
0\leq -2 +\#\Gamma + \sum_{z\in \Gamma^b}(2m_z+1) - \sum_{z\in\Gamma^-} \mu_{\rm CZ}(P_{z}).
\end{equation}
Here $m_z:=m_i$ whenever $\util$ is asymptotic to $\gamma_i\subset L$ at $z\in \Gamma^b$, and the $\alpha_-$-closed Reeb orbit $P_z$ is the asymptotic limit of $\util$ at $z\in  \Gamma^-$.
\end{theorem}

From now on we fix $\jbar \in \J_{\rm reg}$ and an ordered set $\Gamma_0\subset S^2$, $\#\Gamma_0=n$. We write $\M_{\jtil_+,0,\delta^+}$, $\M_{\jtil_-,0,\delta^-}$ and $\M_{\jbar,0,\delta^+}$ instead of $\M_{\jtil_+,0,\delta^+}(\gamma_1,\dots,\gamma_n;\emptyset)$ etc for simplicity of notation. For the same reason we write $I_0$ instead of $I_0(L,\Sigma,\tau_{\rm gl})$. Let $C_+ \in \M_{\jtil_+,0,\delta^+}$ be the curve whose existence is assumed in the statement of Proposition~\ref{prop_main_compactness_global}. After translating up in the $\R$-direction, we can also view $C_+$ as an element of $\M_{\jbar,0,\delta^+}$, see Remark~\ref{rmk_special_case_non-cylindrical=cylindrical}. Since (H$4'_+$) is the same as (H$4_+$), we can invoke Lemma~\ref{lemma_implied_existence_step1} to obtain a sequence
\begin{equation}
C_l = [\vtil_l=(b_l,v_l),S^2,j_l,\Gamma_0,\emptyset]
\end{equation}
in the same connected component of $C_+$ satisfying
\begin{equation}
\lim_{l\to\infty} \min b_l = -\infty.
\end{equation}
By SFT compactness, see~\ref{ssec_SFT_comp}, up to choice of a subsequence we may assume that $C_l$ converges to a holomorphic building
\begin{equation*}
{\bf u}= \{ \{\util_m=(a_m,u_m),S_m,j_m,\Gamma^+_m,\Gamma^-_m,D_m)\},\{\Phi_m\} \} \qquad  m\in\{-k_-,\dots,k_+\}
\end{equation*}
of height $k_-|1|k_+$, with $k_->0$. We shall prove that $k_+=0$, $k_-=1$, $\util_0$ coincides with $\R\times L$, and $\util_{-1}$ represents the desired $C_-$. This is accomplished by the sequence of claims listed below. \\

\noindent {\it Claim I. Every asymptotic limit $\gamma$ of ${\bf u}$ contained in $M\setminus L$ satisfies $\mu_{\CZ}^{\tau_{\rm gl}}(\gamma) \in \Z \setminus I_0$ and ${\rm int}(\gamma,\Sigma)=0$. In particular, every asymptotic limit which is a closed Reeb orbit of $\alpha_+$ is contained in $L$.}

\medskip

\noindent {\it Proof of Claim I.} We argue indirectly and assume that there exists a level of ${\bf u}$ containing an asymptotic limit $\gamma=(x,T)$ in $M \setminus L$ such that $\mu_{\rm CZ}^{\tau_{\rm gl}}(\gamma) \in I_0$ or ${\rm int}(\gamma,\Sigma)\neq 0$. The action of every asymptotic limit of ${\bf u}$ is bounded from above by~$A$. This follows from the energy bounds for the sequence $\widetilde v_l$. Using hypothesis (H$4'$) we conclude that ${\rm int}(\gamma,\Sigma) \neq 0$. For every large $l$ we find an embedded loop $\beta_l:\R / \Z \to  S^2 \setminus \Gamma_0$ such that $v_l \circ  \beta_l$ is $C^0$-close to the loop $t \mapsto x(Tt)$.  But Proposition~\ref{prop_cobordisms_intersection_crucial} and positivity of intersections imply that $C_l$ can be homotoped to $C_+=[\util_+=(a_+,u_+),S^2,j,\Gamma,\emptyset]$ in $M\setminus L$. By Proposition~\ref{prop_Seifert_no_intersections}, $u_+$ is a proper embedding of $S^2 \setminus \Gamma$ into $M\setminus L$ inducing the same class in $H_2(M,L)$ as $\Sigma$ we obtain
$$
0={\rm int}(u_+\circ \beta_l,u_+(S^2\setminus\Gamma))={\rm int}(u_+\circ \beta_l,\Sigma)={\rm int}(v_l\circ \beta_l,\Sigma) = {\rm int}(\gamma,\Sigma)\neq0,
$$ which is a contradiction. // 

\medskip

Define $I^-_0:= (-\infty,c-1] \cap \Z$ and $I^+_0:=[C+1,+\infty) \cap \Z$,
so that we have $I_0^- \cup I_0^+ = \Z\setminus I_0$. 

\begin{remark}
In the statement of Claim I the number $\mu_{\CZ}^{\tau_{\rm gl}}(\gamma)$ should be understood as the Conley-Zehnder index taken with respect to the Reeb flow of $\alpha_+$ or of $\alpha_-$, depending on the level and on the sign of the puncture.
\end{remark}

\noindent {\it Claim II. If $-k_- \leq m \leq k_+$ and $Y$ is a connected component of $S_m\setminus \Gamma_m^+ \cup \Gamma_m^-$ such that $\wtil := \widetilde u_m|_{Y}$ is non-constant and its image is not contained in $\R \times L$, then  $\widetilde w$ has no negative puncture whose asymptotic limit is contained in $L$, and the image of~$\wtil$ does not intersect $\R \times L$.}

\medskip

\noindent {\it Proof of Claim II.} We argue indirectly and assume that $\widetilde w = (b,w)$ admits a negative puncture $\bar z$ whose  asymptotic limit is $\gamma_i^k$ for some $i\in \{1,\ldots,n\}$ and some integer $k\geq 1$. By hypotheses, $\wtil$ has a non-trivial asymptotic formula near $\bar z$.  Since $\mu_{\rm CZ}^{\tau_\Sigma}(\gamma_i^k) \geq \mu_{\rm CZ}^{\tau_\Sigma}(\gamma_i)\geq 1$ we obtain that $\wind_\infty (\widetilde w,\bar z,\tau_\Sigma) \geq 1$, see Theorem \ref{thm_precise_asymptotics} for a more precise description of the behavior of $\widetilde w$ near $\bar z$. This implies that we can find an embedded loop $\beta:\R / \Z \to Y$ around $\bar z$ so that ${\rm int}(w \circ \beta, \Sigma) \neq 0$. Hence for all large $l$ we find an embedded loop $\beta_l:\R / \Z \to S^2 \setminus \Gamma$ so that ${\rm int} (v_l \circ  \beta_l,\Sigma) \neq 0$. This leads to a contradiction as in the proof of Claim I.

Any intersection of $\widetilde w$ with $\R \times L$ must be isolated since $\R \times L$ is also holomorphic. By positivity and stability of intersections, such an intersection implies that $C_l$ intersects $\R \times L$ for all large $l$. This is impossible since Proposition~\ref{prop_cobordisms_intersection_crucial} and positivity of intersections together imply that $C_l$ does not intersect $\R\times L$. //

\medskip

\noindent {\it Claim III. $k_+=0$.}

\medskip

\noindent {\it Proof of Claim III.} Assume, by contradiction, that $k_+>0$. Suppose first that there exists a connected component $Y$ of $S_{k_+}\setminus\Gamma^+_{k_+}\cup\Gamma^-_{k_+}$ such that $\util_{k_+}|_Y$ is non-constant and $\util_{k_+}(Y)\not\subset\R\times L$. The asymptotic limit of $\util_{k_+}$ at a positive puncture~$z_*$ of $\util_{k_+}|_Y$ is a prime closed Reeb orbit given by one of the components of $L$. Moreover, $\util_{k_+}|_Y$ has a non-trivial asymptotic formula at~$z_*$. If $\wind_\infty(\util_{k_+},z_*,\tau_\Sigma)\neq0$ then, by Theorem~\ref{thm_precise_asymptotics}, a small embedded loop $\beta:\R / \Z \to Y\setminus\Gamma^+_{k_+}\cup\Gamma^-_{k_+}$ winding once around $z_*$ is mapped by $u_{k_+}$  to a loop satisfying ${\rm int}(u_{k_+} \circ \beta,\Sigma) \neq 0$. We find for large $l$ an embedded loop $\beta_l:\R / \Z \to S^2\setminus\Gamma_0$ such that ${\rm int}(v_l \circ \beta_l,\Sigma)\neq0$. As in the proof of Claim I, this fact and Proposition~\ref{prop_cobordisms_intersection_crucial} lead to ${\rm int}(u_+ \circ \beta_l,u_+(S^2\setminus\Gamma))\neq0$, a contradiction to Proposition~\ref{prop_Seifert_no_intersections}. Thus
\begin{equation}\label{wind_infty_claim_III}
\wind_\infty(\util_{k_+},z_*,\tau_\Sigma)=0
\end{equation}
for every positive puncture $z_*$ of $\util_{k_+}|_Y$. By hypothesis (H$4'$) and Claims I and II, $\util_{k_+}|_Y$ has no negative punctures and $u_{k_+}(Y)\subset M\setminus L$. Hypothesis (H5) and~\eqref{wind_infty_claim_III} together imply that the set of punctures of $\util_{k_+}|_Y$ consists of $\Gamma^+_{k_+}$. It follows that $\util_{k_+}$ is constant on every other connected component of $S_{k_+}\setminus \Gamma_{k_+}^+ \cup \Gamma_{k_+}^-$. In particular, $\util_{k_+}$ has no negative punctures, in contradiction to $k_+>0$. We showed that for every connected component $Y\subset S_{k_+}\setminus \Gamma_{k_+}^+ \cup \Gamma_{k_+}^-$ the map $\util_{k_+}|_Y$ is either constant or its image is contained in $\R\times L$. In the latter case, $\util_{k_+}|_Y$ is an unbranched cover of one of the cylinders $\R\times \gamma_i$ because asymptotic limits at distinct positive punctures of $\util_{k_+}$ are geometrically distinct. By the same token, if there is a component $Y'\subset S_{k_+}\setminus \Gamma^+_{k_+} \cup \Gamma^-_{k_-}$ so that $\util_{k_+}|_{Y'}$ is constant then Lemma~\ref{lemma_constant_components} provides an intersection between two geometrically distinct closed Reeb orbits, and this is absurd. Hence all components of $\util_{k_+}$ are unbranched covers of one of the cylinders in $\R\times L$. In particular, $\util_{k_+}$ consists of only trivial cylinders, with  no nodes or constant components, a contradiction to stability. This shows that $k_+=0$. //

\medskip

\begin{definition}
\label{def_components_above_directly_above}
Fix a level $m$ strictly below the top level and let $\Lambda \subset \Gamma^+_m$. For each $z\in\Lambda$ we consider the connected component $B_z$ of $S_m$ containing $z$.
\begin{itemize}
\item A connected component $B'$ of $S_k$ is said to be {\it directly above $\Lambda$} if $k>m$ and if there exists $z\in\Lambda$ such that points in $B'\setminus \Gamma^+_k\cup\Gamma^-_k\cup D_k$ can be connected to points in $B_z\setminus \Gamma^+_m\cup\Gamma^-_m\cup D_m$ through a path in $S^{{\bf u},r}$ that only goes down negative punctures or passes through nodes, but never goes up a positive puncture.
\item A connected component $B''$ of $S_k$ is said to be {\it above $\Lambda$} if $k>m$ and if points in $B''\setminus \Gamma^+_k\cup\Gamma^-_k\cup D_k$ can be connected to points in $B'\setminus \Gamma^+_j\cup\Gamma^-_j\cup D_j$, for some $B'$ directly above $\Lambda$, through a path in $S^{{\bf u},r}$ that never reaches levels $\leq m$.
\end{itemize}
\end{definition}

\begin{remark}
\label{rmk_reaching_the_top}
In our particular geometric set-up, once a positive puncture $z$ on a level below the top is fixed one can always reach a curve on the top level by only following components directly above~$z$.
\end{remark}

\begin{remark}
The curves given by the components above some $\Lambda$ form a building, called the {\it (sub-)building above $\Lambda$}.
\end{remark}

\noindent {\it Claim IV. The following assertions hold:
\begin{itemize}
\item Let $m\leq -1$, and let $Y$ be a connected component of $S_m\setminus \Gamma_m^+ \cup \Gamma_m^-$ such that $\widetilde u_m|_{Y}$ is non-constant. If the asymptotic limit at some negative puncture of $\widetilde u_m|_Y$ is contained in $L$, then there exists $i$ such that $\util_m|_Y$ is the trivial cylinder over $\gamma_i$, and the entire sub-building above the (unique) positive puncture of $\util_m|_Y$ has no nodes and consists of trivial cylinders over $\gamma_i$.
\item If $\widetilde u_{m}$ has a positive puncture whose asymptotic limit is contained in $L$ then this asymptotic limit is simply covered. Moreover, no two distinct positive punctures of $\util_{m}$ have the same asymptotic limit in $L$.
\end{itemize}}

\medskip

\noindent {\it Proof of Claim IV.}
Let us prove the first claim. We write $\wtil = \util_m|_Y$, and assume that there is a negative puncture of~$\widetilde w$ whose asymptotic limit is contained in $\gamma_i \subset L$, for some $i \in \{1,\dots,n\}$. It follows directly from Claim~II that the image of  $\widetilde w$ is equal to $\R \times \gamma_i$. Denote by $\Lambda$ the set of positive punctures of~$\wtil$. Iterated applications of Claim II imply that the (non-constant) finite-energy maps defined on components directly above $\Lambda$ (see Definition~\ref{def_components_above_directly_above}) have images equal to $\R\times\gamma_i$. In particular they are covers of $\R\times\gamma_i$, possibly branched. Some curve directly above $\Lambda$ lying on the very top level is a trivial cylinder over $\gamma_i$ since all asymptotic orbits at the top punctures of ${\bf u}$ are simply covered and geometrically distinct. Remark~\ref{rmk_reaching_the_top} was used. This forces the sub-building ${\bf u}_{\Lambda}$ above $\Lambda$ to have a very particular form: it has no nodes, and all its levels consist of precisely one curve which is a trivial cylinder over $\gamma_i$. Here the reader should not forget that ${\bf u}$ is a limit of genus zero curves, and that $\R\times L$ is a $\jbar$-holomorphic surface. Lemma~\ref{lemma_constant_components} is used. In particular, $\#\Lambda=1$ and also $\wtil$ is a trivial cylinder over $\gamma_i$.

Now we address the second claim. If $m=0$ then this claim is automatically satisfied, hence we can assume that $m\leq-1$. Consider a positive puncture $z^*$ of $\util_{m}$ whose asymptotic limit is $\gamma_i^{k}$ for some $i\in \{1,\ldots,n\}$ and some $k \geq 1$. By Claim II and the first assertion of Claim IV, $k=1$ and the sub-building above $z^*$ is formed by trivial cylinders over $\gamma_i$.  If $\util_{m}$ is also asymptotic to $\gamma_i$ at a positive puncture $z^{**} \neq z^*$ then, as before, the sub-building above $z^{**}$ is also formed by unbranched covers of $\R \times \gamma_i$. Since $\util_0$ has only one positive puncture asymptotic to $\gamma_i$ we arrive at a contradiction to $z^*\neq z^{**}$. Hence such a puncture $z^{**}$ cannot exist. This finishes the proof of Claim IV. //

\medskip

\noindent {\it Claim V. Let $Y$ be a connected component of $S_m\setminus \Gamma_m^+ \cup \Gamma_m^-$  such that $\util:=\util_m|_Y$ is non-constant. Assume that $\util$ has a positive puncture whose asymptotic limit is contained in $M\setminus L$ and satisfies $\mu_{\CZ}^{\tau_{\rm gl}} \in I_0^-$. Then $m\leq -1$  and at least one of the following alternatives holds:
\begin{itemize}
\item $\util$ contains a negative puncture whose asymptotic limit lies in $M\setminus L$ and has Conley-Zehnder index $\mu_{\CZ}^{\tau_{\rm gl}} \in I_0^-$;
\item $\util$ contains a positive puncture whose asymptotic limit lies in $M\setminus L$ and has Conley-Zehnder index $\mu_{\CZ}^{\tau_{\rm gl}} \in I_0^+$.
\end{itemize}
}

\medskip

\noindent {\it Proof of Claim V.} We can identify $Y = S^2\setminus\Gamma$ where $\Gamma$ is the set of non-removable punctures of $\util$. By Claim III we know that $m\leq 0$.  If $m=0$ then $\util$
is $\bar J$-holomorphic and asymptotic limits at its positive punctures are components of $L$, in conflict with the assumption on $\util$. We conclude that $m\leq -1$ and $\util$ is $\jtil_-$-holomorphic.

Consider the splitting $\Gamma = \Gamma^b \cup \Gamma^+ \cup \Gamma^-$ as in \eqref{splittingGamma}.
By the first assertion of Claim IV and our hypotheses, every asymptotic limit of $\util$ at a negative puncture is contained in $M\setminus L$.
By the second assertion in Claim IV, all asymptotic limits in $\Gamma^+$ are contained in~$M\setminus L$, and all asymptotic limits in $\Gamma^b$ are geometrically distinct.
In particular, by Claim I, every asymptotic limit of $\util$ at a puncture in $\Gamma^+ \cup \Gamma^- = \Gamma \setminus \Gamma^b$ satisfies $\mu_{\CZ}^{\tau_{\rm gl}} \in I_0^- \cup I_0^+$ and its intersection number with $\Sigma$ vanishes.
By assumption, $\Gamma^+\neq\emptyset$.

Assume, by contradiction, that no asymptotic limit at a negative puncture of $\util$ satisfies $\mu_{\CZ}^{\tau_{\rm gl}} \in I_0^-$ and no asymptotic limit at a puncture of $\util$ in $\Gamma^+$ satisfies $\mu_{\CZ}^{\tau_{\rm gl}} \in I_0^+$. It follows that the Conley-Zehnder indices of the asymptotic limits at negative punctures lie in $I_0^+$ and the Conley-Zehnder indices of the asymptotic limits at positive punctures in $\Gamma^+$ lie in $I_0^-$. We claim that these assumptions force
\begin{equation}\label{positive_area_claim_V}
\int_{S^2 \setminus \Gamma} u^* d\alpha_- >0.
\end{equation}
If not then $\util$ is a (possibly branched) cover of $\R\times\beta$ for some simply covered periodic Reeb orbit $\beta$ in $M\setminus L$. We claim that ${\rm int}(\beta,\Sigma)=0$. In fact, at any puncture the asymptotic limit is $\beta^k$ for some $k\geq 1$. Hence $$ 0={\rm int}(\beta^k,\Sigma) = k\ {\rm int}(\beta,\Sigma) \Rightarrow 0 = {\rm int}(\beta,\Sigma) $$ as claimed. By (H$4'$) we must have $\mu_{\CZ}^{\tau_{\rm gl}}(\beta) \in I_0^- \cup I_0^+$. We can estimate
\begin{equation}\label{estimates_CZ_claim_V_useful}
\begin{aligned}
& \mu_{\CZ}^{\tau_{\rm gl}}(\beta) \in I_0^- \Rightarrow \mu_{\CZ}^{\tau_{\rm gl}}(\beta) <0 \Rightarrow \mu_{\CZ}^{\tau_{\rm gl}}(\beta^k) < 0 \ \forall k\geq1 \Rightarrow \mu_{\CZ}^{\tau_{\rm gl}}(\beta^k) \in I_0^- \ \forall k\geq 1, \\
& \mu_{\CZ}^{\tau_{\rm gl}}(\beta) \in I_0^+ \Rightarrow \mu_{\CZ}^{\tau_{\rm gl}}(\beta) > 0 \Rightarrow \mu_{\CZ}^{\tau_{\rm gl}}(\beta^k) > 0 \ \forall k\geq1 \Rightarrow \mu_{\CZ}^{\tau_{\rm gl}}(\beta^k) \in I_0^+ \ \forall k\geq 1.
\end{aligned}
\end{equation}
In particular, all asymptotic limits satisfy either $\mu_{\CZ}^{\tau_{\rm gl}} \in I_0^-$ or $\mu_{\CZ}^{\tau_{\rm gl}} \in I_0^+$, in conflict to the contradiction assumption made on $\util$. This finishes the proof of~\eqref{positive_area_claim_V}.

As consequence of~\eqref{positive_area_claim_V} $\util$ has a non-trivial asymptotic formula at every puncture since $Y=S^2\setminus\Gamma$ is connected. We can then argue that $\wind_\infty(\util,z,\tau_\Sigma)=0$ for every $z\in \Gamma^b$ since, otherwise, for $l$ large we find loops in the image of the $M$-component $v_l$ which intersect $\Sigma$ non-trivially, contradicting a combination of Proposition~\ref{prop_Seifert_no_intersections} and Proposition~\ref{prop_cobordisms_intersection_crucial} as in the proof of Claim~II. We have
$$
\begin{aligned}
\sum_{z\in \Gamma^b} \left(\wind_\infty(\util,z, \tau_{\rm gl})+1\right) &= \sum_{z\in \Gamma^b} (\wind_\infty(\util,z,\tau_\Sigma)+m_z +1) \\
&= \sum_{z\in \Gamma^b} (m_z +1) \leq C_0-\ell_0,
\end{aligned}
$$
where $m_z=m_i$ if $\util$ is asymptotic to $\gamma_i$ at $z\in \Gamma^b$. The last inequality above is explained as follows. 
Since ${\bf u}$ has arithmetic genus zero, $\util$ has at most $n$ positive punctures; this relies on the fact that the domain of $\util$ is connected and that $S^{{\bf u},r}$ is also connected. 
Moreover, the assumptions on $\util$ imply that $\#\Gamma^b \leq n-1$, i.e. $\Gamma^b$ misses at least one binding orbit $\gamma_i$. By Corollary~\ref{coro_ell_0} we have $0 \leq \ell_0\leq 1$, hence by the definition of $C_0$ in \eqref{important_numbers} the sum $\sum_{z\in \Gamma^b} (m_z+1)$ is bounded from above by $C_0-\ell_0$.

Let us denote by $P_z$ the asymptotic limit of $\util$ at the puncture $z\in \Gamma$. Since $\mu_{\rm CZ}^{\tau_{\rm gl}}(P_z) \leq c-1 = 2(c_0-2)+1$ for all $z\in \Gamma^+$, we have
\[
\wind_\infty(\util,z,\tau_{\rm gl}) \leq c_0-2, \ \forall z\in \Gamma^+,
\]
hence
$$
\sum_{z\in \Gamma^+} \wind_\infty(\util,z, \tau_{\rm gl})  \leq (c_0-2)\#\Gamma^+.
$$
Since $\mu_{\rm CZ}^{\tau_{\rm gl}}(P_z) \geq C+1 = 2(C_0-\ell_0+1)$ for all $z\in \Gamma^-$, we have
\[
\wind_\infty(\util,z,\tau_{\rm gl}) \geq C_0-\ell_0 +1, \ \forall z\in \Gamma^-
\]
hence
$$
-\sum_{z\in \Gamma^-}  \wind_\infty(\util,z, \tau_{\rm gl})   \leq -(C_0-\ell_0+1)\#\Gamma^-.
$$
Plugging these inequalities into
\begin{equation}\label{windpiV}
0\leq \wind_\pi(\util) = -2 + \#\Gamma + \sum_{z\in \Gamma^b \cup \Gamma^+}\wind_\infty(\util,z,\tau_{\rm gl}) -  \sum_{z\in \Gamma^-}\wind_\infty(\util,z,\tau_{\rm gl})
\end{equation}
we obtain, in view of Lemma~\ref{lemsl},
$$
\begin{aligned}
0 &\leq C_0-\ell_0 + (c_0-1)\#\Gamma^+ -(C_0-\ell_0)\#\Gamma^- -2 \\
&= (C_0 - \ell_0)(1-\#\Gamma^-) + (c_0-1)\#\Gamma^+ - 2 \\
&\leq (C_0 - \ell_0) + (c_0 - 1) - 2 \\
&\leq C_0+c_0-3 = -1
\end{aligned}
$$
where in the third line we used that $c_0-1<0$, see~\eqref{desigcC}. This contradiction concludes the proof of Claim~V. //

\medskip

\noindent {\it Claim VI. Let $Y$ be a connected component of $S_m\setminus \Gamma_m^+ \cup \Gamma_m^-$  such that $\util:=\util_m|_Y$ is non-constant. Assume that $\util$ has a negative puncture whose asymptotic limit is contained in $M\setminus L$ and has Conley-Zehnder index $\mu_{\CZ}^{\tau_{\rm gl}} \in I^+_0$. Then at least one of the following alternatives holds:
\begin{itemize}
\item $\util$ contains a negative puncture whose asymptotic limit lies in $M\setminus L$ and has Conley-Zehnder index $\mu_{\CZ}^{\tau_{\rm gl}} \in I^-_0$;
\item $\util$ contains a positive puncture whose asymptotic limit lies in $M\setminus L$ and has Conley-Zehnder index $\mu_{\CZ}^{\tau_{\rm gl}} \in I^+_0$.
\end{itemize}
Moreover, if $m=0$ then the first alternative holds.
}

\medskip

\noindent {\it Proof of Claim VI.} We can identify $Y = S^2\setminus \Gamma$, where $\Gamma$ consists of non-removable punctures of $\util$. Write the components as $\util=(a,u)$, and consider the splitting $\Gamma = \Gamma^b \cup \Gamma^+ \cup \Gamma^-$ as in~\eqref{splittingGamma}. Claim III implies $m\leq 0$. The assumptions on $\util$ imply that its image is not contained in $\R\times L$. By Claim~II all asymptotic limits at negative punctures are contained in $M\setminus L$. All asymptotic limits at punctures in $\Gamma^+$ are contained in $M\setminus L$, in fact if $m=0$ this is true because $\Gamma^+$ is empty in this case, and if $m<0$ then this follows from the second assertion of Claim~IV.

Our argument is indirect and we proceed assuming that no asymptotic limit at a negative puncture of $\util$ has Conley-Zehnder index in $I^-_0$, and no asymptotic limit at a positive puncture in $\Gamma^+$ has Conley-Zehnder index in~$I^+_0$. Hence, by Claim I, asymptotic limits at negative punctures satisfy $\mu_{\CZ}^{\tau_{\rm gl}} \in I^+_0$ and have zero intersection number with $\Sigma$, and asymptotic limits at punctures in $\Gamma^+$ satisfy $\mu_{\CZ}^{\tau_{\rm gl}}\in I^-_0$ and have zero intersection number with $\Sigma$. We consider two cases: $m=0$ and $m<0$.

Assume $m=0$, in which case $\util$ is $\jbar$-holomorphic. Every asymptotic limit at a positive puncture of $\util$ is a simply covered orbit in $L$, and such asymptotic limits are mutually distinct. In particular, $\Gamma^+ = \emptyset$ and $\util$ is somewhere injective. By assumption $\Gamma^-\neq\emptyset$. Moreover, $\util$ has a non-trivial asymptotic formula as in Theorem~\ref{thm_precise_asymptotics} at its punctures, and $\wind_\infty(\util,z,\tau_\Sigma)=0$ for all $z\in\Gamma^b$. This last claim is proved by the same argument in the proof of Claim~II. Denote by $P_z$ the asymptotic limit at $z\in \Gamma$. From $\wind_\infty(\util,z,\tau_\Sigma)=0 \ \forall z\in\Gamma^b$  we get $\wind_\infty(\util,z,\tau_{\rm gl})=m_z \ \forall z\in\Gamma^b$ where $m_z=m_i$ if $\util$ is asymptotic to $\gamma_i$ at $z\in \Gamma^b$. We get
$$
\sum_{z\in \Gamma^b} (2m_z +1) =2\sum_{z\in \Gamma^b} (m_z +1) -\#\Gamma^b \leq 2C_0-\#\Gamma^b,
$$
and
$$
-\sum_{z\in \Gamma^-}  \mu_{\rm CZ}(P_z)   \leq -2(C_0-\ell_0+1)\#\Gamma^-.
$$
Using Theorem \ref{thm_transversality} we can estimate
\begin{equation}\label{fred_ind_ineq_claim_VI}
\begin{aligned}
0 &\leq 2C_0-\#\Gamma^b  -2(C_0-\ell_0+1)\#\Gamma^- -2+\#\Gamma^b +\#\Gamma^- \\
& = (2\ell_0-1-2C_0)\#\Gamma^- + 2C_0 - 2. 
\end{aligned}
\end{equation}
By Corollary~\ref{coro_ell_0} we know that $2\ell_0-1-2C_0 \leq -3$. Hence the expression on the right-hand side of~\eqref{fred_ind_ineq_claim_VI} strictly decreases when $\#\Gamma^-$ increases. When $\#\Gamma^-=1$ this expression is strictly negative, no matter the value of $\ell_0 \in \{0,1\}$. This is in contradiction to~\eqref{fred_ind_ineq_claim_VI}.

Now we handle the case $m\leq -1$. In this case $\util$ is $\widetilde J_-$-holomorphic. We split the argument into two subcases: $\Gamma^b\neq\emptyset$ and $\Gamma^b=\emptyset$.

Assume first that $\Gamma^b\neq\emptyset$. At all punctures $\util$ has a non-trivial asymptotic formula as in Theorem~\ref{thm_precise_asymptotics}, $u^*d\alpha_-$ does not vanish identically, and $\wind_\infty(\util,z,\tau_\Sigma)=0$ for all $z\in\Gamma^b$ just as before. Denote by $P_z$ the asymptotic limit of $\util$ at $z\in \Gamma$. We have
$$
\sum_{z\in \Gamma^b} \left(\wind_\infty(\util,z, \tau_{\rm gl})+1\right) = \sum_{z\in \Gamma^b} (m_z +1) \leq C_0,
$$
where $m_z=m_i$ if $\util$ is asymptotic to $\gamma_i$ at $z\in \Gamma^b$. Furthermore, since $$ \mu_{\rm CZ}^{\tau_{\rm gl}}(P_z) \leq c-1 = 2(c_0-2)+1, \ \forall z\in \Gamma^+, $$ we have $\wind_\infty(\util,z,\tau_{\rm gl}) \leq c_0-2, \forall z\in \Gamma^+$. Hence
$$
\sum_{z\in \Gamma^+} \wind_\infty(\util,z, \tau_{\rm gl})  \leq (c_0-2)\#\Gamma^+.
$$
Now since $\mu_{\rm CZ}(P_z) \geq C+1 = 2(C_0-\ell_0+1), \forall z\in \Gamma^-$, we must necessarily have $\wind_\infty(\util,z,\tau_{\rm gl}) \geq C_0-\ell_0 +1, \forall z\in \Gamma^-$. Hence
$$
-\sum_{z\in \Gamma^-}  \wind_\infty(\util,z, \tau_{\rm gl})   \leq -(C_0-\ell_0+1)\#\Gamma^-.
$$
Combining these inequalities with Remark~\ref{rmk_wind_pi_wind_infty}, with $c_0\leq0$ and with $\Gamma^-\geq1$ we obtain
$$
\begin{aligned}
0 &\leq \wind_\pi(\util) = -2 + \#\Gamma + \wind_\infty(\util) \\
& \leq C_0 + (c_0-1)\#\Gamma^+ -(C_0-\ell_0)\#\Gamma^- -2 \\
& \leq \ell_0-2 - (C_0-\ell_0)(\#\Gamma^--1)\\
& \leq -1,
\end{aligned}
$$
which is a contradiction.

Now assume $\Gamma^b=\emptyset$. Then all asymptotic limits of $\util$ are contained in $M\setminus L$. Note that $u^*d\alpha_-$ does not vanish identically. In fact, if it did then $\util$ would be a (possibly branched) cover of $\R\times \beta$ over some prime periodic Reeb orbit $\beta \subset M\setminus L$ satisfying ${\rm int}(\beta^k,\Sigma)=0 \ \forall k\geq0$. Arguing as in Claim~V using (H$4'$), see~\eqref{estimates_CZ_claim_V_useful}, we would conclude that either $\mu_{\CZ}^{\tau_{\rm gl}}(\beta^k) \in I_0^+ \ \forall k\geq1$ or $\mu_{\CZ}^{\tau_{\rm gl}}(\beta^k) \in I_0^- \ \forall k\geq1$, in conflict with the contradiction assumption made on $\util$.

Denote by $P_z$ the asymptotic limit of $\util$ at $z\in \Gamma$. Arguing as in the previous subcase we find $$ \wind_\infty(\util,z,\tau_{\rm gl}) \leq c_0-2, \forall z\in \Gamma^+ \qquad \wind_\infty(\util,z,\tau_{\rm gl}) \geq C_0-\ell_0 +1, \forall z\in \Gamma^-. $$
Hence
$$
\sum_{z\in \Gamma'^+} \wind_\infty(\util,z, \tau_{\rm gl})  \leq (c_0-2)\#\Gamma^+
$$
and
$$
-\sum_{z\in \Gamma'^-}  \wind_\infty(\util,z, \tau_{\rm gl})   \leq -(C_0-\ell_0+1)\#\Gamma^-.
$$
Combining these inequalities with Remark~\ref{rmk_wind_pi_wind_infty} we obtain
$$
0\leq \wind_\pi(\util) \leq (c_0-1)\#\Gamma^+ -(C_0-\ell_0)\#\Gamma^- -2 <0,
$$
which is a contradiction.  The proof of Claim VI is now complete. // 

\medskip

\noindent {\it Claim VII. We have $k_-=-1$ and, moreover, all asymptotic limits of $\util_{-1}$ at its positive punctures are contained in $L$, are simply covered, and mutually distinct.}

\medskip

\noindent {\it Proof of Claim VII.} Assume, by contradiction, that $k_- \leq -2$. We shall prove that this forces $\sqcup_m S_m$ to have an infinite number of connected components.

Let $Y_0$ be a connected component of $S_{k_-} \setminus \Gamma_{k_-} \cup \Gamma_{k_-}$ such that $\wtil_0 := \util_{k_-}|_Y$ is non-constant. We can identify $Y_0 = S^2\setminus \Gamma$ where $\Gamma$ is the set of non-removable punctures. All punctures in $\Gamma$ are positive, and the image $\wtil_0$ is not contained in $\R\times L$. We can pick $Y_0$ in a way that one of the asymptotic limits of $\wtil_0$ lies in $M\setminus L$. Indeed, if not then by Claim~IV every level $\util_{j}$, $k_-+1\leq j \leq -1$ is formed by trivial cylinders in $\R \times L$. 
Together with the assumption $k_-\leq-2$ and the stability condition, this implies that in every level $k_- < m < 0$ strictly between the bottom level and the top level there is a constant component, and possibly many nodes (at least three). Combining with the genus zero assumption, it follows that at the top level there are at least two distinct positive punctures with geometrically equal asymptotic limits, and this is a contradiction to the fact that all asymptotic limits at the positive punctures of the top level are geometrically distinct orbits in~$L$. 
Moreover, $\wtil_0$ does not intersect $\R \times L$ by Claim~II. Claim I implies that asymptotic limits of $\wtil_0$ in $M\setminus L$ must have Conley-Zehnder index $\mu_{\CZ}^{\tau_{\rm gl}}$ in $I_0^- \cup I_0^+$ and its intersection number with $\Sigma$ vanishes.

Claim V provides at least one asymptotic limit at a positive puncture of $\wtil_0$ which lies in $M\setminus L$ and satisfies~$\mu_{\CZ}^{\tau_{\rm gl}} \in I_0^+$. Denote this asymptotic limit by $\gamma$. Now we find a connected component $Y_1$ of $S_{k_-+1}\setminus \Gamma_{k_-+1} \cup \Gamma_{k_-+1}$ such that $\wtil_1 := \util_{k_-+1}|_{Y_1}$ has $\gamma$ as asymptotic limit at a negative puncture. In particular, since $\mu_{\CZ}^{\tau_{\rm gl}}(\gamma)\in I^+_0$, Claim VI applies and at least one of the following holds:
\begin{itemize}
\item[(i)]  $\wtil_1$ has a negative puncture whose asymptotic limit lies in $M\setminus L$ and has Conley-Zehnder index in $I_0^-$;
\item[(ii)] $\wtil_1$ has a positive puncture whose asymptotic limit lies in $M\setminus L$ and has Conley-Zehnder index in $I_0^+$.
\end{itemize}
In case (i) we find a finite-energy map $\wtil_2$ corresponding to a connected component $Y_2$ of $S_{k_-} \setminus \Gamma_{k_-} \cup \Gamma_{k_-}$ which has a positive puncture whose asymptotic limit lies in $M\setminus L$ and its Conley-Zehnder index lies in $I_0^-$. 
In case (ii) we find a finite-energy map $\wtil_2$ corresponding to a connected component $Y_2$ of $S_{k_-+2} \setminus \Gamma_{k_-+2} \cup \Gamma_{k_-+2}$ which has a negative puncture whose asymptotic limit lies in $M\setminus L$ and has Conley-Zehnder index in $I_0^+$.
Back to case (i) we apply Claim V to obtain a positive puncture of $\wtil_2$ whose asymptotic limit lies in $M\setminus L$ and its Conley-Zehnder index lies in $I_0^+$. This gives the next element $\wtil_3$ of the sequence.
Back to case (ii) we apply Claim VI to again obtain two possibly non-excluding possibilities: $\wtil_2$ has a negative puncture whose asymptotic limit lies in $M\setminus L$ and has Conley-Zehnder index in $I_0^-$, or there is a positive puncture of $\wtil_2$ whose asymptotic limit lies in $M\setminus L$ and has Conley-Zehnder index in $I_0^+$. We choose one of the possibilities that apply, in the first case we go down a level and in the second case we go up a level, hence obtaining the next element $\wtil_3$ of the sequence. Our procedure to find the next element of the sequence can be indefinitely continued: from $\wtil_j$ ones changes the level by $+1$ or $-1$ in order to find $\wtil_{j+1}$ for which Claim~V or Claim~VI applies. This allows us to construct a sequence $\wtil_i$. One always goes up a level through a positive puncture whose asymptotic limit is in $M\setminus L$ and satisfies $\mu_{\CZ}^{\tau_{\rm gl}} \in I_0^+$, or goes down a level through a negative puncture whose asymptotic limit is in $M\setminus L$ and satisfies $\mu_{\CZ}^{\tau_{\rm gl}} \in I_0^-$. Since the arithmetic genus of the building {\bf u} is zero, whenever $i\neq j$ the domains of $\wtil_i$ and $\wtil_j$ are distinct connected components of $\sqcup_m S_m\setminus \Gamma_m^+\cup\Gamma_m^-$. Hence $\sqcup_m S_m$ has infinitely many connected components. This is absurd. We have proved that $k_-=-1$.

Now assume, by contradiction, that a $\widetilde J_-$-holomorphic finite-energy map corresponding to the restriction of $\util_{-1}$ to a connected component of $S_{-1}\setminus \Gamma^+_{-1} \cup \Gamma^-_{-1}$  has a positive puncture whose asymptotic limit lies in $M\setminus L$. By Claim~I and Claim~V we can assume that such an asymptotic limit satisfies $\mu_{CZ}^{\tau_{\rm gl}} \in I_0^+$. Using the same procedure as before, one finds that $S_0 \sqcup S_{-1}$ has infinitely many connected components, absurd. Now it is easy to argue, using Claim~II, that asymptotic limits at $\Gamma_{-1}^+$ are geometrically distinct simply covered components of $L$. //

\medskip

\noindent {\it Claim VIII. The top level is precisely $\R\times L$.}

\medskip

\noindent {\it Proof of Claim VIII.} Assume there exists a connected component $Y$ of $S_0 \setminus \Gamma^+_0 \cup \Gamma^-_0$ such that $\util=\util_0|_Y$ is non-constant and its image is not contained in $\R \times L$. By Claim~II the asymptotic limits at the negative punctures of $\util$ lie in $M\setminus L$, in contradiction to Claim VII. In particular, $\util$ has no negative punctures. Arguing as in the proof of Claim III, relying on hypothesis (H5), we conclude that  all positive punctures of $\util_0$ are punctures of $\util$. Moreover, $\util_0$ is constant on every other connected component of $S_0 \setminus \Gamma^+_0 \cup \Gamma^-_0$. In particular, $\util_0$ has no negative punctures, and this is absurd. We conclude that every connected component of $\util_0$ either corresponds to a constant curve or to a curve whose image is contained in $\R \times \gamma_i$ for some $i\in \{1,\ldots, n\}$. In the former case, Lemma~\ref{lemma_constant_components} provides an intersection between two geometrically distinct closed Reeb orbits in $L$, which is again an absurd. Hence the latter case holds and each component of the level $\util_0$ is a trivial cylinder over some $\gamma_i$. Claim VIII is proved. //

\medskip

\noindent {\it  Claim IX. The level $\util_{-1}$ provides an element of  $\M_{\jtil_-,0,\delta^-}(\gamma_1,\dots,\gamma_n;\emptyset)$.}

\medskip

\noindent {\it Proof of Claim IX.} Let $\util:S^2 \setminus \Gamma \to \R \times M$ be a non-constant $\widetilde J_-$-holomorphic map given by the restriction of $\util_{-1}$ to a connected component of $S_{-1}\setminus \Gamma^+_{-1} \cup \Gamma^-_{-1}$; here $\Gamma$ is the set of its (non-removable) punctures. By Claim VII, $\util$ has no negative punctures and it is asymptotic to distinct simply covered components of $L$ at its positive punctures. Moreover, $\util$ does not intersect $\R \times L$ (Claim~II). Hence $\util$ has a non-trivial asymptotic formula near each of its positive punctures. We claim that $\wind_\infty(\util,z, \tau_\Sigma)=0$ for every $z \in \Gamma$. This can be proved by recasting the argument for Claim~III. In fact, if not then we can use Theorem~\ref{thm_precise_asymptotics} to find, for $l$ large enough, a loop $\beta_l$ in the domain of $\vtil_l$ such that ${\rm int}(v_l\circ\beta_l,\Sigma)\neq0$. By Proposition~\ref{prop_cobordisms_intersection_crucial} we can homotope $\vtil_l$ to $\util_+$ (the map representing $C_+$) in the complement of $\R\times L$. We conclude that the algebraic intersection number of the image under $u_+$ of a loop in the domain of $\util_+$ with the image of $u_+$ is zero, in contradiction to Proposition~\ref{prop_Seifert_no_intersections} stating that $u_+$ is an embedding into $M\setminus L$. 
Hypothesis~(H5) implies that $\Gamma = \Gamma_{-1}^+$, i.e. the asymptotic limits of $\util$ are mutually geometrically distinct and consist of all components of $L$. Hence any other component of $\util_0$ is constant. The existence of such a constant component leads to a contradiction given by an application of Lemma \ref{lemma_constant_components}. Hence $\util_0$ has only one connected component represented by $\util$. The curve $\util$ represents the desired element in $\M_{\jtil_-,0,\delta^-}(\gamma_1,\dots,\gamma_n;\emptyset)$: the correct exponential decay to the $\gamma_i$ at the positive punctures follows from $\wind_\infty(\util,z, \tau_\Sigma)=0 \ \forall z\in\Gamma=\Gamma_{-1}^+$. //

\medskip

This finishes the proof of Proposition~\ref{prop_main_compactness_global}, and hence also of Theorem~\ref{main1}.

\appendix

\section{Intersection and linking}

Let $M$ be an oriented $3$-manifold and $\Sigma\subset M$ be an embedded compact oriented surface. Denote $L=\partial\Sigma$, orient $L$ by $\Sigma$. Let $N$ be a small closed tubular neighborhood of $L$ and choose a diffeomorphism $\Psi: N \to L\times \C$ satisfying
\begin{itemize}
\item[(T1)] $\Psi$ is orientation preserving if $N$ is oriented by $M$ and $L\times \C$ is oriented as a product; here $\C$ is given its canonical orientation.
\item[(T2)] $\Psi(N \cap \Sigma) = L \times [0,+\infty)$.
\end{itemize}
Consider a smooth map $\util = (a,u) : \D \to \R\times N$ satisfying $\util(\partial\D) \cap (\R\times L) = \emptyset$. We orient $\D\subset\C$ by $\C$, the real line $\R$ by its canonical orientation and $\R\times M$, $\R\times L$ as products. The following simple statement is heavily used in the compactness analysis from subsection~\ref{ssec_existence_compactness_curves}.

\begin{lemma}\label{lemma_appendix_intersection}
We have $$ {\rm int'}(\util_*[\D],\R\times L) = {\rm int}((u|_{\partial\D})_*[\partial\D],\Sigma) $$ where $[\D]\in H_2(\D,\partial\D)$, $[\partial\D]\in H_1(\partial\D)$ denote the corresponding fundamental classes, the homomorphism ${\rm int'}(\cdot,\R\times L) : H_2(\R\times M,\R\times (M\setminus L)) \to \Z$ counts oriented intersections with $\R\times L$, and ${\rm int}:H_1(M\setminus L) \otimes H_2(M,L) \to \Z$ is the oriented intersection count pairing.
\end{lemma}

\begin{proof}
Let $z:\D\to \C$ denote the $\C$-component of $\Psi\circ u$. Note that $z(e^{i2\pi t}) \neq 0$ for all $t$. In view of our choices of orientations, ${\rm int'}(\util_*[\D],\R\times L)$ is equal to the algebraic count of zeros of the map $z:\D\to\C$. By standard degree theory, this is equal to the winding number around the origin of the loop $t\in\R/\Z \mapsto z(e^{i2\pi t})$. This winding number coincides with ${\rm int}((u|_{\partial\D})_*[\partial\D],\Sigma)$.
\end{proof}

\section{Homology of the complement of the binding}\label{app_comp_H_1}

Let $(\Pi,L)$ be an open book decomposition of the closed, connected and oriented $3$-manifold $M$. The standard orientation of $\R/\Z$ and the map $\Pi$ co-orient the pages. This co-orientation and the orientation of $M$ together orient the pages. We orient $L$ as the boundary of a page. Let $\gamma_1,\dots,\gamma_n$ be the oriented components of $L$. Let $\gamma_i'$ be the oriented loop obtained by pushing $\gamma_i$ into $M\setminus L$ in the direction of a page.

\begin{lemma}\label{lemma_topological}
We have $$ H_1(M\setminus L) \simeq \frac{H_1({\rm page})}{{\rm im}({\rm id}-h_*)} \oplus \Z e $$ where $h$ is a monodromy map of $(\Pi,L)$
and $e$ is any loop in $M\setminus L$ such that $\Pi_*e$ is the positive generator of $H_1(\R/\Z)$. In particular, in the planar case we get $H_1(M\setminus L) \simeq A \oplus \Z\left<e\right>$ where $A$ is the abelian group generated by the $\gamma_i'$ constrained by the single relation $\gamma_1'+\dots+\gamma_n'=0$.
\end{lemma}

In the above statement homology groups are taken with $\Z$ coefficients.

\begin{proof}
By Wang's sequence, see~\cite[Lemma~8.4]{milnor}, we have an exact sequence
\[
H_1({\rm page}) \stackrel{h_*-{\rm id}}{\longrightarrow} H_1({\rm page}) \longrightarrow H_1(M\setminus L) \longrightarrow H_0({\rm page}) \stackrel{h_*-{\rm id}}{\longrightarrow} H_0({\rm page}) \, .
\]
It turns out that $h_*-{\rm id}$ vanishes in $H_0({\rm page})$. 
We get an exact sequence
\[
0 \longrightarrow \frac{H_1({\rm page})}{{\rm im}({\rm id}-h_*)}  \longrightarrow H_1(M\setminus L) \longrightarrow H_0({\rm page}) \longrightarrow 0 \, .
\]
The first assertion of the lemma follows.
In the special case where the page is a sphere with holes we have $h_*= {\rm id}$ on $H_1({\rm page})$ because this homology group is generated by boundary components as in the statement and $h$ is supported away from the boundary.
\end{proof}

\end{document}